\documentclass{amsart}
\usepackage{graphicx} 
\usepackage{hyperref}
\usepackage{amsmath, amsthm, amssymb, amscd}
\usepackage[mathscr]{euscript}
\usepackage{dsfont}
\usepackage{mathtools}
\usepackage{mathabx}
\usepackage{enumerate}
\usepackage{hyperref}
\usepackage{verbatim}
\usepackage{subcaption}
\usepackage{cite}
\usepackage{color}
\usepackage{esint}
\usepackage{tikz-cd}
\usetikzlibrary{shapes.geometric}
\usepackage[shortlabels]{enumitem}
 
\usepackage{bm}

\makeatletter
\newcommand{\labitem}[2]{%
	\def\@itemlabel{\textbf{#1}}
	\item
	\def\@currentlabel{#1}\label{#2}}
\makeatother
\title{On the energy image density conjecture of Bouleau and Hirsch}
\author{Sylvester Eriksson-Bique}
\address[Sylvester Eriksson-Bique]{Department of Math. and Stat.
	P.O. Box 35 \\
	FI-40014 University of Jyväskylä, Finland}
\email{sylvester.d.eriksson-bique@jyu.fi}
\author{Mathav Murugan}
\address[Mathav Murugan]{Department of Mathematics, University of British Columbia,
	Vancouver, BC V6T 1Z2, Canada. }
\email{mathav@math.ubc.ca}
\thanks{SEB is supported by the Research Council of Finland via the project \emph{GeoQuantAM: Geometric and Quantitative Analysis on Metric spaces}, , grant no. 354241. MM is supported in part by NSERC and the Canada Research Chairs program.}
\subjclass[2020]{Primary: 31C25, 31E05, 30L99; Secondary: 49Q15, 26B05, 60J60, 60G30, 46E35, 49J52, 53C23, 31C15, 28A12}
\keywords{Energy image density property, currents, absolute continuity, Cheeger's conjecture, $p$-energy measure, decomposability bundle, cone null sets, Lipschitz approximation, martingale dimension}
\date{\today}

\usepackage[shortlabels]{enumitem}
 
\newtheorem{theorem}[equation]{Theorem}
\newtheorem{lemma}[equation]{Lemma}

\newtheorem{proposition}[equation]{Proposition}
\newtheorem{corollary}[equation]{Corollary}

\numberwithin{equation}{section}

\theoremstyle{definition}
\newtheorem{definition}[equation]{Definition}

\theoremstyle{remark}
\newtheorem{remark}[equation]{Remark}
\newtheorem{conjecture}[equation]{Conjecture}
\newtheorem{example}[equation]{Example}
\newcommand*{\N}{\mathbb{N}}

\DeclareMathOperator{\bphi}{\phi}
\DeclareMathOperator{\blambda}{\boldsymbol{\lambda}}

\DeclareMathOperator*{\esssup}{ess\,sup}

\DeclareMathOperator{\dist}{dist}

\DeclareMathOperator{\Tr}{Tr} 

\DeclareMathOperator{\len}{len}

\DeclareMathOperator{\Lip}{Lip}
\DeclareMathOperator{\LIP}{LIP}

\DeclareMathOperator{\cCap}{\rm Cap}

\DeclareMathOperator{\cL}{\mathcal{L}}

\DeclareMathOperator{\cF}{\mathcal{F}}

\DeclareMathOperator{\cH}{\mathcal{H}}

\newcommand{\one}{\mathds{1}} 

\DeclareMathOperator{\Gr}{Gr}

\DeclareMathOperator{\supp}{supp}

\newcommand{\abs}[1]{{\left\vert\kern-0.25ex #1 \kern-0.25ex\right\vert}}
\DeclarePairedDelimiter\norm{\lVert}{\rVert}

\newcommand\restr[2]{{
		\left.\kern-\nulldelimiterspace 
		#1 
		\vphantom{\big|} 
		\right|_{#2} 
}} 

\def\XXint#1#2#3{{\setbox0=\hbox{$#1{#2#3}{\int}$ }
		\vcenter{\hbox{$#2#3$ }}\kern-.6\wd0}}

\begin{document}

	\begin{abstract}
		We affirmatively resolve the energy image density conjecture of Bouleau and Hirsch (1986). Beyond the original framework of Dirichlet structures, we establish the energy image density property in several related settings. In particular, we formulate a version of the property that encompasses  strongly local, regular Dirichlet forms, Sobolev spaces defined via upper gradients, and self-similar energies on fractals, thereby unifying these under a single framework. As applications, we prove the finiteness of the martingale dimension for diffusions satisfying sub-Gaussian heat kernel bounds, and we obtain a new proof of a conjecture of Cheeger concerning the Hausdorff dimension of the images of differentiability charts in PI spaces. The proof of the energy image density property is based on a structure theorem for measures and normal currents in $\mathbb{R}^n$ due to De Philippis--Rindler, together with the notions of decomposability bundles due to Alberti--Marchese and cone null sets due to Alberti--Csörnyei--Preiss and Bate.
	\end{abstract}

	\maketitle
	
	\section{Introduction}
	\subsection{Background and Overview}
	The goal of this work is a proof of the energy image density (EID) conjecture of Bouleau and Hirsch \cite[p.~251]{bouleauhirsch}. This conjecture arises as a natural generalization of a fundamental result in Malliavin calculus: the non-degeneracy (invertibility) of the Malliavin matrix associated with an $\mathbb{R}^n$-valued random variable implies the absolute continuity of its law with respect to Lebesgue measure. This result forms a key step in Malliavin's  proof of  H\"ormander's hypoellipticity theorem \cite{Hor, Mal,BH91, Nua,Hai}.  The Malliavin matrix admits a natural generalization to local Dirichlet forms which leads to 
	the energy image density conjecture of Bouleau and Hirsch. Roughly speaking, this conjecture asserts that the invertibility of a generalized Malliavin matrix of a random variable implies  the absolute continuity of the law of that random variable.

	For a precise statement of this conjecture, we recall the relevant definitions below.
	\begin{definition} \label{d:dirstructure}
		A \emph{Dirichlet structure}  $(X,\mathcal{X},\mu,\mathcal{E}, \mathcal{F})$ 	is a probability space $(X,\mathcal{X},\mu)$ along with a quadratic form $(\mathcal{E},\mathcal{F})$ on $L^2(X,\mu)$ that satisfies the following properties.
		\begin{enumerate}[(i)]
			\item (densely defined, non-negative definite, quadratic form) $\mathcal{E}: \mathcal{F} \times \mathcal{F} \to \mathbb{R}$ is bilinear, where $\mathcal{F}$ is a dense subspace of $L^2(X,\mu)$ and $\mathcal{E}(f,f) \ge 0$ for all $f \in \mathcal{F}$.
			\item (closed form) $(\mathcal{E},\mathcal{F})$ is a \emph{closed} form; that is, $\mathcal{F}$ is a Hilbert space equipped with the inner product
			\[
			\mathcal{E}_1(f,g):= \mathcal{E}(f,g)+ \langle f,g \rangle_{L^2(\mu)}, \quad \mbox{for all $f,g \in \mathcal{F}$}.
			\]
			\item (Markovian property) $f \in \mathcal{F}$ implies that $\widetilde{f}:=0 \vee (f \wedge 1) \in \mathcal{F}$ and $\mathcal{E}(\widetilde{f},\widetilde{f}) \le \mathcal{E}(f,f)$.
			\item $\one \in \mathcal{F}$ and $\mathcal{E}(\one,\one)=0$, where $\one$ is the constant function on $X$ that is identically one.
			\item (existence of carr\'e du champ operator) For all $f \in \mathcal{F} \cap L^\infty(m)$, there exists $\gamma(f,f) \in L^1(m)$ such that for all $h \in \mathcal{F} \cap L^\infty(m)$, we have 
			\[
			\mathcal{E}(fh,f)- \frac{1}{2}\mathcal{E}(h,f^2)= \int h \gamma(f,f) \,d\mu.
			\]
			\item (strong locality) For all $f,g \in \mathcal{F}$ and $a \in \mathbb{R}$ such that $(f+a)g =0$ implies $\mathcal{E}(f,g)=0$.
		\end{enumerate}
	\end{definition}
	By \cite[Proposition I.4.1.3]{BH91}, for any  Dirichlet structure  $(X,\mathcal{X},\mu,\mathcal{E},\mathcal{F})$  there exists a unique positive symmetric and continuous bilinear form $\gamma: \mathcal{F} \times \mathcal{F} \to L^1(\mu)$ (called the \emph{carr\'e du champ operator}) such that
	\begin{equation} \label{e:def-carreduchamp}
		\frac{1}{2}\mathcal{E}(fh,g)+\frac{1}{2}\mathcal{E}(gh,f)-\frac{1}{2}\mathcal{E}(h,fg)= \int h \gamma(f,g)\,d\mu, \quad \mbox{for all $f,g, \in \mathcal{F} \cap L^\infty$.} 
	\end{equation}
	For a function $f =(f_1,\ldots,f_n) \in \mathcal{F}^n$, we define the carr\'e du champ matrix   $\gamma(f)$ as the $n\times n$ matrix 
	\begin{equation} \label{e:def-cdc-matrix}
		\gamma(f) = \begin{bmatrix}
			\gamma(f_i,f_j)
		\end{bmatrix}_{1 \le i,j \le n} \in \mathbb{R}^{n \times n}.
	\end{equation}
	The carr\'e du champ matrix is the natural generalization of the Malliavin matrix to the setting of strongly local Dirichlet forms. In the special case of the Ornstein--Uhlenbeck Dirichlet form on   Wiener space,  this carr\'e du champ matrix coincides with the classical Malliavin matrix. We refer to \cite[Definition II.2.3.3]{BH91}  and \cite[p.~92]{Nua} for the definitions of the Ornstein--Uhlenbeck Dirichlet form on   Wiener space  and   Malliavin matrix respectively. Hence the energy image density property in Definition \ref{d:eid} and Conjecture \ref{con:eid} below can be viewed as a sweeping generalization of  Malliavin's criterion for absolute continuity of the law of a random variable on Wiener space using non-degeneracy of the Malliavin matrix under minimal regularity assumption \cite[Theorem 5.5]{Mal}.

	Next, we recall the energy image density property introduced by Bouleau and Hirsch.
	\begin{definition} \label{d:eid}
		Let $(X,\mathcal{X},\mu,\mathcal{E},\mathcal{F})$ be a Dirichlet structure with the associated carr\'e du champ operator $\gamma: \mathcal{F} \times \mathcal{F} \to L^1(\mu)$. We say that the Dirichlet structure  $(X,\mathcal{X},\mu,\mathcal{E},\mathcal{F})$ satisfies the \emph{energy image density property}, if 	 for any $n \in \mathbb{N}$, $f \in \mathcal{F}^n$, we have
		\begin{equation} \label{e:imagemeas}
			f_*(\one_{\{\det(\gamma(f))>0\}}\cdot \mu) \ll \mathcal{L}_n,
		\end{equation}
		where $\mathcal{L}_n$ is the Lebesgue measure on $\mathbb{R}^n$.
	\end{definition}
	Let us briefly explain the condition $\det(\gamma(f))>0$ in \eqref{e:imagemeas}. 
	Consider the Dirichlet form for Brownian motion of $\mathbb{R}^n$ given by $\mathcal{E}(g,g)= \int \abs{\nabla g}(x)^2\,dx$ for all $g \in \mathcal{F}=W^{1,2}(\mathbb{R}^n)$, where $\nabla g$ denotes the distributional gradient of $g$. Let $f:\mathbb{R}^n \to \mathbb{R}^n$ be a smooth function whose components are in the Sobolev space $W^{1,2}(\mathbb{R}^n)$. In this case, $\gamma(f_i,f_j)= \nabla f_i \cdot \nabla f_j$ and hence the condition $\det(\gamma(f))(x)>0$ is equivalent to the statement that the differential $Df$ of $f$ at $x$ is surjective, or equivalently, $f$ is a submersion at $x$.  For an abstract Dirichlet form, this condition can be interpreted more generally as requiring that $f$ `varies in all infinitesimal directions' at $x$. Intuitively, the invertibility of the carr\'e du champ matrix ensures that $f$  has enough local variability so that its image is sufficiently `spread out', a viewpoint that directly motivates the energy image density condition in more abstract settings. As we shall see in Definition \ref{d:p-independent}, this condition $\det(\gamma(f))>0$  can also be interpreted as an infinitesimal version of linear independence of the components of $f$; cf.~Lemmas \ref{l:invertibly-independence} and \ref{l:invertibly-independence-dirstructure}.

	The Bouleau-Hirsch conjecture \cite[p.~251]{bouleauhirsch} is stated as follows. 
	\begin{conjecture}[EID conjecture]  \label{con:eid}
		Every Dirichlet structure satisfies the energy image density property.
	\end{conjecture}
	
	Our main result is the positive answer to Bouleau--Hirsch conjecture. 
	\begin{theorem}\label{thm:eid-dirstructure}	
		Let $(X,\mathcal{X},\mu,\mathcal{E},\mathcal{F})$ be a Dirichlet structure with the associated carr\'e du champ operator $\gamma: \mathcal{F} \times \mathcal{F} \to L^1(\mu)$. For any $n \in \mathbb{N}$, $\bphi \in \mathcal{F}^n$, we have
		\begin{equation*}
			f_*(\one_{\{\det(\gamma(\bphi))>0\}}\cdot \mu) \ll \mathcal{L}_n,
		\end{equation*}
		where $\mathcal{L}_n$ is the Lebesgue measure on $\mathbb{R}^n$, and $\gamma(\bphi)$ is the carr\'e du champ matrix as defined in \eqref{e:def-cdc-matrix}.
	\end{theorem}
	
	We discuss some past works related to the energy image density property.
	As evidence towards this conjecture, Bouleau and Hirsch obtained the  energy image density property for scalar-valued functions; that is, when  $n=1$ in \eqref{e:imagemeas}  \cite[Corollaire 6]{Bou}, \cite[Th\'eor\`eme 4.1]{bouleauhirsch}. Furthermore, they verified Conjecture \ref{con:eid} for the Ornstein-Uhlenbeck Dirichlet form on Wiener space \cite[Th\'eor\`eme 9]{BH86b}. In this context, the energy image density property is better known as the Bouleau--Hirsch criterion for absolute continuity and is widely used in stochastic analysis \cite{AP,BNQZ,CFV,CFG,NV,NQ,Sch}. More recently, further evidence toward Conjecture \ref{con:eid} was obtained by Malicet and Poly in the general setting for all values of $n \in \mathbb{N}$: they showed that if $f \in \mathcal{F}^n, n \in \mathbb{N}$ satisfies $\det(\gamma(f))>0$ $\mu$-almost surely, then the law of $f$ is a Rajchman measure on $\mathbb{R}^n$  \cite[Theorem 2.2]{MP}.

	Prior to this work, two main approaches have been used to establish  the energy image density property in higher dimensions $n \ge 2$.  The first approach, originating from Malliavin's seminal paper \cite[p.~196]{Mal}, involves showing that the distributional derivative of the pushforward measure is a Radon measure \cite[Lemma 2.1.1]{Nua} or \cite[Lemma I.7.2.2.1]{BH91}. This approach relies on an integration by parts formula \cite[\textsection 2.1.2]{Nua}.
	However, in order to generalize this integration by parts approach to the setting of Dirichlet forms we need to assume additional smoothness assumptions that the entries of the carr\'e du champ matrix (or generalized Malliavin matrix) belongs to the domain of the Dirichlet form and that the function belongs to the domain of the generator \cite[Theorem I.7.2.2]{BH91}.  
	
	The second approach that is due to Bouleau and Hirsch relies on Federer's co-area formula. Its main advantage   is that it requires minimal assumptions on regularity (compare \cite[Theorem 2.1.2]{Nua} with \cite[Theorem 2.1.1]{Nua}), which makes it particularly useful for applications to stochastic differential equations with less regular coefficients \cite[Section V]{BH86b}. However, implementing this co-area formula approach requires significant additional structural assumptions on the Dirichlet form \cite[Theorem II.5.2.2]{BH91}. Nevertheless, this approach has been successfully extended to a number of examples of interest \cite{Coq,Son,BD09,BD15}. The results of Bouleau and Denis \cite{BD09,BD15}
	on EID property has applications to obtain regularity of solutions to stochastic differential equations with jumps. More generally, all of the applications of energy image density property mentioned above concern  regularity of solutions to stochastic differential equations (with the exception of \cite{AP} which concerns random nodal volumes).

	Even in the classical case of the Ornstein--Uhlenbeck Dirichlet structure on Wiener space, Malliavin already pointed out that the only available method for establishing the energy image density property proceeds via the co-area formula. Indeed, in \cite[p.~86]{Mal-book}, he writes ``it should be emphasized that there does not exist currently an alternative approach to the regularity of laws for $\mathbb{R}^d$-valued functionals with one derivative." 
	This observation highlights the necessity of developing a new approach to Conjecture~\ref{con:eid}, as a co-area formula is not available in general.
	
	\subsection{Outline of the proofs} \label{ss:proof-outline}
	The crux of our proof is a deep structure theorem of  measures and normal currents in $\mathbb{R}^n$ developed by De Philippis and Rindler \cite{DR} that gives a criterion for absolute continuity of measure in terms of linearly independent normal currents as we recall below. The relevant terminology and notation related to currents is recalled in \textsection \ref{ss:alternate}.  
	\begin{theorem} \label{t:dr-current} \cite[Corollary 1.12]{DR} Let $T_1=\vec{T_1} \norm{T_1}, \ldots, T_n= \vec{T_n} \norm{T_n}$ be one-dimensional normal currents on $\mathbb{R}^n$ such that there exists a positive Radon measure $\nu$ on $\mathbb{R}^n$ with the following properties:
		\begin{enumerate}[(i)]
			\item $\nu \ll \norm{T_i}$ for $i=1,\ldots,n$;
			\item For $\nu$-a.e.~$x \in \mathbb{R}^n$, $\operatorname{span}\{\vec{T_1}(x),\ldots,\vec{T_n}(x)\}=\mathbb{R}^n$.
		\end{enumerate}
		Then $\nu \ll \mathcal{L}_n$.
	\end{theorem}
	We provide two distinct proofs of Theorem~\ref{thm:eid-dirstructure}, each invoking Theorem \ref{t:dr-current} in a substantially different manner.
	The first proof (\textsection \ref{ss:alternate}) proceeds by a direct application of the sufficient condition for absolute continuity stated in Theorem~\ref{t:dr-current} to the pushforward measure in the statement of Conjecture \ref{con:eid}. Although this yields a comparatively short argument, it depends essentially on the Hilbert space structure of~$\mathcal{F}$.
	The other proof (\textsection \ref{ss:lsc-approach}), while longer, is more robust and extends beyond the classical framework of Dirichlet structures. Within this approach, we establish the energy image density property in a unified setting encompassing strongly local, regular Dirichlet forms, Sobolev spaces defined via upper gradients, and self-similar energies on fractals. As we explain in \textsection\ref{ss:ext-apply}, these generalizations are motivated by further applications, including the finiteness of the martingale dimension and a new proof of Cheeger’s conjecture concerning the Hausdorff dimension of images of differentiability charts in PI spaces.
	
	Let us first outline how to use Theorem \ref{t:dr-current}  directly to the measure $\nu:=f_*(\one_{\{\det(\gamma(\bphi))>0\}}\cdot \mu)$ where $f \in \mathcal{F}^n$ to obtain the energy image density property.  In order to apply Theorem  \ref{t:dr-current}, we sketch a simple construction of normal currents associated to Dirichlet structures. If 	
	$f=(f_1,\ldots,f_n) \in \mathcal{F}^n$ and $g \in \mathcal{F}$, then the pushforward of the $\mathbb{R}^n$-valued measure $(\gamma(f_i,g)\cdot \mu)_{1 \le i \le n}$ under $f:X \to \mathbb{R}^n$ is an $\mathbb{R}^n$-valued finite Borel measure on $\mathbb{R}^n$ and hence can be viewed as a current with finite mass. If $D(A)$ denotes the domain of the generator $A$ associated with the Dirichlet structure and if we assume in addition that $g \in D(A)$, then the current described above happens to be a normal current in the sense of Federer and Fleming. 
	This is a simple consequence of the chain rule for Dirichlet forms (see Lemma \ref{l:normal-current}).
	If $f \in \mathcal{F}^n$, the currents $T_i= f_*\left( (\gamma(f_j,f_i)\cdot \mu)_{j=1}^n \right)$ for $i=1,\ldots,n$ satisfy the following properties  (see Lemma \ref{l:dr-assume}): 
	\[
	f_* \left(\one_{\{\det(\gamma(f))>0\}}\cdot \mu\right)  \ll \norm{T_i}, \quad \mbox{for all $i=1,\ldots,n$},   
	\]
	where $\norm{T_i}$ denotes the mass measure of $T_i=\vec{T_i} \cdot \norm{T_i}$, and 
	\[
	\operatorname{span}\{\vec{T_1}(x),\ldots,\vec{T_n}(x)\}=\mathbb{R}^n, \quad \mbox{for 	$f_* \left(\one_{\{\det(\gamma(f))>0\}}\cdot \mu\right)$-a.e.}	\]
	If $f \in D(A)^n$, we can apply Theorem \ref{t:dr-current}   with the above normal currents $\{T_i:1 \le i \le n\}$ to obtain the energy image density property. For a general $f \in \mathcal{F}^n$, $T_i, 1 \le i \le n$ need not be   normal currents, and we proceed by an approximation argument using the density of $D(A)$ in the Hilbert space $\mathcal{F}$. Readers interested solely in the proof of Theorem \ref{thm:eid-dirstructure}, and not in the extensions or applications discussed above, may proceed directly to \textsection\ref{ss:alternate}, as the material in \textsection\ref{s:p-Dirspace} and \textsection\ref{s:measures} is not used therein.
	
	The other approach is more closely related to the proof of the energy image density property due to Bouleau and Hirsch in the scalar case ($n=1$), which we recall below in a slightly adapted form \cite[Proof of Theorem I.5.2.3]{BH91}. Let $f \in \mathcal{F}$ and  $K \subset \mathbb{R}$ be  compact  such that $\mathcal{L}_1(K)=0$.
	Consider the sequence of $1$-Lipschitz, continuously differentiable functions $g_k: \mathbb{R} \to \mathbb{R}$ defined by 
	\[
	g_k(0)=0, \quad 	g_k'(x):= 1 \wedge (n \dist(x,K)), \quad \mbox{for all $x \in \mathbb{R}, k\in \mathbb{N}$.}
	\]
	Then it is easy to see that $g_k$ converges to the identity map pointwise, and for any $f \in \mathcal{F}$, $g_k \circ f$ converges weakly in the Hilbert space $\mathcal{F}$ to $f$. Thus by the weak lower semicontinuity of energy measures together with the chain rule, we obtain
	\begin{align*}
		\int_{f^{-1}(K)} \gamma(f,f)\,d\mu  &\le \liminf_{n \to \infty}\int_{f^{-1}(K)} \gamma(g_k \circ f,g_k \circ f)\,d\mu\\
		& =  \liminf_{n \to \infty}\int_{f^{-1}(K)} \abs{g_k'(f(x))}^2\gamma( f,  f)\,d\mu =0,
	\end{align*}
	where the last equality follows from the fact that $g_k'=0$ on $K$. By the regularity of the measure $f_*(\gamma(f,f) \cdot \mu)$, this establishes the energy image density property in the scalar case. In the case $n\ge 2$, the same strategy applies, but the construction of suitable approximations is significantly more delicate and relies on sophisticated tools such as the notion of decomposability bundle due to Alberti and Marchese \cite{AlbMar}, cone null sets due to Alberti,  Cs\"ornyei,  and Preiss \cite{ACPdiff,ACPstruct},  a deep structure theorem of singular measures and normal currents in $\mathbb{R}^n$ due to De Philippis and Rindler mentioned above in the other approach \cite{DR}, and a new approximation result (Proposition \ref{prop:approximations}). Our approximations are only Lipschitz instead of $C^1$, so we also develop a suitable substitute for chain rule for Lipschitz functions in $\mathbb{R}^n$  (see  Proposition \ref{prop:multidimcontraction-structure}). This type of argument has recently been used to address somewhat similar problems: by di Marino, Lu\v{c}i\'c, and Pasqualetto \cite{pasqualetto} for a question of Fukushima on the closability of energy forms, and by Alberti, Bate, and Marchese \cite{BAM} for the closability of differential operators.

	Next, we describe the ideas behind the  construction of approximations used in our weak lower semicontinuity approach.  
	Let $f \in \mathcal{F}^n$ with $n \ge 2$. Since $\gamma(f)$ is a nonnegative semidefinite matrix-valued function, positivity of the determinant can be expressed as  
	\[
	\{\det(\gamma(f))>0\} = \bigcup_{k=1}^\infty \Bigl\{ \inf_{\abs{\blambda}=1} \blambda^T \gamma(f) \blambda \,\ge\, k^{-1}\Tr(\gamma(f)) >0 \Bigr\},
	\]
	where $\blambda \in \mathbb{R}^n$ and $\Tr(\cdot)$ denotes the trace. Thus it is enough to show that for any $\delta>0$, the pushforward  
	\[
	\nu := f_*(\one_A \mu)
	\quad \text{satisfies} \quad
	\nu \ll \mathcal{L}_n,
	\]
	with  
	\[
	A = \left\{ \inf_{\abs{\blambda}=1} \blambda^T \gamma(f)\blambda \,\ge\, \delta \Tr(\gamma(f)) >0 \right\}.
	\]  
	
	Unlike the scalar case, the argument here proceeds by contradiction. Suppose there exist $\delta>0$ and $f \in \mathcal{F}^n$ such that $\nu$ is not absolutely continuous with respect to $\mathcal{L}_n$. Then one finds a compact set $K \subset \mathbb{R}^n$ with $\mathcal{L}_n(K)=0$ but $\nu(K)>0$. In this case the restriction $\one_K \nu$ admits a decomposability bundle of dimension at most $n-1$, $\nu$-a.e. on $K$.  
	
	The decomposability bundle $T^{AM}_\nu(x)$, introduced by Alberti and Marchese \cite{AlbMar}, assigns to each $x \in \mathbb{R}^n$, a subspace of $\mathbb{R}^n$ capturing the directions in which every Lipschitz function is $\nu$-a.e. differentiable. They proved the existence  of this bundle, and combined with the results of \cite{DR} one obtains the following key consequence: for any $\epsilon>0$, there exist a compact $B \subset K$ with $\nu(B)>0$ and a unit vector $\blambda=(\lambda_1,\ldots,\lambda_n) \in \mathbb{R}^n$ such that every rectifiable curve $\sigma:[0,l]\to\mathbb{R}^n$ satisfies  
	\begin{equation}\label{e:cnull-intro}
		\abs{\langle \blambda, \sigma'(t)\rangle} \le \epsilon \abs{\sigma'(t)}, \qquad \text{for $\mathcal{L}_1$-a.e. } t \in \sigma^{-1}(B).
	\end{equation}  
	
	Heuristically, \eqref{e:cnull-intro} expresses that the set $B$ is almost orthogonal to $\blambda$ in an infinitesimal sense, a property formalized through cone-null sets (see Definition \ref{d:cone-null} and Lemma \ref{l:conenull}). The existence of such $B$ follows from the link between decomposability bundles and cone-null sets \cite[Lemma 7.5]{AlbMar}, together with the result of De Philippis--Rindler ensuring that $\dim T^{AM}_\nu(x) \le n-1$ for $\nu$-a.e. $x \in K$ (see Propositions \ref{prop:decompoabilitybundlesinuglar} and \ref{prop:decompositionconenull}).

	The estimate \eqref{e:cnull-intro} satisfied by $B$ allows us to construct a sequence of $1$-Lipschitz approximations $g_k:\mathbb{R}^n \to \mathbb{R}, k \in \mathbb{N}$ of $g(y):= \langle \blambda, y \rangle$ such that $$\lim_{k \to \infty} g_k(y)= g(y) \quad \mbox{for all $y \in \mathbb{R}^n$}, \quad g_k(0)=g(0)=0, \quad \mbox{for all $k \in \mathbb{N}$,}$$
	and 
	\[
	\Lip_a[g](y) \le \epsilon, \quad \mbox{for all $k \in \mathbb{N}, y \in B$,}
	\]
	where $\Lip_a[g]$ is the asymptotic Lipschitz constant (see \eqref{e:def-Lipa} and \textsection \ref{ss:approximation}). This estimate along with weak lower semicontinuity of energy measures and a  version of chain rule for Lipschitz functions   implies that 
	\begin{align*}
		\int_{f^{-1}(B) \cap A}	\gamma\left( \sum_{i=1}^n \lambda_i f_i,  \sum_{i=1}^n \lambda_i f_i\right) \,d\mu &\le \liminf_{k \to \infty} \int_{f^{-1}(B) \cap A}\gamma\left( g_k \circ f,  g_k \circ f \right) \,d\mu\\
		&\le \liminf_{k \to \infty} \int_{f^{-1}(B) \cap A} (\Lip_a[g] \circ f)^2 \sum_{i=1}^n \gamma (f_i,f_i) \,d\mu\\	\\
		&\le \epsilon^2 \int_{f^{-1}(B) \cap A} \sum_{i=1}^n \gamma (f_i,f_i)\, d\mu.
	\end{align*}
	On the other hand, the definition of $A$ implies that
	\[
	\int_{f^{-1}(B) \cap A}	\gamma\left( \sum_{i=1}^n \lambda_i f_i,  \sum_{i=1}^n \lambda_i f_i\right)\, d\mu \ge \delta \int_{f^{-1}(B) \cap A} \sum_{i=1}^n \gamma (f_i,f_i)\, d\mu >0.
	\]
	The above two estimates lead to the desired contradiction if we choose $\epsilon^2 < \delta$.
	
	The approximating functions $g_k$ are constructed by an infimizing procedure, and the convergence $g_k\to g$ is shown with a compactness argument. A model for this type of argument comes from the proof of Heinonen and Koskela for the equality of modulus and capacity \cite[Proposition 2.17]{heinonenkoskela}. Our argument adapts this proof to give an approximation for functions; see Proposition \ref{prop:approximations}. A similar approach was used by Cheeger \cite[Sections 5,6,9]{Ch}, which inspired the methods in \cite{seblipdense,bate2024fragment, continuousdense}. 
	
	\subsection{Extensions and applications} \label{ss:ext-apply}
	A remarkable feature of the energy image density property is its broad generality. We establish this property for all strongly local, regular Dirichlet forms \cite{FOT,CF} in Theorem~\ref{thm:eid-mmdspace}, and further extend it to \emph{$p$-Dirichlet spaces} and \emph{$p$-Dirichlet structures}, which generalize regular strongly local Dirichlet forms and Dirichlet structures, respectively. These spaces can be viewed as analogues of the classical Sobolev space $W^{1,p}$ and have recently attracted considerable attention in the analysis on metric spaces \cite{Ch,Shanmun,HKST} and on fractals \cite{Kus93, Kigami,pgasket,cao,Murugan-Shimizu} (see Example~\ref{ex:p-dir} for details). We apply the energy image density property for regular, strongly local Dirichlet forms to prove the finiteness of the martingale dimension under sub-Gaussian heat kernel bounds, thereby verifying a recent conjecture \cite[Conjecture 3.12]{murugananalytic}. Moreover, the extension to $p$-Dirichlet spaces yields a new proof of Cheeger’s conjecture on the Hausdorff dimension of images of charts in Lipschitz differentiability spaces \cite[Conjecture 4.63]{Ch}, previously resolved in \cite{DMR}; see Proposition \ref{prop:cheeger}.
	
	The formulation of the energy image density property for regular Dirichlet forms requires some care, since energy measures need not be absolutely continuous with respect to the reference measure \cite{Kus89,Kus93,Hino-sing,KM20}; in this setting, the energy measure of $f$ is the analogue of $\gamma(f,f)\,\mu$ in Dirichlet structures.
	In this case, we define the carr\'e du champ matrix by taking Radon–Nikodym derivatives of energy measures with respect to a \emph{minimal energy dominant measure} (see Definition~\ref{d:minimal-energy-dominant}). Such a measure has the property that the energy measure of every function in the domain of the Dirichlet form is absolutely continuous with respect to it, which allows us to define a suitable analogue of the carr\'e du champ matrix $\gamma_\nu(f)$ for each $f \in \mathcal{F}^n$ (see~\eqref{e:def-cdc-reg}).

	The modified carr\'e du champ matrix $\gamma_\nu(f)$  is essential for characterizing and analyzing martingale dimension.   Martingale dimension originates from a result of Ventcel’ \cite{Ven}, who showed that every square-integrable martingale additive functional of Brownian motion in $\mathbb{R}^n$ can be represented as a sum of $n$ stochastic integrals with respect to the coordinate martingales. More generally, the martingale dimension of a Markov process is the smallest integer $k$ such that every martingale additive functional, under suitable integrability conditions, can be expressed as a sum of $k$ stochastic integrals with respect to a fixed family of $k$ martingale additive functionals \cite[Definition~3.3]{hinoenergymeasures}. The concept traces back to \cite{MW,DV}, and, building on Kusuoka’s ideas \cite{Kus89,Kus93}, Hino showed that martingale dimension is equivalently given by the maximal rank of the modified carr\'e du champ matrix \cite[Theorem~3.4]{hinoenergymeasures}. 	The martingale dimension can be interpreted as the dimension of the tangent space of a measurable Riemannian structure associated to the Dirichlet form \cite[Theorem~3.4]{Hin13b}. For strongly local Dirichlet forms whose heat kernel satisfies Gaussian bounds, the martingale dimension coincides with the dimension of Cheeger’s measurable cotangent bundle constructed in his work on the almost-everywhere differentiability of Lipschitz functions on metric measure spaces \cite{Ch}, \cite[Theorem~3.2]{murugananalytic}.
	
	We now explain how the energy image density property can be used to obtain upper bounds on the martingale dimension. Suppose that the domain of the Dirichlet form contains a dense set of $\alpha$-H\"older continuous functions for some $\alpha \in (0,1]$, and that the underlying space has finite Hausdorff dimension $d_H$. This situation occurs for diffusion processes with sub-Gaussian heat kernel bounds, where the existence of such a dense set follows from the H\"older regularity of the heat kernel. By the continuity of energy measures, it is enough to bound the rank of the carr\'e du champ matrix $\gamma_\nu(f)$ for functions $f \in \mathcal{F}^n$ whose components lie in this dense subset of $\alpha$-H\"older functions. 
	
	Since the image of the set $\{\gamma_\nu(f)>0\}$ under such an $\alpha$-H\"older map $f$ can increase Hausdorff dimension by at most a factor of $\alpha^{-1}$, the set $f(\{\gamma_\nu(f)>0\})$ has Hausdorff dimension at most $d_H/\alpha$. As a consequence of the energy image density property, if the martingale dimension is at least $n$, then there exists $f \in \mathcal{F}^n$, $\alpha$-H\"older continuous, such that $f(\{\gamma_\nu(f)>0\})$ supports a nonzero measure absolutely continuous with respect to $\mathcal{L}_n$. In this case, the Hausdorff dimension of $f(\{\gamma_\nu(f)>0\})$ equals $n$, and we thus obtain an upper bound $d_H/\alpha$ for the martingale dimension. 
	
	We note that the study of martingale dimension for diffusions satisfying sub-Gaussian heat kernel bounds dates back to the seminal work of Barlow and Perkins, who first raised the question for Brownian motion on the Sierpiński gasket \cite[Problem 10.6]{BP}. Although this problem was soon solved by Kusuoka \cite{Kus89,Kus93}, the martingale dimension remains unknown for many examples, such as generalized Sierpiński carpets. The best available estimates for martingale dimension are restricted to self-similar spaces and Dirichlet forms (see, for example, \cite[Theorem 3.5]{Hin13a}). On the other hand, there are several examples of diffusions satisfying sub-Gaussian heat kernel bounds where the underlying space and Dirichlet form are not self-similar \cite{BH,HKKZ}. In such cases, even the finiteness of the martingale dimension was not known prior to this work. We hope that the energy image density property can lead to better understanding of martingale dimension as a number of questions still remain open \cite[\textsection 3.2]{murugananalytic}.

	As another application of the energy image density property, we include a short new proof of Cheeger's conjecture on the Hausdorff dimension of image of differentiability charts in PI spaces (Proposition \ref{prop:cheeger}). Cheeger's conjecture was resolved  by a result of De Philippis, Marchese and Rindler \cite[Theorem~4.1.1]{DMR} that shows that the image of reference measure under differentiability chart is absolutely continuous with respect to the Lebesgue measure. The setting of \cite{DMR} is more general as their result is for Lipschitz differentiability spaces of which PI spaces are a subclass.

	Yet another application of our results is an answer to a question raised by Ambrosio and Kirchheim \cite[p.~15]{AKircheim}. A result of Preiss (\cite[Theorem 3.3]{AKircheim}) states that for any finite measure $\mu$ on $\mathbb{R}^2$ that is not absolutely continuous with respect to $\mathcal{L}_2$, there exists a sequence of continuously differentiable, Lipschitz functions $g_n: \mathbb{R}^2 \to \mathbb{R}^2$ converging   pointwise to the identity map with uniformly bounded Lipschitz constant, and such that 
	\[
	\lim_{n \to \infty}	\int_{\mathbb{R}^2} \det(\nabla g_n)\,d\mu < \mu(\mathbb{R}^2).
	\]
	We show that this theorem is true for $\mathbb{R}^n$ for all $n \in \mathbb{N}$ in Theorem \ref{t:detjacobian}. This result along with the same argument in \cite[Proof of Theorem  3.8]{AKircheim} implies that any top-dimensional Ambrosio--Kirchheim metric current on $\mathbb{R}^n$ has absolutely continuous mass measure for any $n \in \mathbb{N}$, a result obtained  earlier by different methods in \cite[Theorem 1.15]{DR} (see Corollary \ref{cor:metric-current}).

	\subsection{Outline of the paper}
	The approach based on weak lower semicontinuity and approximations mentioned in \textsection\ref{ss:proof-outline} applies   to a broad class of examples. To capture the generality of this approach, in \textsection\ref{s:p-Dirspace} we introduce two abstract frameworks of $p$-Dirichlet spaces and $p$-Dirichlet structures that generalize regular strongly local Dirichlet forms and Dirichlet structures respectively to a non-linear setting. We develop some basic   notions such as capacity, quasicontinuous functions in this framework (\textsection\ref{ss:cap-qcts}) along with the notion of $p$-independence (Definition \ref{d:p-independent}) that is a generalization of the non-degeneracy condition on carr\'e du champ matrix. In \textsection\ref{s:measures}, we develop the necessary preliminary results on decomposability bundles, cone null sets and Lipschitz approximations for the approach based on lower semicontinuity. In \textsection\ref{s:proofs}, we present the two different proofs based on lower semicontinuity (\textsection\ref{ss:lsc-approach}) and on normal currents constructed from Dirichlet forms (\textsection\ref{ss:alternate}).
	In \textsection\ref{s:applications}, we discuss different applications of our results such as a finiteness theorem for martingale dimension (\textsection\ref{ss:martingale-dim}), a new proof of Cheeger's conjecture (\textsection\ref{ss:cheeger}) and an answer to a problem raised by Ambrosio and Kirchheim regarding a generalization of a result due to Preiss (\textsection\ref{ss:preiss}).
	\subsection*{Notation}
	\noindent
	For a subset $A \subset X$, we write $\one_A$ for the indicator function of $A$. When $A = X$, we abbreviate $\one_X$ as $\one$.
	
	\smallskip
	\noindent
	In any metric space $(X,d)$, we denote by
	\[
	B(x,r) = \{ y \in X : d(x,y) < r \}
	\quad \text{and} \quad
	\overline{B}(x,r) = \{ y \in X : d(x,y) \le r \}
	\]
	the open and closed balls of radius $r > 0$ centered at $x \in X$.
	
	\smallskip
	\noindent
	The inner product in $L^2(X,\mu)$ is denoted by 
	$\langle \cdot, \cdot \rangle_{L^2(X,\mu)}$ or simply 
	$\langle \cdot, \cdot \rangle_{L^2(\mu)}$.
	For $x, y \in \mathbb{R}^n$, we denote by $\abs{x}$ and 
	$\langle x, y \rangle$ the Euclidean norm and inner product, respectively. 
	Unless specified otherwise, $\mathbb{R}^n$ is equipped with the Euclidean distance.
	
	\smallskip
	\noindent
	For a square matrix $M$, we denote by $\det(M)$ and $\Tr(M)$ its determinant and trace, respectively.
	
	\smallskip
	\noindent
	If $\mu$ is a measure and $f$ is a measurable function, then 
	$f\mu$ (or $f \cdot \mu$) denotes the measure 
	$A \mapsto \int_A f\,d\mu$. 
	If $\mu$ and $\nu$ are two measures on the same space, we write 
	$\mu \ll \nu$ and $\mu \perp \nu$ to indicate absolute continuity and singularity, respectively. We write $\mu \leq \nu$ if $\mu(A)\leq \nu(A)$ for all measurable sets $A$. We denote the measure $\one_A \mu$ also as $\mu|_A$.
	
	\smallskip
	\noindent
	If $f:X \to Y$ is a function and $A \subset X$, then $\restr{f}{A}: A \to Y$ denotes the restriction of $f$ to $A$.
	
	\smallskip
	\noindent
	For $a,b \in \mathbb{R}$, we write $a \vee b := \max(a,b)$ and $a \wedge b := \min(a,b)$.
	\section{Background on \texorpdfstring{$p$}{p}-Dirichlet spaces} \label{s:p-Dirspace}
	\subsection{Lipschitz functions}
	
	If $X$ is a metric space, we call $f:X \to \mathbb{R}$ ($L-$)Lipschitz if 
	\[
	\frac{|f(x)-f(y)|}{d(x,y)}\le L
	\]
	for all $x,y\in X$. The smallest constant $L$ such that $f$ is $L-$Lipschitz is the Lipschitz constant of $f$ and is denoted $\LIP[f]$. For $x\in X$ the asymptotic Lipschitz constant of $f$ at $x$ is denoted
	\begin{equation} \label{e:def-Lipa}
		\Lip_a[f](x)=\lim_{r\to 0} \LIP[f|_{B(x,r)}].
	\end{equation}
	We denote the vector space of $\mathbb{R}$-valued Lipschitz functions on $X$ by $\Lip(X)$.
	
	Let  $f \in \Lip(\mathbb{R}^n)$ and $S \subset \mathbb{R}^n$ be a countable dense set. Then the function $x \mapsto \LIP[f|_{B(x,r)}]$ is Borel measurable as 
	\[
	x \mapsto  \LIP[f|_{B(x,r)}]= \sup_{(p,q) \in S \times S, p \neq q} \frac{\abs{f(p)-f(q)}}{\abs{p-q}} \one_{B(p,r)} (x) \one_{B(q,r)}(x)
	\]	
	is a supremum of a countable family of Borel functions. Therefore $\Lip_a(f)(x)= \lim_{r \downarrow 0} \LIP[f|_{B(x,r)}]$ is also Borel measurable.
	
	A map $f$ is called ($L-$)bi-Lipschitz, if
	\[
	L^{-1}\le \frac{|f(x)-f(y)|}{d(x,y)}\le L
	\]
	for all $x,y\in X$.
	
	\subsection{\texorpdfstring{$p$}{p}-Dirichlet spaces}
	
	Dirichlet structures are defined in Definition \ref{d:dirstructure}. In this section, we define a variant of this setting to $p\neq 2$ for locally compact spaces. 
	
	\begin{definition} \label{d:dir-space}
		We call $(X,d,\mu, \mathcal{E}_p, \mathcal{F}_p, \Gamma_p)$ a local $p$-Dirichlet space if the following hold:
		\begin{enumerate}  
			\item \textbf{Locally compact space}: $(X,d,\mu)$ is a locally compact metric space equipped with a Radon measure $\mu$.
			\item \textbf{Completeness:} 	The space $\mathcal{F}_p$  is a subspace of $L^p(X,\mu)$ and
			$\mathcal{E}_p: \mathcal{F}_p \to [0,\infty)$ is a non-negative function such that $\mathcal{F}_p$ is a Banach space when equipped with the norm $\norm{f}_{\mathcal{F}_p}=(\norm{f}_{L^p}^p+\mathcal{E}_p(f))^{1/p}$.
			\item \textbf{Homogeneity:} For all $f\in \mathcal{F}_p$ there exists a finite non-negative Borel measure $\Gamma_p \langle f \rangle$ on $X$ such that $\Gamma_p \langle f \rangle(X)= \mathcal{E}_p(f)$   and for all $\lambda \in \mathbb{R}$
			\[
			\Gamma_p\langle \lambda f\rangle = |\lambda|^p \Gamma_p\langle f\rangle.
			\]
			\item \textbf{Sublinearity:} For every $f,g\in \mathcal{F}_p$ and every Borel set $A\subset X$, 
			\begin{equation} \label{e:def-sublinear}
				\Gamma_p\langle f+g\rangle(A)^{\frac{1}{p}}\le \Gamma_p\langle f\rangle(A)^{\frac{1}{p}}+\Gamma_p\langle g\rangle(A)^{\frac{1}{p}}
			\end{equation}
			\item \textbf{Chain rule:} For all $f\in \mathcal{F}_p$, $g\in \Lip(\mathbb{R})$ with $g(0)=0$,  we have $g \circ f \in \mathcal{F}_p$ and
			\begin{equation} \label{e:def-chain}
				d\Gamma_p\langle g\circ f\rangle \le \LIP[g]^p d\Gamma_p\langle f\rangle.
			\end{equation}
			\item \textbf{Locality:} For every $f\in \mathcal{F}_p$ and open set $A\subset X$, if $f|_A=c$ $\mu$-a.e., then $\Gamma_p\langle f\rangle(A)=0$. 
			\item \textbf{Weak lower semicontinuity:} For every $f\in L^p(X)$, and for any sequence of functions $(f_i)_{i \in \mathbb{N}}$ in $\mathcal{F}_p$ such that $f_i\to f \in L^p(X)$ with $\sup_{i\in \mathbb{N}} \mathcal{E}_p(f_i)<\infty$, then $f\in \mathcal{F}_p$ and
			\begin{equation} \label{e:def-lsc}
				\Gamma_p\langle f \rangle(A)\le \liminf_{i\to \infty} \Gamma_p\langle f_i \rangle(A)
			\end{equation}
			for every Borel set $A\subset X$. 
		\end{enumerate}
	\end{definition}

	\begin{remark} \label{r:local-subadditive}
		$(X,d,\mu, \mathcal{E}_p, \mathcal{F}_p, \Gamma_p)$ a local $p$-Dirichlet space.
		We record some simple consequences of the definition above.
		\begin{enumerate}[(i)]
			\item 	Suppose $f,g \in \mathcal{F}_p$, $O\subset X$ be open such that $\restr{(f-g)}{O} \equiv c$  for some $c \in \mathbb{R}$, then 
			\begin{equation} \label{e:loc-subl}
				\Gamma_p \langle f \rangle (A)=	\Gamma_p \langle g \rangle (A)
			\end{equation}
			for all Borel sets $A \subset O$. This follows from sublinearity, homogeneity, along with $\Gamma_p \langle f-g \rangle (A) \le \Gamma_p \langle f-g \rangle (O)=0$.
			\item If $(f_n)_{n \in \mathbb{N}}$ is a sequence converging to $f \in \mathcal{F}_p$ in the Banach space $(\mathcal{F}_p, \norm{\cdot}_{\mathcal{F}_p})$, by sublinearity and homogeneity properties, for any Borel set $A$, we have 
			\begin{equation}\label{e:conv-emeas}
				\lim_{n \to \infty} \Gamma_p \langle f_n \rangle(A) =\Gamma_p \langle f \rangle(A).
			\end{equation}
		\end{enumerate}
	\end{remark}
	
	Currently, there is no standard axiomatization for local $p$-Dirichlet spaces. In fact, rather than an axiomitization, in each setting of interest a different set of crucial properties has been verified for a given energy construction. Each, however, has a restrictive assumption built into it, and the assumptions used in one setting are incompatible with the other.
	\begin{enumerate}
		\item For potential theory in metric measure spaces, one often uses an upper gradient based approach \cite{Ch,Shanmun,AGS,HKST}. An illustrative example of this is equipping $\mathbb{R}^n$  with the Lebesgue measure $\mathcal{L}_n$ and $\ell_t$-metric for $t\in [1,\infty]$, in which case the Sobolev space equals the usual Sobolev space, but the energy of Sobolev functions is given by $\Gamma_p\langle f\rangle(A)=\int_A \|\nabla f\|_{\ell_s}^p d\mathcal{L}_n$, where $s\in [1,\infty]$ is the dual exponent from $\frac{1}{s}+\frac{1}{t}=1$. This type of energy in metric measure spaces is non-trivial only if the $p$-Modulus of all curves in the space is positive, which fails to hold for many fractals of interest. We verify that Sobolev spaces defined using upper gradients are $p$-Dirichlet spaces in Lemma \ref{lem:upper-gradient=p-Dirichlet-space}.
		\item Self-similar $p$-energies on self-similar fractals can be obtained via fixed-point arguments and $\Gamma$-limits of rescaled energies \cite{Kus93, Kigami}. A good example of this is the $p$-energy constructed for a Sierpi\'nski gasket \cite{pgasket,cao}, and the Sierpi\'nski carpet \cite{Murugan-Shimizu}. In these settings, one has all of the above properties, and additionally one can obtain reflexivity, $p$-Clarkson's inequalities and generalized $p$-contraction properties, see e.g. \cite{Kajino-Shimizu, Kajino-Shimizu-penergy}. These contraction properties fail to hold in infinite dimensional metric space settings and many natural self-similar energies, such as the $\ell_t$-energies mentioned in the previous bullet point. Also, in many settings self-similarity is too restrictive, e.g. on boundaries of hyperbolic groups. 
		\item Every regular, strongly local Dirichlet form can be viewed as a $p$-Dirichlet space with $p=2$ (see Lemma \ref{lem:df-dirspace}).
	\end{enumerate}
	It may be that fewer axioms would suffice. 
	Our axioms in Definition \ref{d:dir-space} are designed to encompass diverse frameworks such as regular strongly local Dirichlet forms, Sobolev spaces and energies defined via upper gradients, and self-similar $p$-energies on fractals. The aim is to provide a unified framework within which the energy image density property can be established in these various settings, while maintaining close analogies with the theory of analysis on metric spaces. Indeed, this general framework gives a setting to resolve a generalized version of this conjecture, and to prove an extension of Cheeger's conjecture to independent collections of functions, Proposition \ref{prop:independent}. 
	
	The axioms here are obtained by essentially removing any self-similarity, uniform convexity, generalized contraction, or reflexivity assumptions from the frameworks in \cite{Kigami, Kajino-Shimizu, Murugan-Shimizu, Kajino-Shimizu-penergy}, while replacing them with the chain rule and lower-semicontinuity. These latter assumptions are satisfied in both of the frameworks above, but the first ones are not in general satisfied in the upper gradient based approaches. Indeed, here are a few examples of energies that satisfy our axioms but fail one or more of these assumptions.
	\begin{example} \label{ex:p-dir}
		\begin{enumerate} 
			\item Equip the plane $\mathbb{R}^2$ with the Lebesgue measure and $\ell_1$-metric. Then the upper gradient energy is given by $$\Gamma_p\langle f\rangle(A)=\int_A\max\{|\partial_x f|, |\partial_y f|\}^p \mathcal{L}_2.$$ This energy is self-similar but is not uniformly convex, and fails the contraction properties in \cite{Kajino-Shimizu}, but the resulting Sobolev space is reflexive. The usual $p-$Dirichlet energy on $\mathbb{R}^2$, corresponding with using the $\ell_2$ metric is uniformly convex and comparable to this energy. This corresponds to a general renorming result for finite dimensional vector spaces. This example is also a self-similar $p$-energy in the sense of \cite{Murugan-Shimizu} if restricted on $[0,1]^2$, which is an attractor of four scaling maps by a facter $1/2$. This shows that without some infinitesimal assumptions, even self-similar $p$-energies may be non-unique.
			\item Let $(X=\prod_{i=1}^\infty [0,2^{-i}], d_1, \mu=\prod_{i=1}^\infty 2^i\mathcal{L}_1)$ be the infinite product space equipped with the $\ell_1$-metric and product of re-scaled Lebesgue measures. This space is compact and infinite dimensional, and one can show that in this example, the associated upper gradient Sobolev space is not reflexive, since it corresponds to the energy $f\to \int_X \sup_{i\in \mathbb{N}} |\partial_i f| d\mu$, for $f$ Lipschitz; see the example in \cite[Proposition 44]{ACM}. As shown in \cite{ACM}, this  space is not reflexive. Therefore, it can also not be uniformly convex. Note that by results in \cite{teriseb}, an upper gradient based Sobolev space can fail to be reflexive only if the space has infinite Hausdorff dimension.  Note that the chain rule and lower-semicontinuity here follow from general results in analysis on metric spaces, see e.g. \cite[Section 2]{gigli} and \cite[Proof of Theorem 6.1]{AGS}; in comparing these and other references note that \cite{AGS} establishes the equivalence of different definitions of upper gradient based Sobolev spaces in a metric measure space.
		\end{enumerate}
	\end{example}
	The final condition of weak lower semicontinuity is named here with the reflexive case in mind. Indeed, if the space $\cF_p$ is reflexive, then one can show that a sequence $f_i\in \cF_p$ converges in the weak topology if and only if $f_i\to f$ weakly in $L_p(X)$ and $\sup_{i\in\mathbb{N}} \mathcal{E}_o(f_i)<\infty$; see e.g. \cite[Theorems 3.18 and 3.19]{Brezis}. While sometimes the structure may not be reflexive, this terminology draws also parallels to the definition of Weaver derivations in \cite{weaver}, the notion of convergence used in the density in energy result in \cite[Introduction, Definition 4.2 and Theorem 7.2]{AGS}, and the weak continuity for metric currents \cite{AKircheim}.  
	We note that reflexivity of the Banach space $(\mathcal{F}_p, \norm{\cdot}_{\mathcal{F}_p})$ is a sufficient condition for the weak lower semicontinuity of energy measure in Definition \ref{d:dir-space}. The following statement and its proof is close to \cite[Proposition 4.10]{Kajino-Shimizu} and shows that weak lower semicontinuity often follows from reflexivity.
	\begin{lemma} \label{l:suff-lsc}
		Let  $(X,d,\mu, \mathcal{E}_p, \mathcal{F}_p, \Gamma_p)$ satisfy properties (1)-(6) in Definition \ref{d:dir-space}. If the Banach space   $(\mathcal{F}_p, \norm{\cdot}_{\mathcal{F}_p})$  is reflexive, then  $(X,d,\mu, \mathcal{E}_p, \mathcal{F}_p, \Gamma_p)$ satisfies the lower semicontinuity property in Definition  \ref{d:dir-space}-(7).
	\end{lemma}
	\begin{proof}
		Let	$f\in L^p(X)$, and let $(f_i)_{i \in \mathbb{N}}$ be a sequence of functions  in $\mathcal{F}_p$ such that $f_i\to f \in L^p(X)$ with $\sup_{i\in \mathbb{N}} \mathcal{E}_p(f_i)<\infty$. Let $A$ be a Borel subset of $X$. Pick a subsequence $(f_{i_k})_{k \in \mathbb{N}}$ such that 
		\begin{equation} \label{e:lsc1}
			\lim_{k \to \infty} \Gamma_p \langle f_{i_k} \rangle(A)= \liminf_{i \to \infty} \Gamma_p \langle f \rangle (A).
		\end{equation}
		By passing to a further subsequence if necessary and by using reflexivity of $(\mathcal{F}_p, \norm{\cdot}_{\mathcal{F}_p})$, Kakutani's theorem (see \cite[Theorem 3.18]{Brezis}), we may assume that $f_{i_k}$ converges weakly in  $(\mathcal{F}_p, \norm{\cdot}_{\mathcal{F}_p})$ to $f$.
		By Mazur's lemma (see \cite[Exercise 3.4-(1)]{Brezis} or \cite[Lemma 3.14]{Kajino-Shimizu}), for each $l \in \mathbb{N}$, there exists $N(l) \in \mathbb{N}$ with $N(l)>l$, $\{\alpha_{l,k}: l \le k \le N(l)\} \subset [0,1]$ such that $\sum_{k=l}^{N(l)} \alpha_{k,l}=1$ for all $l \in \mathbb{N}$, and $g_l:= \sum_{k=l}^{N(l)} \alpha_{l,k} f_{i_k}$ converges to $f$  in the norm topology of  $(\mathcal{F}_p, \norm{\cdot}_{\mathcal{F}_p})$ as $l \to \infty$. Thus we have 
		\begin{align*}
			\Gamma_p\langle f\rangle(A)^{1/p} &\stackrel{\eqref{e:conv-emeas}}{=} \lim_{l \to \infty}	\Gamma_p\langle g_l\rangle(A)^{1/p} \\
			&\le \lim_{l \to \infty} \sum_{k=l}^{N(l)} \alpha_{k,l} \Gamma_p \langle f_{i_k} \rangle (A)^{1/p} \quad \mbox{(by sublinearity and homogeneity)} \\
			&= \lim_{k \to \infty}  \Gamma_p \langle f_{i_k} \rangle(A)^{1/p}  \quad \mbox{(by $\sum_{k=l}^{N(l)} \alpha_{k,l}=1$)}
		\end{align*} 
		which along with \eqref{e:lsc1} implies \eqref{e:def-lsc}.
	\end{proof}

	\begin{definition}
		A $p$-Dirichlet space $(X,d,\mu, \mathcal{E}_p, \mathcal{F}_p, \Gamma_p)$ is regular if we have that $C_0(X)\cap \mathcal{F}_p$ is dense in $\mathcal{F}_p$ and dense in $C_0(X)$ in the uniform norm,  where $C_0(X)$ is the space of continuous functions on $X$ with compact support.
	\end{definition}
	The following lemma can be viewed as an improved version of the chain rule in Definition \ref{d:dir-space}-(5), where the global Lipschitz constant $\LIP[g]$ is replaced with the asymptotic Lipschitz constant function $\Lip_a[g]$.  We will later relax the continuity assumption in Proposition \ref{prop:quasicont-conseq}-(2).
	\begin{lemma}\label{lem:contlocal} Let $(X,d,\mu,\mathcal{E}_p, \mathcal{F}_p, \Gamma_p)$ be a regular  $p$-Dirichlet space 
		For every $g\in \Lip(\mathbb{R})$ with $g(0)=0$ and every $f\in \mathcal{F}_p\cap C(X)$ we have  $g \circ f \in \mathcal{F}_p$ and 
		\[
		\Gamma_p\langle g\circ f\rangle  \le (\Lip_a[g] \circ f)^p\cdot \Gamma_p\langle f\rangle.
		\]
	\end{lemma}
	\begin{proof}
		Let $x_0\in X$, and let $\epsilon>0$. By chain rule, we already know $g\circ f \in \mathcal{F}_p$. Next, since $f$ is continuous, there exists a $\delta>0$ such that $f(B(x_0,\delta))\subset B(f(x_0),\epsilon)$. Let $\tilde{g}_{x_0,\delta}\in \Lip(\mathbb{R}^n)$ be equal to $g|_{f(B(x_0,\delta))}$ on $f(B(x_0,\delta))$ and which is  $\LIP[g|_{B(f(x_0),\epsilon)}]$-Lipschitz.  The existence of such a function $\tilde{g}_{x_0,\delta}\in \Lip(\mathbb{R}^n)$  follows from the  McSchane's extension theorem \cite[Theorem 6.2]{Hei}. Set $g_{x_0,\delta}=\tilde{g}_{x_0,\delta}-\tilde{g}_{x_0,\delta}(0)$. We have $g_{x_0,\delta} \circ f(x)-g \circ f(x) =c:=\tilde{g}_{x,\delta}(0)$ for $x\in B(x_0,\delta)$, and thus by \eqref{e:loc-subl}
		\begin{align*}
			\Gamma_p\langle g\circ f\rangle |_{B(x_0,\delta)} &=\Gamma_p\langle \tilde{g}_{x_0,\delta}\circ f\rangle |_{B(x,\delta)}  \\
			&\le \LIP[g|_{B(f(x_0),\epsilon)}]^p \cdot\Gamma_p\langle f \rangle|_{B(x_0,\delta)} \\
			&\le \LIP[g|_{B(f(x),2\epsilon)}]^p \cdot \Gamma_p\langle f \rangle|_{B(x_0,\delta)}.
		\end{align*}
		The space $X$ can be covered by such balls $B(x_0,\delta)$ and we get
		\[
		\Gamma_p\langle g\circ f\rangle \le \LIP[g|_{B(f(x),2\epsilon)}] \cdot \Gamma_p\langle f \rangle,
		\]
		and sending $\epsilon \to 0$, the claim follows by the dominated convergence theorem.
	\end{proof}
	
	The following is an improved version of locality for continuous functions for arbitrary Borel sets. The argument is similar to \cite[Proposition 2.22]{Ch}.
	
	\begin{lemma} \label{lem:loc-cont}
		For every $f\in \mathcal{F}_p \cap C(X)$ and any Borel set $A\subset X$, if $f|_A\equiv c$, then $\Gamma_p\langle f\rangle(A)=0$. 
	\end{lemma}
	\begin{proof}
		Define for each $k \in \mathbb{N}$, a $1$-Lipschitz functions $\widetilde{g}_k, g_k:\mathbb{R} \to \mathbb{R}$ as $\widetilde{g}_k(y)=\max(y-c-k^{-1},\min(y-c+k^{-1},0))$ $g_k(y):= \widetilde{g}_k(y)-\widetilde{g}_k(0)$, so that  
		\[
		\LIP[g_k]=1, \quad g_k(0)=0 \quad \mbox{for all $k \in \mathbb{N}$}, \quad \lim_{k \to \infty} g_k(y)= y, \quad \mbox{for all $y \in \mathbb{R}$.}
		\]
		Hence by chain rule, we have $g_k \circ f \in \mathcal{F}_p \cap C(X)$ for all $k \in \mathbb{N}$ and
		\begin{equation} \label{e:cl1}
			\sup_{k \in \mathbb{N}} \mathcal{E}_p(g_k \circ f) \le \mathcal{E}_p(f).
		\end{equation}
		By $\abs{g_k \circ f- f} \le 2 \abs{f}$, $\lim_{k \to \infty} g_k \circ f =f$, and dominated convergence theorem we have \begin{equation} \label{e:cl2}
			\lim_{k\to \infty}\norm{g_k \circ f - f}_{L^p}=0.
		\end{equation}
		Note that $g_k \circ f \equiv g_k(c)$ on the open set $f^{-1}((c-k^{-1},c+k^{-1}))$,   and hence by locality
		\begin{equation}\label{e:cl3}
			\Gamma_p \langle g_k \circ f \rangle  \left(f^{-1}((c-k^{-1},c+k^{-1}))\right)=0.
		\end{equation}
		Hence by lower semicontinuity,  and locality, we obtain
		\begin{align*}
			\Gamma_p\langle f \rangle(A) &\le \liminf_{k \to \infty} 	\Gamma_p\langle g_k \circ f \rangle(A) \quad \mbox{(by \eqref{e:cl1},\eqref{e:cl2} and \eqref{e:def-lsc})}\\
			&\le \liminf_{k \to \infty} 	\Gamma_p\langle g_k \circ f \rangle(f^{-1}(c-k^{-1},c+k^{-1})) \quad \mbox{($A \subset f^{-1}(c-k^{-1},c+k^{-1})$)}\\
			&=0 \quad \mbox{(by \eqref{e:cl3})}.
		\end{align*}
	\end{proof}
	The following estimate is a version of chain rule for vector valued continuous functions. This is a very weak version of the generalized contraction property used in \cite{Kajino-Shimizu, Kajino-Shimizu-penergy}. Unlike  these properties, we show that this weaker contraction property follows from the our axioms and does not depend on any auxiliary assumptions. We will later relax the continuity assumption in Proposition \ref{prop:quasicont-conseq}-(3). 
	\begin{lemma}\label{lem:multidimcontraction} Let $(X,d,\mu,\mathcal{E}_p, \mathcal{F}_p, \Gamma_p)$ be a local $p$-Dirichlet space.
		For every $g\in \Lip(\mathbb{R}^n)$ with $g(0)=0$ and every $f=(f_1,\dots, f_n) \in (\mathcal{F}_p \cap C_0(X))^n$, we have  $g \circ f \in \mathcal{F}_p$ and
		\begin{equation} \label{e:c-rule1}
			\Gamma_p\langle g \circ f\rangle \le \left(1 \vee n^{(p-2)/2}\right)\LIP[g]^p \cdot \sum_{i=1}^n \Gamma_p\langle f_i\rangle.
		\end{equation}
		Moreover for all $f=(f_1,\dots, f_n) \in (\mathcal{F}_p \cap C_0(X))^n$, we have
		\begin{equation} \label{e:c-rule2}
			\Gamma_p\langle g \circ f\rangle \le \left(1 \vee n^{(p-2)/2}\right) (\Lip_a[g] \circ f)^p \cdot \sum_{i=1}^n \Gamma_p\langle f_i\rangle.
		\end{equation}
	\end{lemma}
	\begin{proof}
		First, consider a linear function $g(x)=\sum_{i=1}^n \lambda_ix_i$, and note that $\|\blambda\|_2=\LIP[g]$ for $\blambda=(\lambda_i)_{i=1}^n$. In the following calculation, $\|\blambda\|_p$ is the $\ell_p$-norm on $\mathbb{R}^n$.
		By  H\"older's inequality and monotonicty of $\ell_p$-norms, we have
		\begin{equation} \label{e:crule1}
			\Gamma_p\langle \sum_{i=1}^n \lambda_i x_i \rangle(A) \le \|\blambda\|_{\frac{p}{p-1}}^p \sum_{i=1}^k \Gamma_p\langle u_i \rangle(A) \leq \left(1 \vee n^{(p-2)/2}\right) \|\lambda\|_{2}^p \sum_{i=1}^k \Gamma_p\langle u_i \rangle(A).
		\end{equation}

		Thus, \eqref{e:c-rule1} is true for linear functions. Second, we extend the claim \eqref{e:c-rule1} to piecewise linear functions.  Such a function can be expressed using minima and maxima of functions, and thus we consider these first.
		If $u,v\in \mathcal{F}_p, a \in [0,\infty)$, then $\min(u,v+a)\in \mathcal{F}_p$. Indeed, $\min(u,v+a)=\frac{u+v}{2}-\frac{(|u-v-a|-a)}{2}$ and thus $\min(u,v+a)\in \mathcal{F}_p$ by applying the chain rule to the $1$-Lipschitz function $g(t)= \abs{t-a}-a$ and sublinearity. Furthermore 
		\begin{align*}
			\Gamma_p\langle \min(u,v+a)\rangle(A)^\frac{1}{p} &\stackrel{ \eqref{e:def-sublinear},\eqref{e:def-chain}}{\le} \Gamma_p\langle \frac{u+v}{2} \rangle(A)^\frac{1}{p} + \Gamma_p\langle \frac{|u-v|}{2} \rangle(A)^\frac{1}{p}\\
			& \stackrel{\eqref{e:def-sublinear}}{\le} \Gamma_p\langle u \rangle(A)^\frac{1}{p} + \Gamma_p\langle v \rangle(A)^\frac{1}{p}.
		\end{align*}
		By induction, for any $k \in \mathbb{N}$, for all $u_1,\dots, u_k\in \mathcal{F}_p, a_1,a_2,\ldots,a_k \in [0,\infty)$ with $a_1=0$, we have
		\begin{equation} \label{e:cr0}
			\min_{1 \le i \le k}(u_i+a_i)\in \mathcal{F}_p.
		\end{equation}             
		
		Next, let $g$ be a piecewise linear function with $g(0)=0$, where a function $g$ is \emph{piecewise linear}, if there is a partition to closed sets $(A_i)_{i=1}^k$ s.t. $g|_{A_i}$ are affine linear $a_j+g_j$,  where $(g_j)_{j=1}^k$  are linear functions amd $a_i\in \mathbb{R}$. Without loss of generality assume that each $A_i$ has non-empty interior and $0\in A_1$. Any such function can be expressed as $g=\max_{i=1,\dots, N}(\min_{j\in S_i} a_j +g_j),$ where $S_i\subset \{1,\dots,k\}$ is some collection of subsets for $i=1,\dots, N$; for a proof of this fact see e.g. \cite[Theorem 4.1]{maxmin}. Let $\widetilde{g_j}=\min_{j\in S_i} a_j +g_j$. We can assume by re-indexing if necessary that $\widetilde{g_1}(0)=g(0)=0$ and $a_j\geq 0$ for all $j\in S_1$, with at least one $a_j=0$. Then, by the previous claim, we get $\widetilde{g_j}\circ f-\widetilde{g_j}(0) \in \cF_p$ for all $j=1,\dots, N$. Next
		\[
		g = \max_{i=1,\dots, N} b_i+\widetilde{g_j}-\widetilde{g_j}(0),
		\]
		with $b_i=\widetilde{g_j}(0)\leq 0$ and $b_1=0$. By negating the function $u_i$ in \eqref{e:cr0}, we get $g\circ f \in \cF_p$.
		
		Thus, by the previous claims and induction  $g(f_1,\dots, f_n)\in \cF_p$ for every piecewise linear function. Further, for every $A_i$, we have that the difference $g\circ f|_{f^{-1}(A_i)}-g_j \circ f|_{f^{-1}(A_j)}$ is a constant. The set $f^{-1}(A_j)$ is a closed set, and thus by locality in Lemma \ref{lem:loc-cont}, the argument in Remark \ref{r:local-subadditive}-(i), and \eqref{e:crule1}, we have  
		\[
		\Gamma_p\langle g\circ f\rangle|_{A_i}=\Gamma_p\langle g_i\circ f\rangle|_{f^{-1}(A_i)}\leq  \left(1 \vee n^{(p-2)/2}\right) \LIP[g_i]^p \sum_{i=1}^n \Gamma_p\langle f_i\rangle|_{f^{-1}(A_i)}.
		\]
		Since $A_i$ have non-empty interior, we have $\LIP[g_i]\leq \LIP[g]$. Thus, the claim \eqref{e:c-rule1} follows now for all piecewise linear functions $g$ since the sets $f^{-1}(A_i)$ cover the space.
		
		Next, we approximate every Lipschitz function $g$ by a sequence of piecewise linear Lipschitz functions $g_i$ with $\LIP[g_i]\leq \LIP[g]$. This can be shown in many ways, and we adapt an argument from \cite[Proposition 2.2 and Corollary 2.4]{resnetapprox}, which uses Lipschitz neural networks. Let $S=\{s_i:i\in \mathbb{N}\}$ be an enumeration of a countable dense set  in $\mathbb{R}^n$, and assume $s_1=0$. Let $S_k=\{s_1,\dots, s_k\}$ and for $k\in \mathbb{N}$
		\[
		h_k(x)=\max\left\{\frac{1}{\|s_i-s_j\|}|\langle s_i-s_j, x\rangle| : i\neq j, i,j=1,\dots, k\right\}.
		\]
		The functions $h_k$ are piecewise linear and $1$-Lipschitz and $h_k(s_i-s_j)=\|s_i-s_j\|$ for $i,j=1,\dots, k$. Let 
		\[
		g_k(x)=\min\{ g(s_j)+\LIP[g]h_k(x-s_j) : j=1,\dots, n\}.
		\]
		It is direct to see that each $g_k$ is $\LIP[g]$-Lipschitz, and piecewise linear. Further $g_k|_{S_k}=g|_{S_k}$, and $g_k(0)=0$ follow from $h_k(s_i-s_j)=\|s_i-s_j\|$ for $i,j=1,\dots, k$. From these and density of $S$ it follows that $g_k(x)\to g(x)$ for all $x\in \mathbb{R}^n$. Now,  $|g_k \circ f(x)| \leq \LIP[g_k] \sum_{i=1}^n |f_i(x)|$, and thus $g_k \circ f\in L^p(X)$ and by dominated convergence $g_k\circ f \to g\circ f$ in $L^p(X)$. Further, by \eqref{e:c-rule1} for piecewise linear functions, we obtain 
		\[
		\Gamma_p\langle g_k \circ f\rangle \leq \left(1 \vee n^{(p-2)/2}\right) \LIP[g]^p \sum_{i=1}^n d\Gamma_p\langle f_i\rangle.
		\]
		Thus, the claim follows by sending $k\to \infty$ and the weak lower-semicontinuity  \eqref{e:def-lsc}.

		The proof of \eqref{e:c-rule2} from \eqref{e:c-rule1} follows from locality and McSchane's theorem by using the same argument as in the proof of Lemma \ref{lem:contlocal}.  
	\end{proof}
	\begin{remark}
		In the case of strongly local Dirichlet forms ($p=2$) if we assume in addition that $g \in \mathcal{C}^1(\mathbb{R}^n) \cap \Lip(\mathbb{R}^n)$, the estimate \eqref{e:c-rule1} is an easy consequence of chain rule  for energy measures  with $g(0)=0$ \cite[Corollary I.6.1.3]{BH91}, \cite[Theorem 3.2.2]{FOT}. We sketch this for the case of Dirichlet structures as given in Definition \ref{d:dirstructure}. The chain rule \cite[Corollary I.6.1.3]{BH91} implies that  for any $f \in \mathcal{F}^n$ and $G \in C^1(\mathbb{R}^n) \cap\Lip(\mathbb{R}^n)$ as above with $G(0)=0$ then the associated carr\'e du champ can be expressed as
		\[
		\gamma(G(f),G(f)):= (\nabla G \circ f)^T \gamma(f) (\nabla G \circ f), \quad \mbox{$\mu$-a.e.,  }
		\]
		where $\gamma(f)$ denotes the carr\'e du champ matrix and $\nabla G$ denotes the gradient of $G$ as a column vector. It is easy to verify that (see the proof of Lemma \ref{l:invertibly-independence}) $\gamma(f)$ is a symmetric non-negative definite matrix almost everywhere. So the desired conclusion \eqref{e:c-rule2} follows from the fact that $\Lip_a[G] = \abs{\nabla G}$ for any $G \in C^1(\mathbb{R}^n) \cap\Lip(\mathbb{R}^n)$ along with the fact that maximal eigenvalue of a symmetric non-negative definite matrix is less than or equal to its trace. However, for Lipschitz functions that are not necessarily $C^1$, the estimate \eqref{e:c-rule2} in Lemma \ref{lem:multidimcontraction} seems new to the best of our knowledge even for Dirichlet forms. The estimate \eqref{e:c-rule1} for the case of Dirichlet forms follows from the same argument in \cite[Proof of Proposition I.3.3.1]{BH91}; see also \cite[Exercise I.3.5]{BH91}. Very recently, a  generalization of such contraction properties have been comprehensively studied  \cite[Theorem A.2]{Kajino-Shimizu}. 
	\end{remark}

	\subsection{Independence}
	In this subsection, we introduce the notion of \emph{$p$-independence}, which may be viewed as an infinitesimal analogue of the linear independence of energy measures.  This concept serves as a natural generalization of the non-degeneracy condition $\det(\gamma(f)) \neq 0$ appearing in the formulation of the energy image density property as we   verify in Lemmas \ref{l:invertibly-independence} and \ref{l:invertibly-independence-dirstructure}.

	We recall the notion of an lattice infimum which is also sometimes referred to as essential infimum in the literature. 
	\begin{definition} \label{d:lattice-inf}
		Let $\Lambda$ be some $\sigma$-finite Borel measure on $(X,d)$ and let $\mathcal{F}$ be any collection of measurable functions. Then
		\[
		G= \bigwedge_{f\in \mathcal{F}} f 
		\]
		is a   measurable  function   that  satisfies the following properties. 
		\begin{enumerate}
			\item ($G$ is a lower bound for $\mathcal{F}$) For all $f \in \mathcal{F}$, we have $G(x)\le f(x)$ for $\Lambda$-a.e. $x\in X$, and
			\item ($G$ is maximal among lower bounds) For any other function $\widetilde{G}(x)$ which satisfies the condition in (1),  we have $\widetilde{G}(x)\le G(x)$ for $\Lambda$-a.e. $x\in X$.
		\end{enumerate}
	\end{definition}
	The lattice infimum depends on the reference measure, and we write
	\[
	\bigwedge^\Lambda_{f\in \mathcal{F}} f,  
	\]
	to emphasize the measure $\Lambda$.  This is  crucial for us, since we will be considering spaces equipped with multiple, mutually singular measures. Moreover, it is immediate from the second property in the definition that the lattice infimum $\bigwedge^\Lambda_{f\in \mathcal{F}} f$ is uniquely defined up to $\Lambda$-null sets.

	With these, we can define a notion of independence. We denote the unit sphere in $\mathbb{R}^n$ as $\mathbb{S}^{n-1}:= \{ \blambda \in \mathbb{R}^n \vert \abs{\blambda} =1\} \subset \mathbb{R}^{n}$.
	
	\begin{definition} \label{d:p-independent}
		A map $\bphi=(\phi_1,\dots, \phi_n):X\to \mathbb{R}$ with $ \phi_i\in \mathcal{F}_p$ for all $i=1,\dots,n$ is called $p$-independent in a measurable set $A\subset X$, if 
		\[
		\bigwedge^{\Lambda_{\bphi}}_{\blambda \in \mathbb{S}^{n-1}} \frac{d\Gamma_p\langle \sum_{i=1}^n \lambda_i \phi_i\rangle}{d\Lambda_{\bphi}}>0  \ \ \ \ \ \Lambda_{\bphi}\text{-a.e. in } A
		\]
		where $\blambda = (\lambda_1,\dots,\lambda_n)\in  \mathbb{S}^{n-1} \subset \mathbb{R}^n$ and
		$\Lambda_{\bphi} =\sum_{i=1}^n \Gamma_p\langle\phi_i\rangle$.
	\end{definition}
	In this particular case, it is easier to describe the lattice infimum. Let $S$ be any countable dense subset of $\mathbb{S}^{n-1}$. Then,
	\[
	\bigwedge^{\Lambda_{\bphi}}_{|\blambda|=1} \frac{d\Gamma_p\langle \sum_{i=1}^n \lambda_i \phi_i\rangle}{d\Lambda_{\bphi}}=\inf_{\blambda \in S} \frac{d\Gamma_p\langle \sum_{i=1}^n \lambda_i \phi_i\rangle}{d\Lambda_{\bphi}}.
	\]
	This can be seen from the fact that sublinearity yields that 	for all $\blambda=(\lambda_1,\dots, \lambda_n), \blambda'=(\lambda_1,\dots,\lambda_n')$,  we have
	\begin{align*}
		\MoveEqLeft{\abs{\left(\frac{d\Gamma_p\langle \sum_{i=1}^n \lambda_i \phi_i\rangle}{d\Lambda_{\bphi}}(x)\right)^{1/p}-\left(\frac{d\Gamma_p\langle \sum_{i=1}^n \lambda'_i \phi_i\rangle}{d\Lambda_{\bphi}}(x)\right)^{1/p}}}\\ &\le \sum_{i=1}^n \abs{\lambda_i-\lambda_i'}\left(\frac{d\Gamma_p\langle  \phi_i\rangle}{d\Lambda_{\bphi}}(x)\right)^{1/p} \quad \mbox{for $\Lambda_{\bphi}$-a.e. $x\in X$.}
	\end{align*}
	
	This implies that $\blambda \to \frac{d\Gamma_p\langle \sum_{i=1}^n \lambda_i \phi_i\rangle}{d\Lambda_{\bphi}}(x)$ is Lipschitz continuous on $S$ for $\Lambda_{\bphi}$-a.e. $x\in X$. This continuity allows one to verify the properties of an lattice infimum.

	Definition \ref{d:p-independent} has appeared in various guises in many  works: \cite{Ch, bate2024fragment,teriseb,murugananalytic,hinomartingaledim,hinoenergymeasures}. In different forms (via essentially a dual construction involving Weaver derivations), it appears in \cite{gigli, schioppaalberti,schioppaweaver}. Some similar ideas can also be seen in the early work of \cite{weaver}. In the case of Dirichlet forms, the notion of $p$-independence is essentially equivalent to the invertibility of the carr\'e du champ matrix (see Lemmas \ref{l:invertibly-independence} and \ref{l:invertibly-independence-dirstructure}).
	
	We record an elementary lemma for future use.
	\begin{lemma}\label{lem:essinf}
		A map $\bphi=(\phi_1,\dots, \phi_n):X\to \mathbb{R}$ with $ \phi_i\in \mathcal{F}_p$ for all $i=1,\dots,n$ is $p$-independent in $A\subset X$ if and only if there exist measurable sets $A_n$ such that
		\begin{enumerate}
			\item $\Lambda_{\bphi}(A\setminus \bigcup_{n=1}^\infty A_i)=0$, 
			\item For each $i \in \mathbb{N}$, we have
			\[
			\bigwedge^{\one_{A_i}\cdot\Lambda_{\bphi}}_{\blambda \in \mathbb{S}^{n-1}}	\frac{d\Gamma_p\langle \sum_{j=1}^n \lambda_j \phi_j\rangle}{d\Lambda_{\bphi}}\geq \frac{1}{i},
			\]
			where $\one_{A_i} \cdot \Lambda_{\bphi}$ is the measure given by the restriction of $\Lambda_{\bphi}$ to $A_i$.
		\end{enumerate}
	\end{lemma}
	\begin{proof}
		This is quite direct from the definition. Indeed, let $G:X \to \mathbb{R}$ be the lattice infimum in the definition of independence and let $A_i=G^{-1}(1/i,\infty)$. Then, the lemma follows by the definition of an lattice infimum.
	\end{proof}

	\subsection{Capacities and quasicontinuous functions} \label{ss:cap-qcts}
	It is well known that, in the theory of Dirichlet forms and Sobolev spaces, every function admits a quasicontinuous representative possessing enhanced continuity properties \cite[Theorem 2.1.2]{FOT}, \cite[Theorem 4.19]{EvGa}.
	We show that an analogous statement holds for $p$-Dirichlet spaces and use these quasicontinuous representatives to formulate the energy image density property in this framework (Theorem~\ref{thm:eid-dirspace}).

	For a local $p$-Dirichlet space $(X,d,\mu,\mathcal{E}_p, \mathcal{F}_p, \Gamma_p)$, and a subset $A \subset X$, we define the \emph{capacity} of $A$ as
	\[
	\cCap(A)= \inf\{ \norm{f}_{L^p}^p + \mathcal{E}_p(f) : f \in \mathcal{F}_p, f=1 \text{ in a neighborhood of } A\},
	\]
	with the usual convention that $\inf \emptyset = \infty$.
	
	It is immediate from the definition that $\cCap(A) \le \cCap(B)$ whenever $A \subset B$. We also need the following countable sub-additivity of capacities. For Dirichlet forms, stronger countable sub-additivity properties are known but the following version is sufficient for our purposes \cite[Theorem 2.1.1]{FOT}.
	\begin{lemma} \label{l:cap}
		Let $(X,d,\mu, \mathcal{E}_p,\mathcal{F}_p, \Gamma_p)$ be a $p$-Dirichlet space, where $p \in (1,\infty)$. Then  for any sequence of subsets $(A_n)_{n \in \mathbb{N}}$ of $X$, we have 
		\begin{equation} \label{e:subad}
			\cCap \left(\bigcup_{n=1}^\infty A_n \right)^{1/p}\le 2^{(p-1)/p} \sum_{n=1}^\infty \cCap(A_n)^{1/p}. 
		\end{equation}
	\end{lemma}
	\begin{proof}
		Without loss of generality, we may assume that $\sum_{n=1}^\infty \cCap(A_n)^{1/p} <\infty$.
		Let $\epsilon>0$ be arbitrary. For each $k \in \mathbb{N}$, let $f_k \in  \mathcal{F}_p$ be such that $f_k \equiv 1$ $\mu$-a.e.~on $O_k$, where $O_k$ is an open set containing $A_k$ and 
		\begin{equation} \label{e:subad1}
			\left(\norm{f_k}_{L^p}^p +	\mathcal{E}_p(f_k)^p\right)^{1/p}  \le  \cCap(A_k)^{1/p} + 2^{-k}\epsilon.
		\end{equation}
		By replacing $f_k$ with $\abs{f_k}$ and using chain rule \eqref{e:def-chain}, we may assume that $f_k \ge 0$ for all $k \in \mathbb{N}$.
		By \eqref{e:cr0}, for any $n \in \mathbb{N}$, we have $\max_{1 \le i \le n} f_i \in \mathcal{F}_p$ and 
		\begin{equation} \label{e:subad2}
			\mathcal{E}_p(\max_{1 \le i \le n} f_i )^{1/p} \le \sum_{i=1}^n \mathcal{E}_p(f_i)^{1/p} \le    \sum_{k=1}^\infty \left(\cCap(A_n)^{1/p} + 2^{-k}\epsilon\right) < \infty. 
		\end{equation}
		Since 
		\begin{equation} \label{e:subad3}
			\norm{ \max_{k \in \mathbb{N}} f_k}_{L^p} \le \norm{\sum_{k \in \mathbb{N}} f_k}_{L^p} \le
			\sum_{k=1}^\infty \norm{f_k}_{L^p} \le \sum_{k=1}^\infty \left(\cCap(A_n)^{1/p} + 2^{-k}\epsilon\right) < \infty,
		\end{equation}
		$\max_{k \in \mathbb{N}} f_k \in L^p$. We note that $\max_{1 \le k \le n} f_k$ converges in $L^p$ to $f:=\max_{k \in \mathbb{N}} f_k$ since
		\[
		\norm{f-\max_{1 \le k \le n} f_k}_{L^p} \le \norm{\max_{k \ge n+1} f_k}_{L^p} \le \sum_{k=n+1}^\infty \norm{f_k}_{L^p} \xrightarrow{n \to \infty}0.
		\]
		Thus by  \eqref{e:subad2}  and \eqref{e:def-lsc}, we deduce that $f=\max_{k \in \mathbb{N}} f_k  \in \mathcal{F}_p$ and 
		\begin{equation} \label{e:subad4}
			\mathcal{E}_p(f)^{1/p} \stackrel{\eqref{e:def-lsc}, \eqref{e:subad2}}{\le} \liminf_{n \to \infty} \mathcal{E}_p(\max_{1 \le k \le n} f_k)^{1/p} \stackrel{\eqref{e:subad2}}{\le}  \sum_{k=1}^\infty 	\mathcal{E}_p(f_k)^{1/p}.
		\end{equation}
		Note that $f=\max_{k \in \mathbb{N}} f_k  \in \mathcal{F}_p$ satisfies $f=1$ $\mu$-a.e.~on the neighborhood $\cup_{k \in \mathbb{N}} O_k$ of  $\cup_{k \in \mathbb{N}} A_k$ and hence
		\begin{align*}
			\cCap \left(\bigcup_{k \in \mathbb{N}}A_k\right)^{1/p} &\le 	\left(\norm{f}_{L^p}^p +	\mathcal{E}_p(f)\right)^{1/p} \le \norm{f}_{L^p} +	\mathcal{E}_p(f)^{1/p}\\
			&\le \sum_{k=1}^\infty \left( \norm{f_k}_{L^p} +	\mathcal{E}_p(f_k)^{1/p} \right) \quad \mbox{(by \eqref{e:subad3}, \eqref{e:subad4})} \\
			&\le 2^{(p-1)/p} \sum_{k=1}^\infty \left( \norm{f_k}_{L^p}^p +	\mathcal{E}_p(f_k)  \right)^{1/p} \\
			&\le 2^{(p-1)/p} \sum_{k=1}^\infty \left(\cCap(A_k)^{1/p} + 2^{-k}\epsilon\right).
		\end{align*}
		Letting $\epsilon \downarrow 0$, we obtain \eqref{e:subad}.
	\end{proof}
	
	Together with the notion of $p$-independence introduced in Definition \ref{d:p-independent}, we introduce the following notion of quasicontinuous functions, which will be needed for the formulation of the energy image density property for regular $p$-Dirichlet spaces.
	\begin{definition}
		A  (pointwise defined) function $f:X \to \mathbb{R}$  is called \emph{quasicontinuous}, if for every $\epsilon>0$ there exists an open set $O\subset X$ with $\cCap(O)<\epsilon$ and $f|_{X\setminus O}$ is continuous.  
	\end{definition}
	We prove the existence and uniqueness of quasicontinuous representatives for functions in the $p$-Dirichlet space, extending the corresponding result known for regular Dirichlet forms \cite[Theorem 2.1.3, Lemma 2.1.4]{FOT}.
	\begin{proposition} \label{prop:quasicont}
		If $(X,d,\mu,\mathcal{E}_p,\mathcal{F}_p , \Gamma_p)$ is a regular local $p$-Dirichlet space, then the following hold
		\begin{enumerate}
			\item \textbf{Existence:} Every $f\in \mathcal{F}_p$ has a quasicontinuous representative: there exists a function $\tilde{f}:X \to \mathbb{R}$ which is quasicontinuous and $f=\tilde{f}$ a.e. Moreover, $\tilde{f}:X \to \mathbb{R}$ can be chosen such that there exists a sequence $(f_n)_{n \in \mathbb{N}}$ of functions in $\mathcal{F}_p \cap C_0(X)$ and a decreasing sequence of open sets $(B_n)_{n \in \mathbb{N}}$ such that 
			\[
			\cCap(B_n)^{1/p} \le 2^{-n}, 
			\]
			and
			\[\abs{f_k(x)-\tilde{f}(x)} \le 2^{1-n} \quad \mbox{for all $k \ge n$ and $x \in B_n^c$.} \]
			\item \textbf{Uniqueness:} If $f,g\in \mathcal{F}_p$ are quasicontinuous and $f=g$ a.e., then $\cCap(\{f\neq g\})=0$.
		\end{enumerate}
	\end{proposition}
	
	\begin{proof}
		\begin{enumerate}
			\item The first follows by a standard argument, e.g. \cite[Theorem 2.1.3]{FOT} or \cite[Theorem 3.7]{Shanmun}. Since  $C_0(X) \cap \mathcal{F}_p$ is dense in the Banach space $(\mathcal{F}_p,\norm{\cdot}_{\mathcal{F}_p})$, we can find a sequence $f_i\in \mathcal{F}_p\cap C_0(X)$ with $f_i\to f$ in $\mathcal{F}_p$ and $\norm{f_i-f_j}_{\mathcal{F}_p}\le 4^{-\min(i,j)-2}$. Let $A_{i}=\{x : \abs{f_i-f_{i+1}}> 2^{-i-1}\}$, so that 
			\[
			\cCap(A_i)^{1/p} \le 2^{i+1} \norm{f_i-f_{i+1}}_{\mathcal{F}_p} \le 2^{-i-3}.
			\]
			Let $B_N=\bigcup_{i=N}^{\infty} A_i$. The set $B_N$ is open since $A_i$ is open for all $i\in \N$. By  Lemma \ref{l:cap}, we  have $$\cCap(B_N)^{1/p}\le 2^{(p-1)/p} \sum_{n=N}^\infty \cCap(A_n)^{1/p} \le 2 \sum_{n=N}^\infty 2^{-n-3}\le 2^{-N}.$$ For $x\not\in B_N$ we have $|f_i-f_{i+1}|\le 2^{-i-1}$ for all $i\geq N$, and thus $f_i(x)$ converges uniformly outside $B_N$ and we have 
			\[\abs{f_n(x)-\lim_{n\to \infty}f_n(x)} \le \sum_{k=N}^\infty \abs{f_k(x)-f_{k+1}(x)}   \le 2^{-N}, \quad \mbox{for all $x \in B_N^c$ and $n \ge N$.}\]
			Let $\tilde{f}(x)$ be this pointwise limit, which is defined outside $\bigcap_{N=1}^\infty B_N$. Clearly, $\mu(\bigcap_{N=1}^\infty B_N)=0$, and thus $\tilde{f}(x)$ is a representative of $f$. Since $f_i$ are continuous, then $\tilde{f}$ is also continuous outside $B_N$ by uniform convergence outside of this set. Thus, $\tilde{f}$ is a quasicontinuous representative of $f$.
			\item 	It suffices to consider the case $g=0$ by replacing $f$ with $f-g$. Fix $\delta>0$. Let $A=\{x: \abs{f(x)}>\delta\}$. We will show that $\cCap(A)=0$, from which the claim follows by exhausting the set where $f\neq 0$ by sets of this form and by using Lemma \ref{l:cap}. Fix $\epsilon>0$ and let $O_\epsilon$ be the open set such that $f|_{X\setminus O_\epsilon}$ is continuous and $\cCap(O_\epsilon)<\epsilon$. Fix an open set $U=\{x\in X\setminus O_\epsilon: |f(x)|>\delta\}$ which is relatively open in $X\setminus O_\epsilon$. Thus, $V=U\cup O_\epsilon$ is open. Let $h_\epsilon$ be the function such that $h_\epsilon \ge  1$ on $O_\epsilon$ and with $\norm{h_\epsilon}_{\mathcal{F}_p}\le \epsilon$. Define $\tilde{h}_\epsilon = \max(h_\epsilon, |f|/\delta)= 2^{-1}(h_\epsilon+|f|/\delta - \abs{h_\epsilon- |f|/\delta})$. By an argument similar to the proof of \eqref{e:cr0}, we obtain $\|\tilde{h}_\epsilon\|_{\mathcal{F}_p}\le \epsilon$ and $\tilde{h}_\epsilon \geq 1$ for all $x\in V$. Since $V$ is open and $V \supset A$, we conclude $\cCap(A)\le \epsilon$ for every $\epsilon>0$, and the claim follows. \qedhere
		\end{enumerate}	
	\end{proof}
	Parts (1) and (2) of the following result are well-known for regular Dirichlet forms. We refer to \cite[Lemma 3.2.4]{FOT}, \cite[Corollary I.7.1.2]{BH91}    respectively    for Dirichlet forms.  Part (3) is new to the best of our knowledge and will play an important role in the proof of the energy image density property.
	\begin{proposition} \label{prop:quasicont-conseq}
		Let  $(X,d,\mu, \mathcal{E}_p, \mathcal{F}_p, \Gamma_p)$ be a regular $p$-Dirichlet space.
		If $f\in \mathcal{F}_p$ is quasicontinuous and let $A\subset X$ be Borel measurable, then the following hold: 
		\begin{enumerate}
			\item If $\cCap(A)=0$, then $\Gamma_p\langle f\rangle(A)=0$.
			\item If $g \in \Lip(\mathbb{R})$ with $g(0)=0$, then $g \circ f \in \mathcal{F}_p$ and 
			\[
			\Gamma_p\langle g\circ f\rangle  \le( \Lip_a[g] \circ f)^p \cdot \Gamma_p\langle f\rangle.
			\]
			\item If $g\in \Lip(\mathbb{R}^n)$ with $g(0)=0$, and $f=(f_1,\dots, f_n)$ with $f_i\in \mathcal{F}_p$ and $f_i$ are quasicontinuous, then
			\begin{equation} \label{e:multi-chainrule}
				\Gamma_p\langle g\circ f\rangle  \le \left(1 \vee n^{(p-2)/2}\right) (\Lip_a[g] \circ f)^p\cdot \sum_{i=1}^n \Gamma_p\langle f_i\rangle.
			\end{equation}
		\end{enumerate}
	\end{proposition}
	
	\begin{proof}
		\begin{enumerate}
			\item 	By density and sub-additivity, for the first claim it suffices to assume that $f \in \mathcal{F}_p \cap C_0(X)$. By chain rule $f_+:= \max(f,0), f_-=\max(-f,0)$ belong to $\mathcal{F}_p$. Since $f=f_+-f_-$, by \eqref{e:def-sublinear} and homogeneity, it suffices to consider the case  $f \in \mathcal{F}_p \cap C_0(X)$ with $f \ge 0$. 		Thus there exists  $M \in (0,\infty)$ such that $0 \le f=\abs{f} \le M$.

			Let $\epsilon>0$ be arbitrary. Choose an open set $O_{k}$ such that $A\subset O_k$ and a non-negative function $h_k\in \mathcal{F}_p$ such that $h_k\geq 1$ $\mu$-a.e.~on $O_k$ and $\norm{h_k}_{\mathcal{F}_p}\le 2^{-k}/M.$ Define
			\[
			f_k = \min(f+Mh_k,M).
			\]
			By the chain rule and sub-linearity, we have 
			\begin{equation}\label{eq:energyestfk}
				\Gamma_p\langle f_k \rangle(E) \stackrel{\eqref{e:def-chain}}{\le} \Gamma_p\langle f+Mh_k \rangle(E) \stackrel{\eqref{e:def-sublinear}}{\le} (\Gamma_p\langle f \rangle(E)^\frac{1}{p}+2M\Gamma_p\langle h_k \rangle(E)^\frac{1}{p})^{p}
			\end{equation}
			for all Borel sets $E\subset X$. Moreover,
			\[
			\abs{f-f_k}= f_k-f\le Mh_k, \quad \mbox{$\mu$-almost everywhere,}
			\]
			and thus
			\[
			\norm{f-f_k}_{L^p} \le M 	\norm{h_k}_{L^p} \le M 	\norm{h_k}_{\mathcal{F}_p}  \le 2^{-k},
			\]
			$f_{k}\to f$ in $L^p(X)$. 
			Since $0 \le f \le M$ and $h_k \ge 1$ $\mu$-a.e.~on $O_k$, we have $f_k \equiv M$ on $O_k$ and hence by locality we have 
			\begin{equation} \label{e:loc-fk}
				\Gamma_p \langle f_k \rangle(O_k) =0, \quad \mbox{for all $k \in \mathbb{N}$.}
			\end{equation}
			Since  $f_{k}\to f$ in $L^p(X)$, we obtain
			\begin{align*}
				\Gamma_p\langle f\rangle(A)
				&\le \liminf_{k\to\infty} \Gamma_p\langle f_k\rangle(A) \quad \mbox{(by \eqref{e:def-lsc} and $\norm{f-f_k}_{L^p} \to 0$)} \\
				&\le \liminf_{k\to\infty} \Gamma_p\langle f_k\rangle(O_k)  \stackrel{\eqref{e:loc-fk}}{=} 0. 
			\end{align*}
			\item	Let $\epsilon>0$ be arbitrary. It suffices to prove the bound 
			\begin{equation} \label{e:suff-bnd}
				\Gamma_p\langle g\circ f\rangle  \le \left( \LIP(\restr{g}{B(f(x),\epsilon)}) \right)^p \, \Gamma_p\langle f\rangle 
			\end{equation}
			and then let $\epsilon \to 0$ and use dominated convergence theorem.  By  part (1), Proposition \ref{prop:quasicont}, and by modifying $f$ on  a set of capacity zero
			we may   that there exists a sequence
			$(f_n)_{n \in \mathbb{N}}\in C(X)\cap \mathcal{F}_p$ and a decreasing sequence of  open sets $(B_n)_{n \in \mathbb{N}}$ such that $f_i$ converge pointwise to $f$ outside $\bigcap_{k=1}^\infty B_k$,   $\mathcal{E}_p(f_k-f)\to 0$ as $k\to\infty$, $\cCap(B_n) \le 2^{-n}$
			and
			\begin{equation} \label{e:fk-conv}
				\abs{f(x)-f_k(x)} \le 2^{1-n}, \quad \mbox{for all $n \in \mathbb{B}, k,n \in \mathbb{N}$ such that $k \ge n$.}
			\end{equation}
			Note that 
			\begin{equation} \label{e:lip-weak}
				\norm{g \circ f_i - g \circ f}_{L^p} \le \LIP(g) \norm{f_i-f}_{L^p} \xrightarrow{i \to \infty} 0, \quad \sup_{i \in \mathbb{N}} \mathcal{E}_p(g \circ f_i) \stackrel{\eqref{e:def-chain}}{\le} \LIP(g)^p \sup_{i \in \mathbb{N}} \mathcal{E}_p(f_i) < \infty.
			\end{equation}
			Thus for any $n \in \mathbb{N}$ such that $2^{1-n}<\epsilon$ and any Borel set $A \subset X$, we have
			\begin{align} \label{e:qc-crule1}
				\Gamma_p\langle g\circ f\rangle(A)  &\le \liminf_{k\to\infty} \Gamma_p\langle g\circ f_k\rangle(A) \quad \mbox{(by \eqref{e:lip-weak} and \eqref{e:def-lsc})} \nonumber \\
				&\le \liminf_{k\to\infty} \int_A \Lip_ag(f_k(x))^p \, \Gamma_p\langle f_k\rangle (dx)\quad \mbox{(by Lemma \ref{lem:contlocal})} \nonumber \\
				&\stackrel{\eqref{e:fk-conv}}{\le} \liminf_{k\to\infty} \int_{(A\setminus B_n)} \left( \LIP(\restr{g}{B(f(x),\epsilon)}) \right)^p \,\Gamma_p\langle f_k\rangle(dx)  + \LIP[g]^p \Gamma_p\langle f_k\rangle(B_n). 
			\end{align}
			By sublinearity, for any $k \in \mathbb{N}$  we have 
			\begin{equation} \label{e:qc-crule2}
				\sup_{A \subset X}\abs{\Gamma_p\langle f \rangle(A)^{1/p}- \Gamma_p\langle f_k \rangle(A)^{1/p} } \stackrel{\eqref{e:def-sublinear} }{\le} \mathcal{E}_p(f-f_k)^{1/p} \xrightarrow{k \to \infty} 0.
			\end{equation}
			By \eqref{e:qc-crule2}, we have 
			\begin{equation} \label{e:qc-crule3}
				\lim_{k \to \infty} \Gamma_p \langle f_k \rangle (B_n) =  \Gamma_p \langle f \rangle (B_n).
			\end{equation}
			Since the measures $\Gamma_p \langle f_k \rangle, k \in \mathbb{N}$ have uniformly bounded mass,  the   estimate \eqref{e:qc-crule2} implies that $\abs{\Gamma_p\langle f \rangle- \Gamma_p\langle f_k \rangle} \xrightarrow{k \to \infty}0$, where $\abs{\nu}$ denotes the total variation measure of a signed measure $\nu$. Hence for any Borel set $A \subset X$, we have
			\begin{align} \label{e:qc-crule4}
				\MoveEqLeft{\abs{\int_{A} \LIP[g|_{B(f(x),\epsilon)}]^p\,\Gamma_p\langle f_k\rangle(dx)-\int_{A} \LIP[g|_{B(f(x),\epsilon)}]^p\,\Gamma_p\langle f\rangle(dx)}} \nonumber \\
				&\le \int_{X}\LIP[g]^p\,\abs{\Gamma_p\langle f \rangle-\Gamma_p\langle f_k \rangle}   (dx) \xrightarrow{k \to \infty} 0.
			\end{align}
			Hence by \eqref{e:qc-crule1}, \eqref{e:qc-crule3} and \eqref{e:qc-crule4},  $n \in \mathbb{N}$ such that $2^{1-n}<\epsilon$ and any Borel set $A \subset X$, we have
			\[
			\Gamma_p\langle g\circ f\rangle(A) \le \LIP(g)^p \Gamma_p \langle f \rangle (B_n) +  \int_{A}  \left( \LIP(\restr{g}{B(f(x),\epsilon)}) \right)^p \,\Gamma_p\langle f\rangle(dx)
			\]
			Sending $n \to \infty$ and $\epsilon \to 0$ yields the claim by the first part and since $\Gamma_p\langle f\rangle$ is a Borel regular finite measure.
			\item  The proof is identical to that of part (2), by approximation using a sequence of functions in $(\mathcal{F}_p \cap C_0(X))^n$ except that we use Lemma \ref{lem:multidimcontraction} instead of Lemma \ref{lem:contlocal}. \qedhere
		\end{enumerate}
	\end{proof}

	\subsection{Regular Dirichlet forms} \label{ss:reg-df}
	We recall the notion of regular Dirichlet form and explain that it is a   $2$-Dirichlet space in the sense of Definition \ref{d:dir-space}. We also study the relationship between the notion of $p$-independence and the invertibility of a Malliavin-type matrix.

	Throughout this subsection (\textsection \ref{ss:reg-df}), we assume that $(X,d)$ is a complete and locally compact separable metric space equipped with a Radon measure $\mu$ with full support.
	Let $(\mathcal{E},\mathcal{F})$ be a \emph{symmetric Dirichlet form} on $L^{2}(X,\mu)$;
	that is, $\mathcal{F}$ is a dense linear subspace of $L^{2}(X,\mu)$, and
	$\mathcal{E}:\mathcal{F}\times\mathcal{F}\to\mathbb{R}$
	is a non-negative definite symmetric bilinear form which is \emph{closed}
	($\mathcal{F}$ is a Hilbert space under the inner product $\mathcal{E}_{1}:= \mathcal{E}+ \langle \cdot,\cdot \rangle_{L^{2}(\mu)}$)
	and \emph{Markovian} ($f^{+}\wedge 1\in\mathcal{F}$ and $\mathcal{E}(f^{+}\wedge 1,f^{+}\wedge 1)\le \mathcal{E}(f,f)$ for any $f\in\mathcal{F}$).
	Recall that $(\mathcal{E},\mathcal{F})$ is called \emph{regular} if
	$\mathcal{F}\cap C_{\mathrm{c}}(X)$ is dense both in $(\mathcal{F},\mathcal{E}_{1})$
	and in $(C_{\mathrm{c}}(X),\|\cdot\|_{\mathrm{sup}})$, and that
	$(\mathcal{E},\mathcal{F})$ is called \emph{strongly local} if $\mathcal{E}(f,g)=0$
	for any $f,g\in\mathcal{F}$ with $\supp_\mu[f]$, $\supp_\mu[g]$ compact and
	$\supp_\mu[f-a\one_{X}]\cap\supp_\mu[g]=\emptyset$ for some $a\in\mathbb{R}$. Here
	$C_{\mathrm{c}}(X)$ denotes the space of $\mathbb{R}$-valued continuous functions on $X$ with compact support, and
	for a Borel measurable function $f:X\to[-\infty,\infty]$ or an
	$\mu$-equivalence class $f$ of such functions, $\supp_\mu[f]$ denotes the support of the measure $|f|\,d\mu$,
	i.e., the smallest closed subset $F$ of $X$ with $\int_{X\setminus F}|f|\,d\mu=0$,
	which exists since $X$ has a countable open base for its topology; note that
	$\supp_\mu[f]$ coincides with the closure of $X\setminus f^{-1}(0)$ in $X$ if $f$ is continuous.
	The pair $(X,d,\mu,\mathcal{E},\mathcal{F})$ of a metric measure space $(X,d,\mu)$ and a strongly local,
	regular symmetric Dirichlet form $(\mathcal{E},\mathcal{F})$ on $L^{2}(X,\mu)$ is termed
	a \emph{metric measure Dirichlet space}, or an \emph{MMD space} in abbreviation.  
	We refer to \cite{FOT} for details of the theory of symmetric Dirichlet forms.
	
	We recall the definition of energy measure.
	Note that $fg\in\mathcal{F}$
	for any $f,g\in\mathcal{F}\cap L^{\infty}(X,\mu)$ by \cite[Theorem 1.4.2-(ii)]{FOT}
	and that $\{(-n)\vee(f\wedge n)\}_{n=1}^{\infty}\subset\mathcal{F}$ and
	$\lim_{n\to\infty}(-n)\vee(f\wedge n)=f$ in norm in $(\mathcal{F},\mathcal{E}_{1})$
	by \cite[Theorem 1.4.2-(iii)]{FOT}.
	
	\begin{definition}\label{d:EnergyMeas}
		Let $(X,d,\mu,\mathcal{E},\mathcal{F})$ be an MMD space.
		The \emph{energy measure} $\Gamma(f,f)$ of $f\in\mathcal{F}$
		associated with $(X,d,\mu,\mathcal{E},\mathcal{F})$ is defined,
		first for $f\in\mathcal{F}\cap L^{\infty}(X,\mu)$ as the unique ($[0,\infty]$-valued)
		Borel measure on $X$ such that
		\begin{equation}\label{e:EnergyMeas}
			\int_{X} g \, d\Gamma(f,f)= \mathcal{E}(f,fg)-\frac{1}{2}\mathcal{E}(f^{2},g) \qquad \textrm{for all $g \in \mathcal{F}\cap C_{\mathrm{c}}(X)$,}
		\end{equation}
		and then by
		$\Gamma(f,f)(A):=\lim_{n\to\infty}\Gamma\bigl((-n)\vee(f\wedge n),(-n)\vee(f\wedge n)\bigr)(A)$
		for each Borel subset $A$ of $X$ for general $f\in\mathcal{F}$. The signed measure $\Gamma(f,g)$ for $f,g \in \mathcal{F}$ is defined by polarization.
	\end{definition}
	We verify that every MMD space can be viewed as a regular, local, $p$-Dirichlet space in the sense of Definition \ref{d:dir-space} with $p=2$.
	\begin{lemma} \label{lem:df-dirspace}
		Let $(X,d,\mu,\mathcal{E},\mathcal{F})$ be an MMD space. Let $\mathcal{F}_2=\mathcal{F}$ and $\mathcal{E}_2: \mathcal{F}_2 \to [0,\infty)$ be defined by $\mathcal{E}_2(f):= \mathcal{E}(f,f)$ for all $f \in \mathcal{F}_2$. Let $\Gamma_2: \mathcal{F}_2 \to \mathcal{M}(X)$ denote the energy measure $\Gamma_2(f):=\Gamma(f,f)$ as given in Definition \ref{d:EnergyMeas}. Then  $(X,d,\mu,\mathcal{E}_2, \mathcal{F}_2, \Gamma_2)$ is a regular, local $2$-Dirichlet space in the sense of Definition \ref{d:dir-space}.
	\end{lemma}
	\begin{proof}
		The first three properties in Definition \ref{d:dir-space}  immediately follow from the definitions. The sublinearity property follows from Cauchy-Schwarz inequality for energy measures \cite[(2.2)]{hinoenergymeasures}. 
		
		In order to show chain rule, consider $g \in \Lip(\mathbb{R})$ with $g(0)=0$. By homogeneity property, we may assume that $\LIP[g]= 1$. In this case,  the chain rule in Definition \ref{d:dir-space}   follows from the definition of energy measure,  and \cite[Proposition I.4.1.1, Theorem I.3.3.3]{BH91}.
		
		The locality property in  Definition \ref{d:dir-space}  follows immediately from \cite[Corollary 3.2.1]{FOT}. Finally, the lower semicontinuity property follows from Lemma \ref{l:suff-lsc} as every Hilbert space is reflexive.
	\end{proof}
	
	Unlike the case of Dirichlet structures, the energy measure need not be absolutely continuous with respect to the reference measure $\mu$. In fact, for Dirichlet forms corresponding to  diffusion processes on many fractals, the energy measures are singular to the reference measure \cite{Kus89, Kus93, Hino-sing, KM20}. Therefore, in order to define a suitable analogue of the carr\'e du champ matrix in \eqref{e:def-cdc-matrix}, we recall the notion of a minimal energy dominant measure. 
	\begin{definition}[{\cite[Definition 2.1]{hinoenergymeasures}}]\label{d:minimal-energy-dominant}
		Let $(X,d,\mu,\mathcal{E},\mathcal{F})$ be an MMD space and let $\Gamma(\cdot,\cdot)$ denote the corresponding energy measure. A $\sigma$-finite Borel measure
		$\nu$ on $X$ is called a \textbf{minimal energy-dominant measure}
		of $(\mathcal{E},\mathcal{F})$ if the following two conditions are satisfied:
		\begin{enumerate} 
			\item[(i)](Domination) For every $f \in \mathcal{F}$, we have $\Gamma(f,f) \ll \nu$.
			\item[(ii)](Minimality) If another $\sigma$-finite Borel measure $\nu'$
			on $X$ satisfies condition (i) with $\nu$ replaced
			by $\nu'$, then $\nu \ll \nu'$.
		\end{enumerate}
		The existence of minimal energy dominant measure is due to Nakao \cite[Lemma 2.2]{Nak} (see also \cite[Lemma 2.3 and Proposition 2.7]{hinoenergymeasures}).  It is immediate that any two minimal energy dominant measures are mutually absolutely continuous.
	\end{definition}

	Given a minimal energy dominant measure $\nu$ on an MMD space $(X,d,\mu,\mathcal{E},\mathcal{F})$ and $\bphi=(\phi_1,\ldots,\phi_n) \in \mathcal{F}^n$, we define the carr\'e du champ matrix $\gamma_\nu(\bphi)$ with respect to $\nu$ as the $n \times n$ matrix given by
	\begin{equation} \label{e:def-cdc-reg}
		\gamma_\nu(\bphi)(x):= \begin{bmatrix}
			\frac{d\Gamma(\phi_i,\phi_j)}{d\nu}(x)
		\end{bmatrix}_{1 \le i,j \le n}, \quad \mbox{for all $x \in X$.}
	\end{equation}

	We show that the notion of $p$-linearly independence (cf. Definition \ref{d:p-independent}) is equivalent to the invertibility of the carr\'e du champ matrix. This allows us to show the equivalence between the classical formulation on energy image density property and our formulation using $p$-independence.
	\begin{lemma} \label{l:invertibly-independence}
		Let $(X,d,\mu,\mathcal{E},\mathcal{F})$ be an MMD space, also viewed as a   regular, local $2$-Dirichlet space $(X,d,\mu,\mathcal{E}_2, \mathcal{F}_2, \Gamma_2)$ as given in Lemma \ref{lem:df-dirspace}. Let $n \in \mathbb{N}$, and   $\bphi=(\phi_1,\ldots,\phi_n) \in \mathcal{F}^n$. Let $\nu$ be a minimal energy dominant measure for $(X,d,\mu,\mathcal{E},\mathcal{F})$ and let $A$ be a Borel subset of $X$. Then the following are equivalent:
		\begin{enumerate}[(a)]
			\item $\det(\gamma_\nu(\bphi))>0$ for $\one_{\{\Tr (\gamma_\nu(\bphi))>0\}}\cdot \nu$-a.e.~$A$.
			\item $\bphi$ is $2$-linearly independent on $A$.
		\end{enumerate}
		Furthermore, $\bphi$ is $2$-linearly independent on the set $I_{\bphi}:=\{x \in X \vert \det(\gamma_\nu(\bphi)(x))>0\}$ and the measures
		$\one_{I_{\bphi}} \cdot \nu$ and $\one_{I_{\bphi}} \cdot \Lambda_{\bphi}$ are mutually absolutely continuous.
	\end{lemma}
	\begin{proof}
		Note that for any $\blambda=(\lambda_1,\ldots,\lambda_n) \in \mathbb{R}^n$, by the bilinearity of energy measure, we have 
		\begin{equation} \label{e:emeas-bilinear}
			0 \le \frac{d\Gamma(\sum_{i=1}^n \lambda_i \phi_i, \sum_{i=1}^n \lambda_i \phi_i)}{d \nu}= \blambda^T \gamma_\nu(f) \blambda, \quad \mbox{$\nu$-a.e.,}
		\end{equation}
		where $\blambda$ is regarded as a column vector. This implies that 
		\begin{equation} \label{e:sn-valued}
			\nu \left(\{x \in X \vert \gamma_\nu(\bphi)(x) \not\in \mathcal{S}_n^+\}\right)=0,
		\end{equation}
		where  $\mathcal{S}_n^+$ denotes the set of symmetric and positive semidefinite $n \times n$-matrices. 
		Also note that 
		\begin{equation} \label{e:rn-trace}
			\frac{d \Lambda_{\bphi}}{d \nu} = \Tr(\gamma_\nu(\bphi)),
		\end{equation}
		where $\Tr(M)$ denotes the trace of a square matrix $M$.
		By \eqref{e:rn-trace}, the measure $\Lambda_{\bphi}$ is mutually absolutely continuous with respect to the measure $\one_{\{\Tr (\gamma_\nu(\bphi))>0\}}\cdot \nu$, and hence   
		\begin{align} \label{e:sing-det}
			\bigwedge_{\blambda \in \mathbb{S}^{n-1}}^{\Lambda_{\bphi}} \frac{d \Gamma_2 \langle \sum_{i=1}^n \lambda_i \phi_i \rangle}{d \Lambda_{\bphi}} &= \bigwedge_{\blambda \in \mathbb{S}^{n-1}}^{\one_{\{\Tr (\gamma_\nu(\bphi)(x))>0\}}\cdot \nu} \Tr (\gamma_\nu(\bphi))^{-1}  \blambda^T \gamma_\nu(f) \blambda\quad \mbox{(\eqref{e:emeas-bilinear}, \eqref{e:rn-trace})} \nonumber\\
			&= \Tr (\gamma_\nu(\bphi))^{-1}  \sigma_1\left( \gamma_\nu(\bphi)\right),
		\end{align}
		where $\sigma_1(M)$ denotes the smallest eigenvalue of $M \in \mathcal{S}_n^+$. 
		
		Note that for any $M \in\mathcal{S}_n^+$, we have $\det(M)>0$ if and only if $\sigma_1(M)>0$. 
		Hence (a) implies that   $\sigma_1\left( \gamma_\nu(\bphi)(x)\right)>0$ for $\one_{\{\Tr (\gamma_\nu(\bphi))>0\}}\cdot \nu$-a.e.~$x \in A$.  This along with \eqref{e:sing-det} shows the equivalence between (a) and (b).
		
		The claim that $\bphi$ is $2$-linearly independent on $I_{\bphi}$ follows immediately from the implication (a)$\implies$(b) as $\det(\gamma_\nu(\bphi))>0$ on $I_{\bphi}$.
		
		Note that 
		\begin{equation} \label{e:tr-det}
			\{M \in \mathcal{S}_n^+ >0 \vert \det(M)>0 \} \subset 	\{M \in \mathcal{S}_n^+ >0 \vert \Tr(M)>0 \}.
		\end{equation}
		Hence \eqref{e:sn-valued} and \eqref{e:tr-det} implies that 
		\begin{equation} \label{e:ac-last}
			\one_{I_{\bphi}} \cdot \nu= \one_{I_{\bphi}} \one_{\{\Tr (\gamma_\nu(\bphi))>0\}}\cdot \nu.
		\end{equation}
		Since the measure $\one_{\{\Tr (\gamma_\nu(\bphi))>0\}}\cdot \nu$ is mutually absolutely continuous to $\Lambda_{\bphi}= \Tr(\gamma_\nu(\bphi))\cdot\nu$, we conclude by \eqref{e:ac-last} that  $\one_{I_{\bphi}} \cdot \nu$ is mutually absolutely continuous with respect to $\one_{I_{\bphi}} \cdot \Lambda_{\bphi}$.
	\end{proof}
	
	\subsection{\texorpdfstring{$p$}{p}-Dirichlet structures}
	The framework of $p$-Dirichlet spaces from Definition \ref{d:dir-space} does not cover the Dirichlet structures of Definition \ref{d:dirstructure}, since the latter assumes no topology on the underlying space. We therefore introduce a parallel framework for Dirichlet structures without a topological assumption; with this framework the proof of the energy image density property for Dirichlet structures proceeds almost identically to the $p$-Dirichlet-space case.
	\begin{definition} \label{d:p-dir-structure} Let $p \in (1,\infty)$.
		We say that $(X,\mathcal{X},\mu,\mathcal{E}_p, \mathcal{F}_p, \Gamma_p)$ is a $p$-Dirichlet structure if it satisfies the following properties:
		\begin{enumerate}
			\item			\textbf{Probability space}: $(X,\mathcal{X},\mu)$ is a probability space and $\mathcal{F}_p$ is a subspace of $L^p(X,\mu)$ such that $\one_X \in \mathcal{F}_p$.
			\item \textbf{Completeness:} 
			$\mathcal{E}_p$ is a semi-norm on $\mathcal{F}_p$ such that $\mathcal{F}_p$ is a Banach space when equipped with the norm $\norm{f}_{\mathcal{F}_p}=(\norm{f}_{L^p}^p+\mathcal{E}_p(f))^{1/p}$.
			\item \textbf{Homogeneity:} For all $f\in \mathcal{F}_p$ there exists a finite non-negative Borel measure $\Gamma_p \langle f \rangle$ on $X$ such that $\Gamma_p \langle f \rangle(X)= \mathcal{E}_p(f)$   and for all $\lambda \in \mathbb{R}$
			\[
			\Gamma_p\langle \lambda f\rangle = |\lambda|^p \Gamma_p\langle f\rangle.
			\]
			\item \textbf{Absolute continuity of energy measure}:  For all $f \in \mathcal{F}_p$, we have $\Gamma_p \langle f \rangle \ll \mu$.
			\item \textbf{Sublinearity:} For every $f,g\in \mathcal{F}_p$ and $A\subset X$, 
			\[
			\Gamma_p\langle f+g\rangle(A)^{\frac{1}{p}}\le \Gamma_p\langle f\rangle(A)^{\frac{1}{p}}+\Gamma_p\langle g\rangle(A)^{\frac{1}{p}}
			\]
			\item \textbf{Chain rule:} For all $f\in \mathcal{F}_p$, $g\in \Lip(\mathbb{R})$ with $g(0)=0$,  we have $g \circ f \in \mathcal{F}_p$ and
			\[
			\Gamma_p\langle g\circ f\rangle  \le \left((\Lip_a[g] \circ f)(x)\right)^p   \,\Gamma_p\langle f\rangle.
			\]
			Note that the measure $\left((\Lip_a[g] \circ f)(x)\right)^p   \,\Gamma_p\langle f\rangle$ is well-defined due to the absolute continuity of energy measure; that is, this measure does not depend on the choice of $f$ from the $\mu$-a.e.~equivalence class.
			\item \textbf{Locality:} For every $f\in \mathcal{F}_p$ and any measurable  set $A\in \mathcal{X}$, if $f|_A=c$ $\mu$-a.e., then $\Gamma_p\langle f\rangle(A)=0$. In particular, $\Gamma_p \langle \one_X \rangle$ is the zero measure.
			\item \textbf{Weak lower semicontinuity:} For every $f\in L^p(X)$, and for any sequence of functions $(f_i)_{i \in \mathbb{N}}\in \mathcal{F}_p$ such that $f_n\to f \in L^p(X)$ with $\sup_{i\in \mathbb{N}} \mathcal{E}_p(f_i)<\infty$, then $f\in \mathcal{F}_p$ and
			\[
			\Gamma_p\langle f \rangle(A)\le \liminf_{i\to \infty} \Gamma_p\langle f_i \rangle(A)
			\]
			for every measurable set $A\subset X$.
		\end{enumerate}
	\end{definition}
	\begin{remark} \label{r:diff}
		Note that the chain rule in Definition \ref{d:p-dir-structure} is stronger than the corresponding assumption in the case of $p$-Dirichlet space (Definition \ref{d:dir-space}-(5)). The chain rule in Definition \ref{d:p-dir-structure} is essentially same as the conclusions in Lemma \ref{lem:contlocal} and Proposition \ref{prop:quasicont-conseq}-(2) without any restrictions such as continuity and quasi-continuity. This is made possible thanks to the absolute continuity assumption.
		
		Similarly, the version of locality in Definition \ref{d:p-dir-structure} holds for all measurable sets, unlike the corresponding property in Definition \ref{d:dir-space} which has the additional restriction that the set $A$ is open.
	\end{remark}
	
	We verify that the above definition contains Dirichlet structure in the sense of Definition \ref{d:dirstructure}. 
	
	\begin{lemma}\label{lem:dirichlet} Let  $(X,\mathcal{X}, \mu,\mathcal{F},\mathcal{E})$ be a Dirichlet structure (in the sense of Definition \ref{d:dirstructure}) with carr\'e du champ $\gamma: \mathcal{F} \times \mathbb{F} \to L^1(X,\mu)$ as characterized by \eqref{e:def-carreduchamp}. 
		Let $\mathcal{F}_2:=\mathcal{F}$, 	 $\mathcal{E}_2(f):= \mathcal{E}(f,f)$ for all $f \in \mathcal{F}_2$  and 
		$\Gamma_2\langle f \rangle := \gamma (f,f) \cdot \mu$ denote the corresponding energy measure. Then  $(X,\mathcal{X}, \mu,\mathcal{E}_2, \mathcal{F}_2, \Gamma_2)$ is a $2$-Dirichlet structure in the sense of Definition \ref{d:p-dir-structure}. 		
	\end{lemma}
	\begin{proof}
		Properties (1), (2), and (4) of Definition \ref{d:p-dir-structure} are immediate. The homogeneity property follows from the bilinearity of $\gamma$ \cite[Proposition I.4.1.3]{BH91}. The sublinearity property follows from Cauchy-Schwarz inequality for energy measure similar to the proof of Lemma \ref{lem:df-dirspace}.
		
		The  chain rule in Definition \ref{d:p-dir-structure} for Lipschitz functions follows from \cite[Corollary I.7.1.2, Theorem I.7.1.1.]{BH91}, Radamacher's theorem on almost everywhere differentiability of Lipschitz functions, and using the fact that for any $g \in \Lip(\mathbb{R})$, we have $\abs{g(x)}' \le \Lip_a[g](x)$ for $\mathcal{L}_1$-a.e.~$x \in \mathbb{R}$, where $g'$ is the derivative of $g$ $\mathcal{L}_1$-a.e. 
		
		The locality property (Definition \ref{d:p-dir-structure}-(7)) follows from \cite[Theorem I.7.1.1]{BH91}.  Finally, the weak lower semicontinuity (Definition \ref{d:p-dir-structure}-(8)) follows from the same argument as given in \cite[Proposition 4.10]{Kajino-Shimizu} or Lemma \ref{l:suff-lsc}.
	\end{proof}

	Since the chain rule and locality property for $p$-Dirichlet structures is valid for all functions in $\mathcal{F}_p$ and for all measurable sets, the same argument in the proof of Lemma \ref{lem:multidimcontraction} yields the following version of Proposition \ref{prop:quasicont-conseq}-(3).
	
	\begin{proposition}\label{prop:multidimcontraction-structure} Let $(X,d,\mu,\mathcal{E}_p, \mathcal{F}_p, \Gamma_p)$ be a $p$-Dirichlet structure and let $n \in \mathbb{N}$.
		For every $g\in \Lip(\mathbb{R}^n)$ with $g(0)=0$ and every $f=(f_1,\dots, f_n) \in \mathcal{F}_p^n$  we have   $g \circ f \in \mathcal{F}_p$ and 
		\begin{equation} \label{e:c-rule-ds}
			\Gamma_p\langle g \circ f\rangle \le \left(1 \vee n^{(p-2)/2}\right) (\Lip_a[g] \circ f)^p \cdot \sum_{i=1}^n \Gamma_p\langle f_i\rangle.
		\end{equation}
	\end{proposition}
	The same argument in the proof of  Lemma \ref{l:invertibly-independence} yields the following analogous result for $p$-Dirichlet structures, where the definition of $p$-independence for $p$-Dirichlet structures is as given in Definition \ref{d:p-independent}.
	\begin{lemma} \label{l:invertibly-independence-dirstructure}
		Let  $(X,\mathcal{X}, \mu,\mathcal{F},\mathcal{E})$ be a Dirichlet structure (in the sense of Definition \ref{d:dirstructure}) with carr\'e du champ $\gamma:\mathcal{F} \times \mathcal{F} \to L^1(\mu)$. Consider this also as a $2$-Dirichlet structure  $(X,\mathcal{X}, \mu,\mathcal{E}_2, \mathcal{F}_2, \Gamma_2)$  in the sense of Definition \ref{d:p-dir-structure} as given by Lemma \ref{lem:dirichlet}.
		Let $n \in \mathbb{N}$,     $\bphi=(\phi_1,\ldots,\phi_n) \in \mathcal{F}_2^n$ and let $\gamma(\bphi)$ denote the associated carr\'e du champ matrix as defined in \eqref{e:def-cdc-matrix}.  Then the following are equivalent:
		\begin{enumerate}[(a)]
			\item $\det(\gamma(\bphi))>0$ for $\one_{\{\Tr (\gamma(\bphi))>0\}}\cdot \mu$-a.e.~$A$.
			\item $\bphi$ is $2$-linearly independent on $A$.
		\end{enumerate}
		Furthermore, $\bphi$ is $2$-linearly independent on the set $I_{\bphi}:=\{x \in X \vert \det(\gamma(\bphi)(x))>0\}$ and the measures
		$\one_{I_{\bphi}} \cdot \nu$ and $\one_{I_{\bphi}} \cdot \Lambda_{\bphi}$ are mutually absolutely continuous.
	\end{lemma}
	
	\subsection{Energy image density properties in different frameworks}	
	Next, we formulate our results concerning the energy image density property in different frameworks: regular $p$-Dirichlet spaces, MMD spaces (regular, strongly local Dirichlet forms), and $p$-Dirichlet structures. We will prove these in \textsection\ref{s:proofs} after developing relevant tools in  \textsection\ref{s:measures}.
	
	The following theorem establishes energy image density property for $p$-Dirichlet spaces.
	\begin{theorem}\label{thm:eid-dirspace}
		Assume that $(X,d,\mu, \mathcal{E}_p, \mathcal{F}_p, \Gamma_p)$  is a regular local $p$-Dirichlet space and let $n \in \mathbb{N}$. If $\bphi=(\phi_1,\dots, \phi_n):X\to \mathbb{R}^n, n \in \mathbb{N}$ is such that each $\phi_i\in \mathcal{F}_p$ is quasicontinuous for $i=1,\dots, n$, and if $\bphi$ is   $p$-independent  on a Borel set $A\subset X$, then 
		\[
		\bphi_*(\one_A\cdot \Lambda_{\bphi})\ll\mathcal{L}_n,
		\]
		where $\Lambda_{\bphi}=\sum_{i=1}^d \Gamma_p\langle \phi_i\rangle$.
	\end{theorem}
	Note that since we do not assume that the energy measure is absolutely continuous with respect to $\mu$, we need to use quasicontinuous representatives in the formulation as the measure $\bphi_*(\one_A\cdot \Lambda_{\bphi})$ is well-defined\footnote{That is, it is independent of the choice of quasicontinuous representatives.} due to Propositions \ref{prop:quasicont}-(2) and \ref{prop:quasicont-conseq}-(1).  
	
	The energy image density property for MMD spaces is stated below with a more classical formulation using invertibility of carr\'e du champ matrix as defined above. 
	\begin{theorem}\label{thm:eid-mmdspace}	
		Let $(X,d,\mu,\mathcal{E},\mathcal{F})$ be an MMD space, and let $n \in \mathbb{N}$.  Let $\nu$ be a minimal energy dominant measure for $(X,d,\mu,\mathcal{E},\mathcal{F})$.  Let $\bphi=(\phi_1,\dots, \phi_n):X\to \mathbb{R}^n, n \in \mathbb{N}$ is such that each $\phi_i\in \mathcal{F}_p$ is quasicontinuous for $i=1,\dots, n$ and let $\gamma_\nu(\bphi)$ denote the associated carr\'e du champ matrix with respect to $\nu$ as defined in \eqref{e:def-cdc-reg}.
		Then 
		\[
		\bphi_*(\one_{\{\det(\gamma_\nu(\bphi))>0\}}\cdot \nu ) \ll \mathcal{L}_n,
		\]
		where $\Lambda_{\bphi}=\sum_{i=1}^d \Gamma_p\langle \phi_i\rangle$ and  $\Phi(x)_{i,j}=\frac{d\Gamma_2\langle \phi_i,\phi_j \rangle}{d\Lambda_{\bphi}}$.
	\end{theorem}
	The justification for the use of quasicontinuous representatives is similar to the remarks following Theorem \ref{thm:eid-dirspace} as Theorem \ref{thm:eid-mmdspace} can be viewed as a special case due to Lemma \ref{l:invertibly-independence}.

	Finally, we state the energy image density property for $p$-Dirichlet structures.
	\begin{theorem}\label{thm:eid-pdirstructure}
		Assume that  $(X,\mathcal{X},\mu,\mathcal{E}_p, \mathcal{F}_p, \Gamma_p)$ is a   $p$-Dirichlet structure and let $n \in \mathbb{N}$. If $\bphi=(\phi_1,\dots, \phi_n):X\to \mathbb{R}^n, n \in \mathbb{N}$ is such that each $\phi_i\in \mathcal{F}_p$  and if $\bphi$ is a $p$-independent function  on a measurable set $A\subset \mathcal{X}$, then 
		\[
		\bphi_*(\one_A \cdot \Lambda_{\bphi})\ll\mathcal{L}_n,
		\]
		where $\Lambda_{\bphi}=\sum_{i=1}^d \Gamma_p\langle \phi_i\rangle$.
	\end{theorem}
	The affirmative answer to Bouleau--Hirsch conjecture follows from Theorem \ref{thm:eid-pdirstructure} and Lemmas \ref{lem:dirichlet}. \ref{l:invertibly-independence-dirstructure}.
	\section{Background on analysis of measures in \texorpdfstring{$\mathbb{R}^n$}{Rn}} \label{s:measures}
	\subsection{Curves, line integrals and upper gradients}
	We recall some basic notions concerning curves, line integrals and upper gradients and refer the reader to \cite[Chapters 5 and 6]{HKST} for a comprehensive introduction. 
	A \emph{curve} in $\mathbb{R}^n$ is a continuous map $\gamma:I \to \mathbb{R}^n$, where $I \subset \mathbb{R}$ is a compact interval. 
	Given a curve $\gamma: I \to \mathbb{R}^n$, the \emph{length of $\gamma$} (denoted by $L(\gamma)$) is the supremum 
	\[
	L(\gamma):=\sup	\sum_{i=1}^k \abs{\gamma(t_i)-\gamma(t_{i-1})}, 
	\] 
	where $k$ varies over $\mathbb{N}$, and the numbers $t_i$ vary over all finite sequences of the form $\min(I)=t_0 < t_1 \ldots <t_k= \max(I)$. A \emph{rectifiable curve} is a curve whose length is finite. For a rectifiable curve $\gamma:I \to \mathbb{R}^n$, there is an associated \emph{length function} $s_\gamma: I \to [0,L(\gamma)]$ defined by 
	\[
	s_\gamma(t):= L\left(\restr{\gamma}{[\min(T),t]}\right), \quad \mbox{for all $t \in I$.}
	\]
	It is known that for any rectifiable curve $\gamma:I \to \mathbb{R}^n$, $s_\gamma$ is continuous and non-decreasing. Furthermore, there is a unique $1$-Lipschitz continuous map $\gamma_s: [0,L(\gamma)] \to \mathbb{R}^n$ such that 
	\[
	\gamma= \gamma_s \circ s_\gamma.
	\]
	The curve $\gamma_s$ is called the \emph{arc length parametrization of $\gamma$.} It immediately follows from the definition that 
	\[
	L\left(\restr{\gamma_s}{[u,v]}\right)= v-u, \quad \,\mbox{for any $0\le u \le v \le L(\gamma)$.}
	\]
	By Radamacher's theorem the derivative $\gamma_s'(t)$ exists for $\mathcal{L}_1$-a.e.~$t\in [0,L(\gamma)]$ and satisfies 
	\[
	\abs{\gamma_s'(t)}= \lim_{v \to t} \abs{\frac{\gamma_s(v)-\gamma_s(t)}{v-t}}=1, \quad \mbox{for $\mathcal{L}_1$-a.e.~$t \in [0,L(\gamma)]$.}
	\]
	Given a rectifiable curve $\gamma:I \to \mathbb{R}^n$ and a non-negative Borel function $g: \mathbb{R}^n \to [0,\infty]$, the \emph{line integral of $g$ over $\gamma$} is defined as 
	\[
	\int_{\gamma}g \,ds:= \int_0^{L(\gamma)} g(\gamma_s(t))\,dt,
	\]
	where $\gamma_s:[0,L(\gamma)]\to \mathbb{R}^n$ is the arc length parametrization of $\gamma$. Next, we recall the notion of an upper gradient due to \cite[\textsection 2.9]{heinonenkoskela}.
	\begin{definition}
		We say that a Borel function $g: \mathbb{R}^n \to [0,\infty]$ is an \emph{upper gradient} of $f: \mathbb{R}^n \to \mathbb{R}$ if 
		\[
		\abs{f(\gamma(a)) - f(\gamma(b))} \le \int_\gamma g\,ds,
		\]
		for all rectifiable curves $\gamma:[a,b] \to \mathbb{R}^n$.
	\end{definition}

	\subsection{Decomposability bundle and cone null sets}
	We  recall the definition of decomposability bundle of Alberti and Marchese  \cite{AlbMar}. 
	Let $\Gr(\mathbb{R}^n)=\bigcup_{d=0}^{n}{\Gr}(d,n)$ where ${\Gr}(d,n)$ denotes the \emph{Grasmannian space} of all $d$-dimensional subspaces of $\mathbb{R}^n$. The space $\Gr(\mathbb{R}^n)$ is a finite disjoint union of smooth manifolds, and inherits from this a natural topology. 
	\begin{definition} \label{d:am-decbundle}
		Let $\mu$   be a finite Borel measure on $\mathbb{R}^n$. Then the \emph{decomposability bundle} is a Borel map $T^{AM}_\mu: \mathbb{R}^n \to \Gr(\mathbb{R}^n)$ such that the following
		statements hold:
		\begin{enumerate}[(i)]
			\item Every Lipschitz function $f \in \Lip(\mathbb{R}^n)$ is differentiable at $\mu$-a.e.~$x \in \mathbb{R}^n$ with respect to the
			linear subspace $T_\mu^{AM}$, that is, there exists a linear function $df_x^{AM}: T^{AM}_\mu(x) \to \mathbb{R}$, such that
			\[ f(x+h) = f(x) + df_x^{AM}(h) + o(\lvert h\rvert)
			\quad \text{for } h \in T^{AM}_\mu(x). \]
			\item The previous statement is sharp:  there exists a Lipschitz function $f \in \Lip(\mathbb{R}^n)$
			such that for $\mu$-a.e.~$x \in \mathbb{R}^n$ and every $v \not\in T^{AM}_\mu(x)$ the derivative of $f$ at $x$ in the
			direction $v$ does not exist. More precisely, for $\mu$-a.e.~$x \in \mathbb{R}^n$, we have
			\[
			\liminf_{t \to 0} \frac{f(x+tv)-f(x)}{t} < \limsup_{t \to 0}  \frac{f(x+tv)-f(x)}{t}, \quad \mbox{for all $v \not\in T^{AM}_\mu(x)$}
			\]
		\end{enumerate}
	\end{definition}
	The definition above is different from the one in \cite[\textsection 2.6]{AlbMar} but is equivalent due to \cite[Theorem 1.1]{AlbMar}.
	It is immediate that the two defining properties determine $T_\mu^{AM}$ uniquely up to $\mu$-almost everywhere equivalence.
	For every finite Radon measure $\mu$ on $\mathbb{R}^n$, a  decomposability bundle $T^{AM}_\mu$ exists by \cite[Lemma 2.4]{AlbMar}. If $n \ge 2$, the decomposability bundle has a canonical pointwise well-defined version   described  in terms of directions of normal currents \cite[Theorem 6.4, \textsection 6.1]{AlbMar}.

	The following proposition, which follows from \cite[Corollary 1.12, Lemma 3.1]{DR} and \cite[Proposition 2.9]{AlbMar} is key to our approach; see also \cite[Theorem 6.4, \textsection 1.5]{AlbMar}.
	\begin{proposition} \label{prop:decompoabilitybundlesinuglar}
		If $\mu$ be a finite singular measure on $\mathbb{R}^n$, then it has a decomposability bundle $T_\mu^{AM}$ and $\dim(T_\mu^{AM}(x))\le n-1$ for $\mu$-a.e. $x\in \mathbb{R}^n$. 
	\end{proposition}
	\begin{proof}
		Suppose to the contrary that $A$ is Borel set such that $\mu(A)>0$ and $T_\mu^{AM}(x)= \mathbb{R}^n$ for all $x \in A$.   Then by strong locality principle for decomposability bundle in \cite[Proposition 2.9-(i)]{AlbMar},  the restriction   $\mu_A:=\one_A \cdot \mu$ of $\mu$ on $A$ satisfies 
		\[
		T^{AM}_{\mu_A}(x)= T^{AM}_{\mu}(x)= \mathbb{R}^n, \quad \mbox{for $\mu_A$-a.e.~$x \in \mathbb{R}^n$.}
		\]
		By \cite[Lemma 3.1 and Corollary 1.12]{DR}, we conclude that $\mu_A \ll \mathcal{L}_n$, which contradicts our assumptions $\mu \perp \mathcal{L}_n$ and $\mu_A(A)=\mu(A)>0$.
	\end{proof}

	We will apply Proposition \ref{prop:decompoabilitybundlesinuglar} in the proof of Theorem \ref{thm:eid-dirspace} via contradiction. A key insight for us is that the existence of a lower dimensional decomposability bundle implies the existence of certain approximations. This argument uses the notion of cone null sets whose definition we recall below \cite[\textsection 4.11]{AlbMar}. 
	Let $v\in \mathbb{R}^n$ be a unit vector in $\mathbb{R}^n$ and $\theta\in (0,\pi/2)$. The \emph{cone about $v$ with angle $\theta$} is denoted
	\[
	C(v,\theta)=\{w\in \mathbb{R}^n : \langle v,w\rangle \geq \cos(\theta)\abs{w}\}.
	\] 
	The topological interior of the cone  is denoted $C^\circ$, and $-C$ denotes the opposite cone $C(-v,\theta)$.

	For the definition of cone null sets below, it is helpful to recall that  	  for any Lipschitz curve $\gamma:I \to \mathbb{R}^n$, the derivative
	\[
	\gamma'(t):= \lim_{s \to t, s \in I} \frac{\gamma(s)-\gamma(t)}{s-t},  
	\]
	exists for $\mathcal{L}_1$-a.e.~$t \in I$ due to Radamacher's theorem.
	\begin{definition}\label{d:cone-null}
		If $C=C(v,\theta)$ is a cone with $v \in \mathbb{S}^{n-1}, \theta \in (0,\pi/2)$ and $K \subset \mathbb{R}^n$ is a compact set, then $K\subset \mathbb{R}^n$ is called $C$-cone null, if for every Lipschitz curve $\gamma:I\to \mathbb{R}^n$  such that $I \subset \mathbb{R}$ is a compact interval, $\gamma'(t)\in C$ for $\mathcal{L}_1$-a.e.~$t\in I$ we have $\cH^1(\operatorname{image}(\gamma) \cap K)=0$. 
	\end{definition}
	The following lemma  is  useful for the computation of upper gradient  on cone null sets.
	\begin{lemma} \label{l:conenull}
		Let $K \subset \mathbb{R}^n$ be compact $C$-cone null set, where $C=C(v,\theta) \subset \mathbb{R}^n, v \in \mathbb{S}^{n-1}, \theta \in (0,\pi/2)$ is a cone. Let $\gamma:I \to \mathbb{R}^n$ be a Lipschitz curve. Then 
		\[
		\gamma'(t) \notin C^\circ \cup (-C)^\circ, \quad \mbox{for $\mathcal{L}_1$-a.e.~$t \in \gamma^{-1}(K)$.}
		\]
	\end{lemma}
	\begin{proof}
		By considering the reversal of the curve $\gamma$, it is enough to show that $	\gamma'(t) \notin C^\circ$ for $\mathcal{L}_1$-a.e.$t \in \gamma^{-1}(K)$; that is, $\mathcal{L}_1(S)=0$, where
		$$S=\{t \in I \vert \mbox{$\gamma(t) \in K$, $\gamma'(t)$ exists and $\gamma'(t) \in C^\circ$}\}.$$
		Assume to the contrary that $\mathcal{L}_1(S)>0$. Let $\epsilon \in (0,\mathcal{L}_1(S))$. By $C^1$-approximation of Lipschitz functions (see \cite[Theorem 5.1.12]{KrPa} or \cite[Theorem 6.11]{EvGa}), there exists a $C^1$ function 
		$\widetilde{\gamma}: \mathbb{R} \to \mathbb{R}^n$ such that
		\[
		\mathcal{L}_1 \left( \left\{ t \in I \vert \gamma(t) \neq \widetilde{\gamma}(t) \mbox{ or }  \gamma'(t) \neq \widetilde{\gamma}'(t) \right\}\right)< \epsilon.
		\]
		Hence  the set
		\begin{equation} \label{e:cnull1}
			\widetilde{S}:= \{t \in S \vert  \widetilde{\gamma}(t)=\gamma(t) \in K, \widetilde{\gamma}'(t)=\gamma'(t) \in C^\circ  \}
		\end{equation}
		satisfies $\mathcal{L}_1(\widetilde{S}) \ge \mathcal{L}_1(S)- \epsilon>0$. Hence there is a compact set $F \subset \mathbb{R}$ such that
		\begin{equation} \label{e:cnull2}
			\mathcal{L}_1(F) >0, \quad \mbox{and } \quad F \subset \widetilde{S}.
		\end{equation}
		Since $\widetilde{\gamma}$ is $C^1$ for each $t \in F$, there is a bounded open interval $U_t \subset  \mathbb{R}$ such that $t\in U_t$ and  
		$\overline{U_t} \subset (\widetilde{\gamma}')^{-1}(C^\circ)$. By compactness of $F$, there exists $k \in \mathbb{N}$ and $\{t_1,\ldots,t_k\}  \subset F$ such that  \begin{equation} \label{e:cnull3}
			\bigcup_{i=1}^k U_{t_i} \supset F. 
		\end{equation}
		Let $\widetilde{\gamma}_i:\overline{U_{t_i}} \to \mathbb{R}^n$ denote the restriction of $\widetilde{\gamma}$ to the closure $\overline{U_{t_i}}$ for each $i=1,\ldots,k$. Each curve  $\widetilde{\gamma}_i$ satisfies   $\widetilde{\gamma}_i'(t) \in C^\circ$ for all $t \in \overline{U_{t_i}}$ for all $i=1,\ldots,k$. Combining this with the fact that $K$ is $C$-cone null and the area formula \cite[Theorem 3.8]{EvGa}, we have 
		\begin{equation*} 
			\mathcal{L}_1 \left( \{t \in \overline{U_{t_i}} \vert \widetilde{\gamma}(t) \in K \}\right)=0, \quad \mbox{for all $i=1,\ldots,k$.}
		\end{equation*}
		Hence by \eqref{e:cnull3} we have 
		\begin{equation}
			\label{e:cnull4} 
			\mathcal{L}_1 \left( \{t \in F \vert \widetilde{\gamma}(t) \in K \}\right)=0.
		\end{equation}
		Note that  \eqref{e:cnull4} along with  \eqref{e:cnull1} and  \eqref{e:cnull2}  gives the desired contradiction since   $\widetilde{\gamma}(t)  \in K$ for all $t \in F$.
	\end{proof}
	By the fundamental theorem of calculus and triangle inequality, for any $C^1$ function $f: \mathbb{R}^n \to \mathbb{R}$, the function $\abs{\nabla f}$ is always an upper gradient.
	The following lemma gives a smaller upper gradient on cone null sets.
	\begin{lemma} \label{l:upper-gradient}
		Let $K \subset \mathbb{R}^n$ be a compact $C(v,\theta)$-cone null set and $f \in C^1(\mathbb{R}^n)$, for some $v \in \mathbb{S}^{n-1}$ and $\theta \in (0,\pi/2)$.
		Then,
		\[
		g = \abs{\nabla f} \one_{\mathbb{R}^n \setminus K} + \sup_{w \in \mathbb{S}^{n-1} \setminus (C^\circ \cup( -C^\circ))} \abs{\langle \nabla f, w \rangle}
		\]
		is an upper gradient for $f$, where $C=C(v,\theta)$ such that
		\begin{enumerate}[(a)]
			\item $g$ is lower semicontinuous;
			\item For any closed ball $B$ such that $B^\circ \supset K$, there exists a non-decreasing sequence of  continuous functions $g_k: \mathbb{R}^n \to [0,\infty)$ such that $\lim_{k \to \infty} g_k(x)= g(x)$ for all $x \in \mathbb{R}^n$ and $g_k(x)=g(x)$ for all $x \in \mathbb{R}^n \setminus B, k \in \mathbb{N}$.
		\end{enumerate}
	\end{lemma}	
	\begin{proof}
		Let $\gamma:[0,L] \to \mathbb{R}^n$ be any rectifiable curve  with arc length parametrization. 
		Note that 
		\begin{equation} \label{e:ug1}
			\abs{f(\gamma(0))- f(\gamma(L))} =\abs {\int_0^L  (f\circ \gamma)'(t)\,dt} \le \int_0^L \abs{ (f\circ \gamma)'(t)}\,dt.
		\end{equation}
		By the chain rule, we have $(f\circ \gamma)'(t) = \langle \nabla f(\gamma(t)), \gamma'(t)\rangle$. Since $\abs{\gamma'(t)}=1$ for $\mathcal{L}_1$-a.e.~$t \in [0,L]$, by Lemma \ref{l:conenull}, we have
		\begin{equation} \label{e:ug2}
			\abs{(f\circ \gamma)'(t)} = \langle \nabla f(\gamma(t)), \gamma'(t)\rangle  \le g(\gamma(t)), \quad \mbox{for $\mathcal{L}_1$-a.e.~$t \in [0,L]$.}
		\end{equation}  
		We conclude that $g$ is an upper gradient for $f$ using \eqref{e:ug1} and \eqref{e:ug2}.
		
		The lower semicontinuity of $g$ is immediate. By Baire's theorem there exists a  non-decreasing sequence of  continuous functions $\widetilde{g_k}: \mathbb{R}^n \to [0,\infty)$ such that $\lim_{k \to \infty} \widetilde{g_k}(x)= g(x)$ for all $x \in \mathbb{R}^n$; for instance, we can choose 
		\[
		\widetilde{g_k}(x):= \inf_{y \in \mathbb{R}^n}\left( g(y)+ k\abs{x-y}\right), \quad \mbox{for all $x \in \mathbb{R}^n$.}
		\]
		Let $B$ be a closed ball such that $K \subset B^\circ$.
		Choose a compact set $\widetilde{K} \subset \mathbb{R}^n$ such that $K \subset \widetilde{K}^\circ \subset \widetilde{K} \subset B^\circ$.
		By Urysohn's lemma, there exists a continuous function $\psi: \mathbb{R}^n \to [0,1]$ such that  $\psi \equiv 0$ on $\widetilde{K}$ and $\psi \equiv 1$ on $\mathbb{R}^n \setminus B$. Since $g$ is continuous on $\mathbb{R}^n \setminus K$ and $\psi \equiv 0$ in a neighborhood of $K$, we have $\psi g \in C(\mathbb{R}^n)$. Hence the sequence of functions 
		\[
		g_k(x):= \psi(x) g(x) + (1-\psi(x))\widetilde{g_k}(x), \quad \mbox{for all $x \in \mathbb{R}^n$}
		\]
		satisfies the desired properties.
	\end{proof}
	The following proposition is a consequence of results in \cite{AlbMar}.
	It implies that every measure whose decomposability bundle has dimension at most $n-1$ can be covered, up to a null set, by a countable collection of compact cone-null sets with cone angles arbitrarily close to $\pi/2$.
	Since the proof uses spherical metric on $\mathbb{S}^{n-1}$, we recall its definition below. The \emph{spherical metric} on $\mathbb{S}^{n-1}$ is defined by $\langle v, w \rangle= \cos \left(d_{\mathbb{S}^{n-1}}(v,w)\right)$ and $d_{\mathbb{S}^{n-1}}(v,w) \in [0,\pi]$ for all $v,w \in \mathbb{S}^{n-1}$. Note that for any $v \in \mathbb{S}^{n-1}, \theta \in (0,\pi/2)$, the cone $C(v,\theta)$ can be described in terms of the spherical metric as
	\[
	C(v,\theta)=   \{w \in \mathbb{R}^n \vert \mbox{ $w=0$, or $w \neq 0$ and $d_{\mathbb{S}^{n-1}}(v,w/\abs{w}) \le \theta$ } \}.
	\]
	\begin{proposition}\label{prop:decompositionconenull}  
		If $\mu$ is a non-zero finite  Borel measure on $\mathbb{R}^n$ with an at most $(n-1)$-dimensional decomposability bundle $T_{\mu}^{AM}: \mathbb{R}^n \to \bigcup_{d=0}^{n-1} \Gr(n,d)$, then for every $\epsilon \in (0,\pi/2)$ there exist a unit vector $w \in \mathbb{S}^{n-1}$ and a compact $C(w,\pi/2-\epsilon)$-cone null set $K \subset \mathbb{R}^n$ such that $\mu(K)>0$. 
	\end{proposition}
	\begin{proof}
		Let $N=\{w_i : i \in I\}$ be an $\epsilon$-net of $S^{n-1}$, with respect to the spherical metric; that is, $N$ is a maximal subset such that any pair of distinct points in $N$ have distance at least $\epsilon$. 
		For each $w_i \in N$, let $$E_i=\{x\in \mathbb{R}^n : T_\mu^{AM}(x)\cap C(w_i,\pi/2-\epsilon)=\{0\}\}.$$ 
		We claim that 
		\begin{equation} \label{e:cover}
			\mathbb{R}^n=\bigcup_{i \in I} E_i.
		\end{equation}
		To this end, let $x \in \mathbb{R}^n$. Since $\dim(T_\mu^{AM}(x)) \le n-1$, there exists $w \in \mathbb{S}^{n-1} \cap T_\mu^{AM}(x)^\perp$, where $V^\perp$ denotes the orthogonal complement of a subspace $V$. By the maximality of $N$, there exists $i \in I$ such that $d_{\mathbb{S}^{n-1}}(w,w_i)<\epsilon$.
		By triangle inequality for the spherical metric, we have 
		\[
		C(w_i, \pi/2 -\epsilon) \cap T_\mu^{AM}(x) \subset 	C(w, \pi/2 -\epsilon+d_{\mathbb{S}^{n-1}}(w,w_i)) \cap T_\mu^{AM}(x) = \{0\}, 
		\]
		where the last equality follows from the fact that $d_{\mathbb{S}^{n-1}}(w,w_i)<\epsilon$ and $w \in T_\mu^{AM}(x)^\perp \cap \mathbb{S}^{n-1}$. This completes the proof of \eqref{e:cover}.

		Since $\mu$ is non-zero, there is an index $i \in I$ such that $\mu(E_i)>0$. Now, by \cite[Lemma 7.5]{AlbMar}, there exists a Borel set $F_i \subset E_i$ such that $F_i$ is $C(w_i,\pi/2-\epsilon)$-cone null, and $\mu(E_i \setminus F_i)=0$. The claim follows then by the inner regularity of $\mu$.
	\end{proof}
	
	In fact, the proof of the previous proposition shows a bit more than we need: There exist compact sets $K_i$ and cones $C_i=C(w_i,\pi/2-\epsilon), i \in \mathbb{N}$ such that
	\begin{enumerate}
		\item $K_i$ are $C_i$-cone null for each $i \in \mathbb{N}$ and
		\item The sets $K_i$ cover $\mathbb{R}^n$ up to a $\mu$-null set; that is, \[
		\mu\left(\mathbb{R}^n \setminus \bigcup_{i=1}^\infty K_i\right)=0.
		\]
	\end{enumerate}
	\subsection{Approximation of Lipschitz functions} \label{ss:approximation}
	We state a new approximation theorem, which is a refinement of an earlier one from \cite{seblipdense,continuousdense}; see therein for a more detailed comparison with other methods used in other works. Inspiration for this approximation comes from \cite{Ch}  and the proof here is modeled after a well-known equality proof of capacity and modulus in \cite[Proof of Proposition 2.17]{heinonenkoskela}; see also the arguments in \cite{Keith}. The scheme is somewhat similar to a different approximation method from \cite[Lemma 6.3]{Bate}, which is based on \cite{ACPdiff, ACPstruct}; see in particular the general metric space result in \cite[Theorem 1.6.]{bate2024fragment} which applies to disconnected spaces. Similar ideas have also appeared in \cite{schioppaalberti,AlbMar, Aliaga}.
	
	We say that a sequence of functions $f_i\in \Lip(\mathbb{R}^n)$ \emph{converges weak-$*$} to $f\in \Lip(\mathbb{R}^n)$ if $f_i\to f$ pointwise and $\sup_{i\in\mathbb{N}}\LIP[f_i]<\infty$.  Note that if $f_i$ converges weak-$*$ to $f \in \Lip(\mathbb{R}^n)$, then $f_i$ converges uniformly on compact subsets of $\mathbb{R}^n$ to $f$. 
	The terminology weak-$*$ convergence comes from the fact that the space of Lipschitz functions is a dual space, and the topology mentioned above coincides with the weak-$*$ topology induced by its predual, the Arens–Eells space \cite[Corollary 3.4]{weaver-book}.
	
	The following weak-$*$ approximation of Lipschitz functions plays a crucial role in our proof of the energy image density conjecture.
	\begin{proposition}\label{prop:approximations} 
		Assume that $f\in \Lip(\mathbb{R}^n)$, $\LIP[f]>0$ and that $g:\mathbb{R}^n \to (0,\LIP[f]]$ is a strictly positive lower semicontinuous function  which is an upper gradient for $f$, satisfying $g|_{\mathbb{R}^n \setminus B(0,R)}\equiv \LIP[f]$ for some $R>0$. Let  $(g_i)_{i=1}^\infty$ be an increasing sequence of continuous functions with $g_i\to g$ pointwise and $g_i|_{\mathbb{R}^n \setminus B(0,R)}\equiv \LIP[f]$ for all $i\in \mathbb{N}$. Then there exists a sequence $f_i$ of $\LIP[f]$-Lipschitz functions $f_i \in \Lip(\mathbb{R}^n)$ that converge weak-$*$ to $f$ and for which
		\[
		\Lip_a[f_i]\leq g_i(x),
		\]
		for all $i\in \mathbb{N}$ and for all $x \in \mathbb{R}^n$.
	\end{proposition}
	\begin{proof}
		Since $g$ is lower semi-continuous, there is a $\delta>0$ such that $g|_{\overline{B(0,R)}}>\delta$. Thus, we see from pointwise convergence, the increasing property and compactness that for all large enough $i\in \mathbb{N}$, we have $g_i(x)>\delta$ for all $x\in \overline{B(0,R)}$. For such $i\in \mathbb{N}$, we have in fact that $g_i(x)>\delta$ for all $x\in \mathbb{R}^n$, by the assumption that $g_i|_{\mathbb{R}^n \setminus B(0,R)}\equiv \LIP[f]$. Up to reindexing, we may assume that $g_i(x)>\delta$ for all   $x \in \mathbb{R}^n$ and $i\in \mathbb{N}$.
		
		Now, define for any $x \in \mathbb{R}^n, i \in \mathbb{N}$,
		\begin{equation}\label{eq:fidef}
			f_i(x):=\min(f(x),\inf_{\gamma} f(\gamma(0))+\int_\gamma g_i ds),
		\end{equation}
		where the infimum is taken over all rectifiable curves $\gamma:[0,1]\to \mathbb{R}^n$ with $\gamma(1)=x$ and $\gamma(0)\in \overline{B(0,R)}$. The rest of the proof shows that this construction works. 
		
		We first argue that $f_i$ is a well-defined and finite Lipschitz function for each $i\in \mathbb{N}$. Let $x\in \mathbb{R}^n$ be fixed. We have the upper bound $f_i(x)\leq f(x)$ by construction. Since the infimum in \eqref{eq:fidef} is obtained by curves with $\gamma(0)\in \overline{B(0,R)}$, we have
		\[
		f_i(x)\geq \min(f(x),\inf_{y\in \overline{B(0,R)}}f(y)).
		\]
		In particular $f_i:\mathbb{R}^n \to \mathbb{R}$.
		
		Next, we estimate the difference $|f_i(x)-f_i(y)|$ for $x,y\in \mathbb{R}$. Let $\sigma:[0,1]\to \mathbb{R}^n, \sigma(t)=(1-t)x+ty$ be a straight line segment from $x$ to $y$. If $\gamma$ is any curve with $\gamma(1)=x$, then the (reparametrization of the) concatenation $\sigma \star \gamma$ satisfies $\sigma \star \gamma(1)=y$. Moreover, we have
		\begin{align}
			f_i(y)&\leq f(\sigma \star \gamma(0))+\int_{\sigma \star \gamma} g_i ds \nonumber \\
			&\leq f(\gamma(0))+\int_\gamma f_i ds + \int_{\sigma} g_i ds \nonumber \\
			&\leq f(\gamma(0))+\int_\gamma f_i ds + \LIP[f]|x-y| \label{eq:fioscbound}.
		\end{align}
		Taking an infimum over all $\gamma$ with $\gamma(1)=x$, we get
		\[
		f_i(y)\leq \inf_{\gamma} f(\gamma(0))+\int_\gamma g_i ds+\LIP[f]|x-y|,
		\]
		since also 
		\[
		f_i(y)\leq f(y)\leq f(x)+\LIP[f]|x-y|,
		\]
		we get $f_i(y)\leq f_i(x)+\LIP[f]|x-y|$ by combining the previous two inequalities.
		Switching the role of $x$ and $y$ yields that $f_i$ is $\LIP[f]$-Lipschitz.
		
		Next, the bound $\Lip_a[f_i](x)\leq g_i(x)$ is obvious for $x\not\in B(0,R)$. Let now $x,y\in B$ where $B$ is any ball in $\mathbb{R}^n$ contained in $B(0,R)$. Then, $\sigma \subset B \subset B(0,R)$ by convexity and we can improve the calculation in \eqref{eq:fioscbound} by replacing $\LIP[f]$ with $\sup_{z\in B} g_i(z)$ to give
		\[
		|f_i(x)-f_i(y)|\leq [\sup_{z\in B} g_i(z)] |x-y|
		\]
		Thus, $\LIP[f_i|_{B}]\leq \sup_{z\in B} |g_i(z)|$. By using this together with the continuity of $g_i$ for $B=B(x,r)$ and sending $r\to 0$ we see $\Lip_a[f_i](x)\leq g_i(x)$.
		
		We are left to show that $f_i(x)\to f(x)$ for every $x\in \mathbb{R}^n$. Fix an $x\in \mathbb{R}^n$ and suppose for the sake of contradiction that this is not the case. The integrands $g_i$ in  \eqref{eq:fidef} are increasing in $i$, and thus one sees that $f_i(x)\leq f_j(x)$ for every $i\leq j$. Thus, the limit $\lim_{i\to \infty} f_i(x)$ exists. Recall that $f_i(x)\leq f(x)$ for every $i\in \mathbb{N}$, and by our contrapositive, there must be some $\eta>0$ so that $f_i(x)<f(x)-\eta$ for all $i\in \mathbb{N}$. 
		
		By the definition of $f_i$, there exists a sequence $(\gamma_i)_{i \in \mathbb{N}}$ of rectifiable curves for which
		\begin{equation} \label{eq:contrary}
			f(\gamma_i(0))+\int_{\gamma_i} g_i ds \le f(x)-\eta, \mbox{where $\gamma_i(1)=x$ for all $i \in \mathbb{N}$}
		\end{equation}
		and $\gamma_i(0)\in \overline{B(0,R)}$. The case $f_i(x)=f(x)$ is precluded, since $f_i(x)<f(x)$.
		
		Recall that $g_i(z)>\delta$ for all $z\in \mathbb{R}^n$. Therefore, we have
		\[
		\delta \len(\gamma_i)\leq \int_{\gamma_i} g_i ds \leq f(x)-\eta-f(\gamma_i(0))\leq f(x)-\inf_{z\in \overline{B(0,R)}} f(z) := M <\infty.
		\]
		Thus, the lengths of $\gamma_i$ are uniformly bounded. Since $\gamma_i(1)=x$, these curves are also all contained in a compact subset of $\mathbb{R}^n$. By Arzel\'a-Ascoli, after a reparametrization, and passing to a subsequence, we may assume that $\gamma_i$ converge uniformly to some curve $\gamma$. 
		
		By the lower-semicontinuity of line integrals, see e.g. \cite[Proposition 4]{Keith} or \cite[Lemma 4.2]{DavidSEBsplitting}, for every $j\in \mathbb{N}$, we have
		\begin{align*}
			\int_\gamma g_j ds &\leq \liminf_{i\to \infty} \int_{\gamma_i} g_j ds \\
			&\leq  \liminf_{i\to \infty} \int_{\gamma_i} g_i ds \quad \mbox{(since $g_i \ge g_j$ for all $i \ge j$)}\\
			&\leq  \liminf_{i\to \infty} f(\gamma_i(1))-f(\gamma_i(0))-\eta \quad \mbox{(by \eqref{eq:contrary})} \\
			&=  f(\gamma(1))-f(\gamma(0))-\eta.
		\end{align*}
		By monotone convergence, we may send $j\to \infty$ to get
		\[
		\int_\gamma g ds = \lim_{j\to\infty} \int_\gamma g_j ds \leq  |f(\gamma(1))-f(\gamma(0))|-\eta,
		\]
		which is a contradiction to $g$ being an upper gradient. Thus, it must have been that $f_i(x)\to f(x)$, and the claim follows.
	\end{proof}
	
	\section{Proofs of energy image density properties} \label{s:proofs}
	
	\subsection{Lower semicontinuity based approach} \label{ss:lsc-approach}
	\begin{proof}[Proof of Theorem \ref{thm:eid-dirspace}]
		By passing to a subset of $A$ using Lemma \ref{lem:essinf},  we can assume that there exists a $\delta>0$ such that  
		\begin{equation} \label{e:main1}
			\bigwedge_{ (\lambda_1,\ldots,\lambda_n) \in \mathbb{S}^{n-1}}^{\one_A \cdot \Lambda_{\bphi}}	\frac{d\Gamma_p\langle \sum_{i=1}^n \lambda_i \phi_i\rangle}{d\Lambda_{\bphi}}\geq \delta.
		\end{equation}
		
		Consider the Lebesgue decomposition of $\bphi_*(\one_A \cdot \Lambda_{\bphi})$ with respect to $\cL_n$ as 
		\[
		\bphi_*(\one_A \Lambda_{\bphi})= \mu_s + \mu_a,
		\]
		where $\mu_s\perp \mathcal{L}_n$ and $\mu_a \ll \mathcal{L}_n$. We aim to show that $\mu_s = 0$.  Assume to the contrary that $\mu_s\neq 0$. Then since $\mu_s$ is a Radon measure, there exists a compact set $B\subset \mathbb{R}^n$ such that $\mu_s(B)>0$, i.e. $\Lambda_{\bphi}( A \cap \bphi^{-1}(B))>0$, and $\mathcal{L}_n(B)=0$.

		We have $\mathcal{L}_n(B)=0$, and thus by Proposition \ref{prop:decompoabilitybundlesinuglar} the measure $\nu= \one_B\bphi_*(\one_A\cdot\Lambda_{\bphi})$  has a $(n-1)$-dimensional decomposability bundle $T_\nu:\mathbb{R}^n\to \bigcup_{d=0}^{n-1}\Gr(n,d)$. 
		Let $\epsilon \in (0,\pi/2)$ be such that
		\begin{equation}\label{e:cont0}
			\delta > \epsilon^p C(n,p)>0, \quad \mbox{where $C(n,p)= 1 \vee n^{(p-2)/2}$.}
		\end{equation}
		Since $\nu$ is non-zero, by Proposition  \ref{prop:decompositionconenull}, there exist  unit vector  $\blambda=(\lambda_{1},\dots, \lambda_{n})\in S^{n-1}$ and a compact cone null set $F$ such that $F$ is $C(\blambda, \pi/2-\epsilon)$-cone null set  and $\nu(F)>0$. 
		Since $\nu=\one_B\bphi_*(\one_A\cdot\Lambda_{\bphi})$, by replacing $F$ with $F \cap B$ if necessary, we may assume that $F \subset B$.
		Let $K =  A \cap \bphi^{-1}(F) \subset A \cap \bphi^{-1}(B)$.	Since $\nu(F)>0$ and $\nu= \one_B\bphi_*(\one_A\cdot\Lambda_{\bphi})$, we have
		\begin{equation}  \label{e:main3a}
			\Lambda_{\bphi} \left( A \cap \bphi^{-1}(F) \right)	=\Lambda_{\bphi} \left( K \right)>0.
		\end{equation} 
		Since $F$ is $C(\blambda,\pi/2-\epsilon)$-cone null, then by applying Lemma \ref{l:upper-gradient} to the function $f(y):=\langle \blambda,y \rangle$, we obtain that 
		\[
		g(y):= \sin(\epsilon) \one_{F}(y) + \one_{\mathbb{R}^n \setminus F}(y)
		\]
		is an upper gradient for $f$ and there exists a non-decreasing  sequence of continuous functions  $(g_j)_{j \in \mathbb{N}}$ such that $\lim_{j \to \infty} g_j(y)= g(y)$ for all $y \in \mathbb{R}^n$ and $g_j(y)= g(y)$ for all $j \in \mathbb{N}$ and for all $y \in \mathbb{R}^n \setminus B(0,R)$, where $R>0$ is such that $B(0,R) \supset F$.  Hence, by Proposition \ref{prop:approximations}, there exists a sequence, there exists a sequence of $1$-Lipschitz functions $(f_j)_{j \in \mathbb{N}}$ weak$*$-converging to $f$ for which
		\[
		\Lip_a[f_j](y) \le g(y) \le \epsilon \one_{F}(y) + \one_{\mathbb{R}^n \setminus F}(y), \quad \mbox{for all $y \in \mathbb{R}^n, j \in \mathbb{N}$.}
		\]
		Thus, 
		\begin{equation} \label{e:lipa-bnd}
			(\Lip_a[f_j] \circ \bphi)(x) \le \epsilon, \quad \mbox{for all $x \in \bphi^{-1}(F)$ and $j \in \mathbb{N}$.}
		\end{equation}
		By replacing $f_j(y)$ by $f_j(y)-f_j(0)$ if necessary, we may assume that $f_{j}(0)=0$ for all $j \in \mathbb{N}$. 
		
		Since $\abs{f \circ \bphi-f_{j} \circ \bphi} \le  \LIP[f - f_j] \abs{\bphi} \le 2 \sum_{i=1}^n \abs{\phi_i} \in L^p$, by dominated convergence theorem, we have 
		\begin{equation}  \label{e:main4}
			\lim_{j \to \infty}	\norm{f \circ \bphi - f_{j} \circ \bphi}_{L^p} =0, \quad \mbox{for all $k \in \mathbb{N}$.}
		\end{equation}
		By Proposition \ref{prop:quasicont-conseq}-(3) and using $\LIP[f_j]\le \LIP[f]=1$, we have
		\begin{equation}  \label{e:main5}
			\sup_{j \in \mathbb{N}}	\mathcal{E}_p(f_{j} \circ \bphi) \le C(n,p) \sum_{i=1}^n \mathcal{E}_p(\phi_i) < \infty.
		\end{equation}

		By lower semicontinuity \eqref{e:def-lsc}, chain rule (Proposition \ref{prop:quasicont-conseq}-(3))   and recalling that $K \subset A \cap \bphi^{-1}(F) \subset  A \cap \bphi^{-1}(B)$, we obtain
		\begin{align} \label{e:cont1}
			\Gamma_p\langle f \circ \bphi \rangle(K) &\le \liminf_{j\to \infty} \Gamma_p\langle f_j \circ \bphi \rangle(K) \quad \mbox{(by \eqref{e:main4},\eqref{e:main5}, \eqref{e:def-lsc})} \nonumber \\
			&\le \liminf_{j\to \infty}  C(n,p) \int_{K} \Lip_a[f_j](\phi(x))^p \Lambda_{\bphi}(dx) \quad \mbox{(by \eqref{e:multi-chainrule})} \nonumber \\
			&\le \epsilon^p  C(n,p) \Lambda_{\bphi}(K), \quad \mbox{(by \eqref{e:lipa-bnd})}
		\end{align}
		On the other hand, by \eqref{e:main1} and $K \subset A$,  we have 
		\begin{equation} \label{e:cont2}
			\Gamma_p\langle f \circ \bphi \rangle(K) \ge \delta \Lambda_{\bphi}(K).
		\end{equation}
		Combining \eqref{e:main3a}, \eqref{e:cont1}, \eqref{e:cont2} and \eqref{e:cont0}, we obtain the desired contradiction.
	\end{proof}
	The remaining cases of the energy image density property follow by the same argument, with only minor adjustments related to accounting for slightly different assumptions.
	\begin{proof}[Proof of Theorem \ref{thm:eid-pdirstructure}]
		The proof is  identical to the proof of Theorem \ref{thm:eid-dirspace}, except that we use Proposition \ref{prop:multidimcontraction-structure} instead of Proposition \ref{prop:quasicont-conseq}-(3).
	\end{proof}
	\begin{proof}[Proof of Theorem \ref{thm:eid-mmdspace}]
		This follows from  Theorem \ref{thm:eid-dirspace}, Lemma  \ref{lem:df-dirspace}, and Lemma \ref{l:invertibly-independence}.
	\end{proof}	
	\begin{proof}[Proof of Theorem \ref{thm:eid-dirstructure}]
		This follows from  Theorem \ref{thm:eid-pdirstructure}, Lemma  \ref{lem:dirichlet}, and Lemma \ref{l:invertibly-independence-dirstructure}.
	\end{proof}

	\subsection{Alternate proof of Bouleau--Hirsch conjecture using normal currents} \label{ss:alternate}
	We recall some basic notions concerning currents in $\mathbb{R}^n$ and refer to \cite{KrPa,Federer} for a detailed exposition.
	
	A \emph{$k$-dimensional current} $T$ in $\mathbb{R}^n$  is a continuous linear functional on the space of smooth, compactly supported $k$-forms $\mathcal{D}^k(\mathbb{R}^n)$. The \emph{boundary} $\partial T$ of $T$ is the $(k-1)$-current is defined by 
	\[
	(\partial T)(\omega):= T(d \omega), \quad \mbox{for all $\omega \in \mathcal{D}^{k-1}(\mathbb{R}^n)$, }
	\]
	where $d\omega$ is the exterior derivative of $\omega$.
	The \emph{mass} of $T$ is defined as the supremum of $T(\omega)$, where $\omega \in \mathcal{D}^k(\mathbb{R}^n)$ varies over all $k$-forms such that $\abs{\omega} \le 1$ everywhere. A current $T$ is said to be \emph{normal} if both $T$ and its boundary $\partial T$ have finite mass.
	
	We will only need the case $k=0$ and $k=1$ for our purposes, and   our currents will often have finite mass. By Riesz theorem, a $0$-dimensional current on $\mathbb{R}^n$ with finite mass can be viewed as a finite  (signed) measure on $\mathbb{R}^n$, while a $1$-dimensional current can be viewed as a $\mathbb{R}^n$-valued finite (signed) measure. Every $1$-dimensional current with finite mass $T$ can be canonically written as 
	\[
	T:= \vec{T} \cdot \norm{T}, 
	\]
	where $\norm{T}$ is a finite (non-negative) measure on $\mathbb{R}^n$ and $\vec{T}$ is a measurable vector field such that $|\vec{T}|=1$, $\norm{T}$-almost everywhere.
	
	Throughout \textsection \ref{ss:alternate}, we fix a Dirichlet structure $(X,\mathcal{X},\mu,\mathcal{E},\mathcal{F})$ in the sense of Definition \ref{d:dirstructure}. 
	\begin{definition}
		Let  $(X,\mathcal{X},\mu,\mathcal{E},\mathcal{F})$ be a Dirichlet structure. The \emph{generator} $A:D(A) \to L^2(X,\mu)$ is the linear operator with domain
		\[
		D(A):= \{ f\in \mathcal{F} \vert \mbox{there exists $C>0$ such that $\abs{\mathcal{E}(f,g)} \le C \norm{g}_{L^2(\mu)}$ for all $g \in \mathcal{F}$}\}.
		\]
		For any $f \in D(A)$, the element $A(f) \in L^2(X,\mu)$ is defined by
		\[
		\mathcal{E}(f,g)= -\langle A(f), g \rangle_{L^2(\mu)}.
		\] 
		
		We define the \emph{energy measure} as $\Gamma(f,g)= \gamma(f,g)\cdot \mu$ for all $f,g \in \mathcal{F}$, where $\gamma: \mathcal{F} \times \mathcal{F} \to L^1(X,\mu)$ is the associated carr\'e du champ.
	\end{definition}

	We define a family of $1$-dimensional currents associated with respect to a Dirichlet structure.
	\begin{definition}
		For a $\mathbb{R}^n$-valued function $f=(f_1,\ldots,f_n) \in \mathcal{F}^n$ and $g \in \mathcal{F}$,  we define $T_{f,g}$ to be the one dimensional current on $\mathbb{R}^n$ associated with the $\mathbb{R}^n$-valued measure $(f_*(\Gamma(f_i,g)))_{1 \le i \le n}$ on $\mathbb{R}^n$; that is,
		\[
		T_{f,g} \left( \sum_{j=1}^n h_j dx_j \right):= \sum_{j=1}^n \int h_j(f(x))\,\Gamma(f_i,g)(dx)  
		\]
		for any smooth, compactly supported $1$-form  $\sum_{j=1}^n h_j dx_j$.
	\end{definition}
	
	Recall that the \emph{boundary} of a $1$-dimensional current $T$ on $\mathbb{R}^n$ is the $0$-dimensional current (or distribution) defined by 
	\[
	\partial T(\phi)= T \left( \sum_{j=1}^n  \frac{\partial \phi}{\partial x_j} dx_j\right), \quad \mbox{for all $\phi \in C_c^\infty(\mathbb{R}^n)$.}
	\]
	Note that if $T$ is a $1$-dimensional current with finite mass, then the boundary $\partial T$ is the negative of the distributional divergence of the $\mathbb{R}^n$-valued measure corresponding to $T$.
	
	The following lemma is a simple consequence of the chain rule for Dirichlet forms.
	\begin{lemma} \label{l:normal-current}
		For any $n \in \mathbb{N}, f \in \mathcal{F}^n$ and $g \in \mathcal{F}$, we have that $T_{f,g}$ is a  $1$-dimensional current with finite mass. If in addition, $g \in D(A)$, then  $T_{f,g}$ is a  $1$-dimensional normal current whose  boundary is
		\[
		\partial T_{f,g} = -f_*(A(g)\cdot \mu).
		\]
	\end{lemma}
	\begin{proof}
		Note that for all $i=1,\ldots,n$, $\Gamma(f_i,g)$ is a measure whose variation   is bounded by $2^{-1}(\Gamma(f,f)+\Gamma(g_i,g))$ by the Cauchy-Schwarz inequality \cite[(2.2)]{hinoenergymeasures}. Hence $T_{f,g}$ is an $\mathbb{R}^n$-valued Borel measure of finite variation and hence a $1$-dimensional current with finite mass.
		
		Given a function $\phi:\mathbb{R}^n\to \mathbb{R}$, we define $\phi_0:\mathbb{R}^n\to \mathbb{R}$ as $\phi_0(y)=\phi(y)-\phi(0)$.
		We now compute the boundary of the current $T_{f,g}$.
		For all $\phi \in C_c^\infty(\mathbb{R}^n)$, $1 \le i \le n$, we have 
		\begin{align*}
			\partial T_{f,g}(\phi)&= T_{f,g}\left(\sum_{j=1}^n \frac{\partial \phi}{\partial x_j} dx_j\right) \\
			&= \sum_{i=1}^n \int \frac{\partial \phi}{\partial x_i}(f(x)) \,\Gamma(f_i,g)(dx) \\
			&= \mathcal{E}(\phi_0(f),g) \quad  \mbox{(by \cite[Corollary I.6.1.3]{BH91} or \cite[Theorem 3.2.2]{FOT})} \\
			&= \mathcal{E}(\phi(f),g) \quad \mbox{(since $\one \in \mathcal{F}$ and $\mathcal{E}(\one,g)=0$.)}
		\end{align*}
		In particular, if $f \in D(A)^n$, then 
		\[
		\partial T_{f,g}(\phi)= -\int_{X} A(g) \phi(f)\,d\mu = \int_{\mathbb{R}^n} \phi(y) \,f_*(-A(g)\cdot \mu) (dy),  
		\]
		for all $\phi \in C_c^\infty(\mathbb{R}^n)$.

		Since $A(g) \in L^2(X,\mu) \subset L^1(X,\mu)$, we have that $f_*(-A(g)\cdot \mu)$ is a finite Radon measure. Hence $T_{f,g}$ is a normal current.
	\end{proof}
	\begin{remark}\label{r:constant}
		The assumption $\one \in \mathcal{F}$ in Definition \ref{d:dirstructure}-(iv) is not crucial for the conclusion that $T_{f,g}$ is a normal current. 
		Without this assumption, we   have 
		\[
		\partial T_{f,g}(\phi)= \mathcal{E}(\phi_0(f),g) = \int_{\mathbb{R}^n} \phi(y) \,f_*(-A(g)\cdot \mu) (dy) + \phi(0) (f_*(A(g)\cdot \mu)) (\mathbb{R}^n) 
		\]
		and hence $\partial T_{f,g}=  -f_*(A(g)\cdot \mu) +  (f_*(A(g)\cdot \mu)) (\mathbb{R}^n)  \delta_0$, where $\delta_0$ is the Dirac measure at $0 \in \mathbb{R}^n$. Thus $T_{f,g}$ is a  normal current.
	\end{remark}
	In our  normal current  approach, our main tool  is the following structure theorem for finite families of normal currents by De Philippis and Rindler recalled in Theorem \ref{t:dr-current}. We recall that this result is also used in the other approach in the Proof of Proposition \ref{prop:decompoabilitybundlesinuglar}.

	We verify the hypotheses (i), (ii) of  Theorem \ref{t:dr-current} in the Lemma below.
	\begin{lemma} \label{l:dr-assume}
		For each $n \in \mathbb{N}$,  and $f=(f_1,\ldots,f_n) \in \mathcal{F}^n$, we have 
		\begin{equation} \label{e:dr1}
			f_* \left(\one_{\{\det(\gamma(f))>0\}}\cdot \mu\right) \ll \norm{T_{f,f_i}},
		\end{equation}
		and 
		\begin{equation} \label{e:dr2}
			\operatorname{span}\{\vec{T}_{f,f_1}(y) ,\ldots, \vec{T}_{f,f_n}(y) \} = \mathbb{R}^n, \quad \mbox{for $f_*\left(\one_{\{\det(\gamma(f))>0\}}\cdot \mu\right)$-almost every $y \in \mathbb{R}^n$.}
		\end{equation}
	\end{lemma}
	\begin{proof}

		We note that 
		$\Gamma(f_i,f_j) \ll \Gamma(f_i,f_i)$ for all $i,j=1,\ldots,n$ (by Cauchy-Schwarz inequality \cite[(2.2)]{hinoenergymeasures}). 
		Hence for each $i=1,\ldots,n$, $T_{f,f_i}$ can be represented as 
		\[
		T_{f,f_i}= \vec{T}_{f,f_i} \norm{T_{f,f_i}},
		\]
		where $\vec{T}_{f,f_i}: \mathbb{R}^n \to \mathbb{R}^n$ is given by
		\[
		\vec{T}_{f,f_i} = \left(\sum_{j=1}^n\left(\frac{df_*\left(\Gamma(f_i,f_j)\right)}{df_*\left(\Gamma(f_i,f_i)\right)}\right)^2 \right)^{-1/2} \left(\frac{df_*\left(\Gamma(f_i,f_1)\right)}{df_*\left(\Gamma(f_i,f_i)\right)},\ldots,\frac{df_*\left(\Gamma(f_i,f_n)\right)}{df_*\left(\Gamma(f_i,f_i)\right)}\right) 
		\]
		and $\norm{T_{f,f_i}}$ is a non-negative measure on $\mathbb{R}^n$ defined by
		\[
		\norm{T_{f,f_i}}:= \left( \sum_{j=1}^n\left(\frac{df_*\left(\Gamma(f_i,f_j)\right)}{df_*(\Gamma(f_i,f_i))}\right)^2 \right)^{ 1/2} \cdot f_*\left(\Gamma(f_i,f_i)\right),
		\]
		such that $\norm{\vec{T}_{f,f_i}(y)}=1$ for $\norm{T_{f,f_i}}$-almost every $y \in \mathbb{R}^n$.
		Since $f_*\left(\Gamma(f_i,f_i)\right) \ll \norm{T_{f,f_i}}$, in order to show \eqref{e:dr1} it suffices to prove
		\begin{equation} \label{e:aconX}
			\one_{\{\det(\gamma(f))>0\}}\cdot \mu \ll \Gamma(f_i,f_i), \quad \mbox{for all $i=1,\ldots,n$.}
		\end{equation}
		Let $1 \le i \le n$ and let $A \in \mathcal{X}$ be such that $\Gamma(f_i,f_i)(A)=0$. Then 
		$\gamma(f_i,f_i)(x)=0$, for $\mu$-almost every $x \in A$ and hence  	$\gamma(f_i,f_j)(x)=0$ for all $j=1,\ldots,n$ and for $\mu$-almost every $x \in A$ (by Cauchy-Schwarz inequality \cite[(2.2)]{hinoenergymeasures}). Therefore $\det(\gamma(f))(x)=0$ for $\mu$-almost every $x \in A$, or equivalently $\left(\one_{\{\det(\gamma(f))>0\}}\cdot \mu \right)(A)=0$. This concludes the proof of \eqref{e:aconX} and therefore \eqref{e:dr1}.
		
		Note that if $G \in L^1(X,\mu)$, then $f_*(G \cdot \mu)  \ll f_*(\mu)$ and if 
		\[
		\widetilde{G}:= \frac{df_*(G \cdot \mu)}{d f_*(\mu)},
		\quad \mbox{then} \quad 
		\mathbb{E}_{\mu}[G \vert \sigma(f)]= \widetilde{G} \circ f, \quad \mbox{$\mu$-almost surely},
		\]
		where $\sigma(f)= \{f^{-1}(B): B \in \mathcal{B}(\mathbb{R}^n)\}$ is the $\sigma$-field generated by $f$. Note that $\widetilde{G}$ is defined up to $f_*(\mu)$-almost everywhere equivalence, and $\mathbb{E}_{\mu}[G \vert \sigma(f)]$ is defined up to $\mu$-almost everywhere equivalence. Therefore the matrix valued random variable $	\widetilde{M}^f: X \to \mathbb{R}^{n \times n}$ defined as the conditional expectation given by
		\[
		\widetilde{M}^f:= \mathbb{E}_{\mu}[\gamma(f) \vert  \sigma(f)]
		\]
		satisfies
		\[
		\frac{d f_*(\Gamma(f_i,f_j))}{df_*(\mu)} \circ f= \widetilde{M}^f_{ij}, \quad \mbox{for all $i,j \in \{1,\ldots,n\}$.}
		\]
		By \cite[Lemma V.1.1.6.2]{BH91}, we have 
		\[
		\mu \left(\{x \in X: \det(\widetilde{M}^f(x))=0,\det(\gamma(f)(x)) \neq 0\}\right)=0,
		\]
		and hence $\det(\widetilde{M}^f(x)) \neq 0$ for $\one_{\{\det(\gamma(f))>0\}} \cdot \mu$-almost every $x \in X$. 
		Since $\widetilde{M}^f$ is $\mu$-almost surely symmetric non-negative definite matrix valued funtion,
		we have that the matrix-valued function $N^f: \mathbb{R}^n \to \mathbb{R}^{n \times n}$ defined by 
		\begin{equation} \label{e:defNf}
			N^f_{ij}:= \frac{d f_*(\Gamma(f_i,f_j))}{d f_*(\mu)}, \quad \mbox{for all $i,j \in \{1,\ldots,n\}$,}
		\end{equation}
		satisfies
		\begin{equation} \label{e:detNf}
			\det(N^f(y)) > 0, \quad \mbox{for $f_*(\one_{\{\det (\gamma(f))>0\}} \cdot \mu)$-almost every $y \in \mathbb{R}^n$.}
		\end{equation}
		Since the columns of $N^f$ are parallel to $\{\vec{T}_{f,f_i}: 1 \le i \le n\}$, we conclude the proof of \eqref{e:dr2}.
	\end{proof}
	\begin{proof}[Alternate proof of Theorem \ref{thm:eid-dirstructure}]
		If we assume in addition that $f \in D(A)^n$, where $A$ is the generator, then the desired result follows immediately by an application of Theorem \ref{t:dr-current} since the hypotheses of \cite[Corollary 1.12]{DR}   follow from Lemmas \ref{l:normal-current} and \ref{l:dr-assume}. 
		
		For an arbitrary $f=(f_1,\ldots,f_n) \in \mathcal{F}^n$, the one-dimensional currents $T_{f,f_i}, i=1,\ldots,n$ need not be normal. So in order to apply \cite[Corollary 1.12]{DR}, we approximate these   currents   $T_{f,f_i}, i=1,\ldots,n$ by a suitable seqence of  one-dimensional \emph{normal} currents. By the density of the domain of the generator $D(A)$ in the Hilbert space $(\mathcal{F}, \mathcal{E}_1)$ (using \cite[(1.3.3) and Lemma 1.3.3-(iii)]{FOT}), for each $i=1,\ldots,n$, there exists a sequence $(f_{i,k})_{k \in \mathbb{N}}$ such that $\lim_{k \to \infty} f_{i,k}=f_i$ in the norm topology of $\mathcal{F}$. By the continuity of carr\'e du champ $\gamma: \mathcal{F} \times \mathcal{F} \to L^1(\mu)$ \cite[Proposition I.4.1.3]{BH91},  the sequence of measures $\Gamma(f_j,f_{i,k})$ converges in variation to $\Gamma(f_j,f_i)$ as $k\to \infty$ for all $i,j \in \{1,\ldots,n\}$. Hence each entry in the sequence of matrix valued functions $N^{f,k}: \mathbb{R}^n \to \mathbb{R}^{n \times n}$, defined by
		\[
		N^{f,k}_{ij}:= \frac{d f_*(\Gamma(f_i,f_{j,k}))}{d f_*(\mu)}, \quad \mbox{for all $i,j \in \{1,\ldots,n\}$, and $k \in \mathbb{N}$}
		\]
		converges  in $L^1(f_*(\mu))$ to the corresponding entry of the matrix valued function $N^f:\mathbb{R}^n \to \mathbb{R}^{n \times n}$ as defined in \eqref{e:defNf}.
		
		By passing to a subsequence if necessary, we may assume that $	N^{f,k}$ converges $f_*(\mu)$-a.e.~to $N^f$. By the continuity of the determinant and \eqref{e:detNf}, we have that 
		\begin{equation} \label{e:convNf}
			\lim_{k \to \infty} \det(N^{f,k}(y))= \det( N^f(y)) > 0 \quad \mbox{for $f_*(\one_{\{\det (\gamma(f))>0\}} \cdot \mu)$-almost every $y \in \mathbb{R}^n$.}
		\end{equation}
		Note that for all $1 \le j \le n, k \in \mathbb{N}$ 
		\[
		T_{f,f_{j,k}}= \vec{T}_{f,f_{j,k}} \norm{T_{f,f_{j,k}}},\quad \mbox{ where } 		\norm{T_{f,f_{j,k}}}:= \left(\sum_{i=1}^n (N_{ij}^{f,k})^2\right)^{1/2} \cdot f_*(\mu).
		\]
		Hence for all $k \in \mathbb{N}$, we have
		\begin{equation} \label{e:dr1-approx}
			\one_{\{\det(N^{f,k})>0\}}\cdot	f_* \left(\one_{\{\det(\gamma(f))>0\}}\cdot \mu\right) \ll \norm{T_{f,f_{j,k}}},
		\end{equation}
		and by the same argument in the proof of Lemma \ref{l:dr-assume}
		\begin{equation} \label{e:dr2-approx}
			\operatorname{span}\{\vec{T}_{f,f_{1,k}}(y) ,\ldots, \vec{T}_{f,f_{n,k}}(y) \} = \mathbb{R}^n, 
		\end{equation}
		for $	\one_{\{\det(N^{f,k})>0\}}\cdot f_*\left(\one_{\{\det(\gamma(f))>0\}}\cdot \mu\right)$-almost every $y \in \mathbb{R}^n$. Thus by Theorem \ref{t:dr-current}, we obtain that 
		\[
		\one_{\{\det(N^{f,k})>0\}}\cdot f_*\left(\one_{\{\det(\gamma(f))>0\}}\cdot \mu\right) \ll \mathcal{L}_n, \quad \mbox{for all $k \in \mathbb{N}$.}
		\]
		Letting $k \to \infty$ and using \eqref{e:convNf}, we obtain $f_*\left(\one_{\{\det(\gamma(f))>0\}}\cdot \mu\right) \ll \mathcal{L}_n$ for all $f \in \mathcal{F}^n$.
	\end{proof}
	\begin{remark} \label{r:mmd-eid-alternate}
		The argument using normal currents presented in the proof of Theorem~\ref{thm:eid-dirspace} also provides an alternative proof of Theorem~\ref{thm:eid-mmdspace}. We briefly indicate the necessary modifications.
		
		Lemma~\ref{l:dr-assume} extends to the setting of MMD spaces with minor changes, the main one being the replacement of $\mu$ by a minimal energy dominant measure. The proof of Lemma~\ref{l:normal-current} requires $\mu$ to be finite, since the inclusion $L^2(X,\mu) \subset L^1(X,\mu)$ is used there. As noted in Remark~\ref{r:constant}, the assumption $\one \in \mathcal{F}$ is not needed. Hence, the proof of Lemma~\ref{l:normal-current} applies to MMD spaces as long as we additionally assume that $\mu$ is finite.
		
		This restriction is not essential for Theorem~\ref{thm:eid-mmdspace}, since we can reduce to the finite-measure case by using the trace Dirichlet form (see \cite[(6.2.4), Theorem~6.2.1]{FOT} for the definition of trace Dirichlet form). Since $\mu$ is $\sigma$-finite (recall that $\mu$ is a Radon measure on a locally compact, separable space), there exists a bounded strictly positive function $h:X \to (0,\infty)$ such that $h\mu$ is a probability measure. Since $h$ is bounded and strictly positive, it is straightforward to verify that the trace Dirichlet form $(\widecheck{\mathcal{E}},\widecheck{\mathcal{F}})$ on $L^2(X,h\mu)$ is an extension of $(\mathcal{E},\mathcal{F})$, that is,
		\[
		\widecheck{\mathcal{F}} \supset \mathcal{F}, \qquad 
		\restr{\widecheck{\mathcal{E}}}{\mathcal{F} \times \mathcal{F}} = \mathcal{E}.
		\]
		Therefore, it suffices to establish the energy image density property for the trace Dirichlet form associated with the finite measure $h\mu$.
		
	\end{remark}
	\section{Applications} \label{s:applications}
	\subsection{Martingale dimension}	 \label{ss:martingale-dim}
	We recall the definition of martingale dimension associated to a Dirichlet form. This notion was called as \emph{index} by Hino \cite[Definition 2.9-(ii)]{hinoenergymeasures} and is equivalent to classical definition of martingale dimension due to \cite[Theorem 3.4]{hinoenergymeasures}.
	\begin{definition}\cite[Definition 2.9]{hinoenergymeasures} \label{d:mdim}
		Let $(X,d,\mu,\mathcal{E},\mathcal{F})$ be an MMD space. Let $\Gamma(\cdot,\cdot)$ denote the corresponding energy measure and let $\nu$ be a minimal energy dominant measure. The martingale dimension is the smallest number $K \in \mathbb{N} \cup \{0,\infty\}$ such that 
		for any $n \in \mathbb{N}$ and for any $f \in \mathcal{F}^n$, we have 
		\[
		\nu-\esssup \operatorname{rank}(\gamma_\nu(f)) \le K.
		\]
		It is easy to verify that this definition does not depend on the choice of $\nu$.
	\end{definition}
	
	Associated with a Dirichlet form is a \textbf{strongly continuous contraction semigroup} $(P_t)_{t > 0}$; that is, a family of symmetric bounded linear operators $P_t:L^2(X,\mu) \to L^2(X,\mu)$ such that
	\[
	P_{t+s}f=P_t(P_sf), \quad \norm{P_tf}_2 \le \norm{f}_2, \quad \lim_{t \downarrow 0} \norm{P_tf-f}_2 =0, 
	\]
	for all $t,s>0, f \in L^2(X,\mu)$. In this case, we can express $(\mathcal{E},\mathcal{F})$ in terms of the semigroup as
	\begin{equation} \label{e:semigroup}
		\mathcal{F}=\{f \in L^2(X,\mu): \lim_{t \downarrow 0}  \frac{1}{t}\langle f - P_t f, f \rangle < \infty \}, \quad \mathcal{E}(f,f)= \lim_{t \downarrow 0}  \frac{1}{t}\langle f - P_t f, f \rangle, 
	\end{equation}
	for all $f \in \mathcal{F}$, where $\langle \cdot, \cdot \rangle$ denotes the inner product in $L^2(X,\mu)$ \cite[Theorem 1.3.1 and Lemmas 1.3.3 and 1.3.4]{FOT}.

	Let $(X,d,\mu,\mathcal{E},\mathcal{F})$ be an MMD space, and let $\{P_t\}_{t>0}$
	denote its associated Markov semigroup. A family $\{p_t\}_{t>0}$ of non-negative
	Borel measurable functions on $X \times X$ is called the
	\emph{heat kernel} of $(X,d,\mu,\mathcal{E},\mathcal{F})$, if $p_t$ is the integral kernel
	of the operator $P_t$ for any $t>0$, that is, for any $t > 0$ and for any $f \in L^2(X,\mu)$,
	\[
	P_t f(x) = \int_X p_t (x, y) f (y)\, d\mu (y) \qquad \mbox{for $\mu$-almost all $x \in X$.}
	\]
	
	Let $\Psi:[0,\infty) \to [0,\infty)$ be a homeomorphism, such that for all $0 < r \le R$,
	\begin{equation}  \label{e:reg}
		C^{-1} \left( \frac R r \right)^{\beta_1} \le \frac{\Psi(R)}{\Psi(r)} \le C \left( \frac R r \right)^{\beta_2}, 
	\end{equation}
	for some constants $1 < \beta_1 < \beta_2$ and $C>1$. 
	Such a function $\Psi$ is said to be a \textbf{scale function}.
	For $\Psi$ satisfying \eqref{e:reg}, we define
	\begin{equation} \label{e:defPhi}
		\Phi(s)= \sup_{r>0} \left(\frac{s}{r}-\frac{1}{\Psi(r)}\right).
	\end{equation}
	
	We say that $(X,d,\mu)$ is \emph{doubling} (or equivalently, $\mu$ is a \emph{doubling measure} on $(X,d)$) if there exists $C \in (1,\infty)$ such that 
	\[
	\mu(B(x,2r)) \le C\mu(B(x,r)), \quad \mbox{for all $x \in X, r>0$.}
	\]
	
	We say that $(X,d,\mu,\mathcal{E},\mathcal{F})$ satisfies the sub-Gaussian \textbf{heat kernel estimates}
	\hypertarget{hke}{$\operatorname{HKE(\Psi)}$}, if there exist $C_{1},c_{1},c_{2},c_{3},\delta\in(0,\infty)$
	and a heat kernel $\{p_t\}_{t>0}$ such that for any $t>0$,
	\begin{align}\label{e:uhke}
		p_t(x,y) &\le \frac{C_{1}}{m\bigl(B(x,\Psi^{-1}(t))\bigr)} \exp \left( -c_{1} t \Phi\left( c_{2}\frac{d(x,y)} {t} \right) \right)
		\qquad \mbox{for $\mu$-a.e.~$x,y \in X$,}\\
		p_t(x,y) &\ge \frac{c_{3}}{m\bigl(B(x,\Psi^{-1}(t))\bigr)}
		\qquad \mbox{for $\mu$-a.e.~$x,y\in X$ with $d(x,y) \le \delta\Psi^{-1}(t)$,}
		\label{e:nlhke}
	\end{align}
	where $\Phi$ is as defined in \eqref{e:defPhi} and $\Psi^{-1}:[0,\infty)\to [0,\infty)$ denotes the inverse of the homeomorphism $\Psi$. 
	
	We obtain a finiteness result for martingale dimension, thereby verifying  \cite[Conjecture 3.12]{murugananalytic}.
	\begin{theorem} \label{t:finite} 
		Let $(X,d,\mu,\mathcal{E},\mathcal{F})$ be an MMD space that satisfies sub-Gaussian heat kernel bounds  \hyperlink{hke}{$\operatorname{HKE(\Psi)}$}, where $\Psi: [0,\infty)\to [0,\infty)$ is a scale function and $\mu$ is a doubling measure on $(X,d)$. Then the martingale dimension of $(X,d,\mu,\mathcal{E},\mathcal{F})$ is finite.
	\end{theorem}
	\begin{proof}
		By \cite[Lemma 1.5.4  and Theorem 1.4.2-(iii)]{FOT}, the set $$\mathcal{G} :=\{P_t( f): f \in   L^\infty(X,\mu) \cap \mathcal{F}\}$$ is dense in the Hilbert space $(\mathcal{F},\mathcal{E}_1)$, where $P_t$ is the semigroup associated with the MMD space. By the H\"older continuity of heat kernel (see \cite[Theorem 3.1 and Corollary 4.2]{BGK}), there exists $\alpha >0$ such that every function in $\mathcal{G}$ admits an $\alpha$-H\"older continuous version.
		
		We claim that the martingale dimension $d_M$ satisfies 
		\begin{equation} \label{e:mdim1}
			d_M \le \frac{d_H(X,d)}{\alpha},
		\end{equation}
		where $d_H(X,d)$ is the Hausdorff dimension of $(X,d)$.
		Consider any $n \in \mathbb{N}$ and for any continuous function $f \in \mathcal{G}^n$ such that $\nu-\esssup \operatorname{rank}(\gamma_\nu(f))=n$. Then the measure $f_*(\one_{\det(\gamma_nu(f))>0}\cdot \nu)$ is non-zero and by the energy image density property (Theorem \ref{thm:eid-mmdspace}) is absolutely continuous with respect to the Lebesgue measure $\mathcal{L}_n$. Hence the image  $f(X)$ has Hausdorff dimension   $n$. By the $\alpha$-H\"older continuity of $f$, we have 
		\[
		n=d_H(f(X)) \le \frac{d_H(X,d)}{\alpha}.
		\]
		By the density of $\mathcal{G}$ in the Hilbert space $(\mathcal{F},\mathcal{E}_1)$ and \cite[Lemma 2.5-(ii)]{hinoenergymeasures}, we obtain \eqref{e:mdim1}.
	\end{proof}
	
	A slight modification of the argument in Theorem \ref{t:finite} also leads to a new proof that the martingale dimension is bounded from above by the Hausdorff dimension if the MMD space satisfies Gaussian heat kernel bounds  \cite[Corollary 3.3]{murugananalytic}. This result was shown in \cite{murugananalytic} by proving an equality between martingale dimension and Cheeger's analytic dimension and using known estimate on the analytic dimension. The proof below uses the energy image density property for MMD spaces in Theorem \ref{thm:eid-mmdspace}.
	\begin{proposition} (Cf. \cite[Corollary 3.3]{murugananalytic})
		Let $(X,d,\mu,\mathcal{E},\mathcal{F})$ be an MMD space that satisfies Gaussian heat kernel bounds  \hyperlink{hke}{$\operatorname{HKE(\Psi)}$}, where $\Psi(r)=r^2$ and $\mu$ is a doubling measure on $(X,d)$. Then the martingale dimension of $(X,d,\mu,\mathcal{E},\mathcal{F})$ is less than or equal to the Hausdorff dimension of $(X,d)$.
	\end{proposition}
	\begin{proof}
		By \cite[Theorem 2.2-(i)]{KZ} and \cite[Proposition 4.8]{KM20}, the set $\Lip_c(X) \cap C_c(X)$ is dense in the Hilbert space $(\mathcal{F},\mathcal{E}_1)$ (see also \cite[Proposition 2.16-(iii)]{murugananalytic} and \cite[Remark 4.6]{KM20}). Thus the same argument in the proof of \eqref{e:mdim1} in Theorem \ref{t:finite} applies with $\alpha=1$.
	\end{proof}
	
	\subsection{Cheeger's conjecture} \label{ss:cheeger}
	We introduce relevant background and recall   Cheeger's conjecture. The \emph{pointwise Lipschitz constant} of a function $f: X \to \mathbb{R}$ is given by
	\[\Lip f(x):= \limsup_{r \to 0} \sup_{d(x,y)<r} \frac{\abs{f(x)-f(y)}}{r}= \limsup_{\substack{y \to x,\\
			y \neq x}} \frac{\abs{f(x)-f(y)}}{d(x,y)}.\]

	We recall the definition of Sobolev space due to Cheeger \cite[Definition 2.2]{Ch}. 
	\begin{definition}
		Let $(X,d,\mu)$ be a metric measure space and $p \in (1,\infty)$. 
		For $f \in L^p(X,\mu)$, we define
		\[
		\norm{f}_{1,p}:=   \norm{f}_{L^p(\mu)}+  \inf_{(g_i)_{i \in \mathbb{N}}} \liminf_{i \to \infty} \norm{g_i}_{L^p(\mu)},
		\]
		where the infimum is taken over all sequences, $(g_i)_{i \in \mathbb{N}}$ such that there exists a
		sequence, $(f_i)_{i \in \mathbb{N}}$ with $f_i \xrightarrow{L^p(\mu)} f$
		and such that $g_i$ is an upper gradient for $f_i$ for all $i \in \mathbb{N}$. The Sobolev space $H_{1,p}:= H_{1,p}(X,d,\mu)$ is defined as $\{f \in L^{p}(X,\mu) \vert \norm{f}_{1,p} <\infty\}$. 
	\end{definition}
	It is known that $H_{1,p}$ is a Banach space \cite[Theorem 2.7]{Ch}. 
	We recall the notion of minimal generalized upper gradient.
	\begin{definition}
		Let $(X,d,\mu)$ be a metric measure space and $p \in (1,\infty)$. 
		The function  $g \in L^p(X,\mu)$ is a \emph{generalized upper gradient} for
		$f \in L^p(X,mu)$, if there exist sequences, $f_i \xrightarrow{L^p(X,\mu)} f, g_i \xrightarrow{L^p(X,\mu)}g$, such that $g_i$ is an upper
		gradient for $f_i$, for all $i \in \mathbb{N}$. \\
		We say that $g$ is a \emph{minimal generalized upper gradient} for $f$ if $g$ is a generalized upper gradient such that 
		\[
		\norm{f}_{1,p}= \norm{f}_{L^p(\mu)}+\norm{g}_{L^p(\mu)}.
		\]
	\end{definition}
	The existence    of minimal generalized upper gradient and its uniqueness up to  $\mu$-null sets follow from \cite[Theorem 2.10]{Ch}. We denote the unique minimal  generalized upper gradient for $f \in L^p(X,\mu)$ by $g_f \in L^p(X,\mu)$.
	
	The following $(1,p)$-Poincar\'e inequality plays an important role in \cite{Ch}. Let $p \in [1,\infty)$. We say that a complete metric space $(X,d,\mu)$ satisfies $(1,p)$-Poincar\'e inequality if there exist $C \in (0,\infty),\Lambda \in [1,\infty)$ such that for any $f \in L^p(X,\mu), x \in X, r>0$, we have 
	\begin{equation} \label{e:def-PI}
		\int_{B(x,r)} \abs{f-f_{B(x,r)}} \, d\mu \le C r \left(\int_{B(x,\Lambda r)} g^p \,d\mu\right)^{1/p},
	\end{equation}
	where $g$ is any upper gradient for $f$ and $f_{B(x,r)}:= \mu(B(x,r))^{-1} \int_{B(x,r)} f\,d\mu$.
	
	We verify that the Sobolev space $H_{1,p}$ can be viewed as a $p$-Dirichlet space in the sense of Definition \ref{d:dir-space}.
	\begin{lemma} \label{lem:upper-gradient=p-Dirichlet-space}
		Let $(X,d,\mu)$ be a complete metric measure space with a doubling measure $\mu$ such that the Poincar\'e inequality \eqref{e:def-PI} holds for some $p \in (1,\infty)$. Let $\mathcal{F}_p:= H_{1,p}$ and define for all Borel sets $A \subset X$ and for all $f \in \mathcal{F}_p$ 
		\[
		\mathcal{E}_p(f):= \int_X g_f^p\,d\mu, \quad \Gamma_p \langle f \rangle (A):= \int_A g_f^p\,d\mu,  
		\]
		where $g_f$ is the minimal generalized upper gradient for $f$.  Then $(X,d,\mu,\mathcal{E}_p,\mathcal{F}_p,\Gamma_p)$ is a regular $p$-Dirichlet space.
	\end{lemma}
	\begin{proof}
		The fact that $(X,d,\mu)$ is a locally compact follows from the fact that metric balls are precompact due to the completeness of $(X,d)$ and doubling property of $\mu$ \cite[Exercise 10.17]{Hei}. The completeness follows from  \cite[Theorem 2.7]{Ch}. The homogeneity immediately follows from homogeneity of the minimal generalized upper gradient \cite[p.~446]{Ch}, while the sublinearity follows from \cite[(3.3)]{Ch}. 
		
		To obtain the chain rule, we note that for any $\phi \in \Lip(\mathbb{R})$ and for any $f \in L^p(X,\mu)$, if 
		$g$ is an upper gradient for $f$, then $\LIP(\phi) g$ is an upper gradient for $\phi \circ f$. This along with \cite[(3.4)]{Ch} implies that $g_{\phi\circ f} \le \LIP(\phi) g_f$ $\mu$-a.e., where  $g_{\phi\circ f}, g_f$ denote the minimal generalized upper gradients of $\phi \circ f$ and $f$ respectively. Hence we obtain chain rule.   Alternately, the chain rule, sublinearity and homogeneity can be obtained by corresponding properties for the minimal $p$-weak upper gradient given in \cite[(6.3.18) and (6.31.9)]{HKST} and $\mu$-a.e.~equivalence between  minimal $p$-weak upper gradient and minimal generalized upper gradient that follows from \cite[Lemma 6.2.2 and Proof of Theorem 10.1.1]{HKST}.
		
		The locality property follows from \cite[Lemma 2.4]{KZ} and the equivalence between  minimal $p$-weak upper gradient and minimal generalized upper gradient mentioned above.
		The weak lower semicontinuity follows from reflexivity  \cite[Theorem 4.48]{Ch} and Lemma \ref{l:suff-lsc}. 
		
		It remains to show regularity. The fact that $\Lip(X) \cap \mathcal{F}_p$ is dense in $\mathcal{F}_p$ follows from \cite[Theorem 4.1]{Shanmun}; see also \cite[Theorem 4.24, Remark 4.25]{Ch}. We can approximate functions in $\Lip(X) \cap \mathcal{F}_p$ by a sequence of functions in $\Lip(X) \cap C_o(X)$ by considering $f\phi_n$, where 
		\[
		\phi_n(x):= n^{-1}\min\left(1,(2n-d(x,x_0))_+\right).
		\]
		It is easy to verify that $f \phi_n$ converges in $\mathcal{F}_p$ to $f$ by using product rule for weak upper gradients \cite[Proposition 6.3.28]{HKST}. Thus $\Lip(X) \cap C_0(X)$ is dense in $\mathcal{F}_p$. 
	\end{proof}
	\begin{remark}
		Lemma \ref{lem:upper-gradient=p-Dirichlet-space} holds also without the Poincar\'e inequality. The completeness, chain rule, sub-linearity and homogeneity follow by the same arguments, which do not rely on the Poincar\'e inequality. The final claim about regularity follows from the density of $C(X)\cap \cF_p$ in $\cF_p$, which was established in \cite[Theorem 1.6]{continuousdense}. This can be upgraded to the density of functions with compact support by the same argument as in the previous proof.
	\end{remark}
	The following   of differentiability on a metric space arises from Cheeger's work.
	\begin{definition} \label{d:derivative} Suppose $f: X \to \mathbb{R}$ and $\phi=(\phi_1,\ldots,\phi_n): X \to\mathbb{R}^n$ are Lipschitz functions on a metric   space $(X,d)$. Then $f$ is \emph{differentiable with respect to $\phi$ at $x_0 \in X$} if there is a unique $\mathbf{a}=(a_1,\ldots,a_n) \in \mathbb{R}^n$ such that $f$ and the linear combination $a\cdot \phi = \sum_{i=1}^n a_i \phi_i$ agree to first order near $x_0$: 
		\[
		\limsup_{x \to x_0} \frac{\abs{f(x)-f(x_0)-\langle \mathbf{a}, \phi(x)-\phi(x_0) \rangle }}{d(x,x_0)}=0.
		\]
		Equivalently,  $\Lip g(x_0)=0$, where $g(\cdot)= f(\cdot)-\sum_{i=1}^n a_i \phi_i(\cdot)$. The tuple $a \in \mathbb{R}^n$ is the \emph{derivative of $f$ with respect to $\phi$} and will be denoted by $\partial_\phi f(x_0)$.
	\end{definition}
	Now that we have a notion of differentiability on a metric space, the analogue of \emph{almost everywhere differentiability} for a metric measure space is given by the notion of a \emph{measurable differentiable structure} defined below.  
	\begin{definition} \label{d:chart}
		A \emph{chart of dimension $n$} on a metric measure space $(X,d,\mu)$ is a pair $(U,\phi)$ where:
		\begin{enumerate}[(i)]
			\item $U \subset X$ is measurable, $m(U) >0$, and $\phi:X \to \mathbb{R}^n$ is Lipschitz.
			\item For every Lipschitz function $f: X \to \mathbb{R}$, there exists a measurable function $\partial_\phi f: U \to \mathbb{R}^n$ such that $f$ is differentiable with respect to $\phi$ at $m$-almost every $x_0 \in U$ and with derivative $\partial_\phi f(x_0)$.
		\end{enumerate}
		A \textbf{measurable differentiable structure}	on $(X,d,\mu)$ is a  countable collection $\{(U_\alpha,\phi_\alpha)\}$ of charts with uniformly bounded dimension such that $$\mu \left( X \setminus \bigcup_{\alpha} U_\alpha\right)=0.$$
	\end{definition}
	Given these notions Cheeger's theorem can be stated as follows.
	\begin{theorem}(\cite[Theorem 4.38]{Ch}) \label{t:cheeger}
		Let $(X,d,\mu)$ be a complete metric space that satisfies the volume doubling property  and   a $(1,p)$-Poincar\'e inequality \eqref{e:def-PI} for some $p \in [1,\infty)$. Then there is a measurable differentiable structure on  $(X,d,\mu)$ such that the dimension of the charts are uniformly bounded above by a constant that depends only on the constants involved in the assumptions.
	\end{theorem}
	Cheeger conjectured that if $(U,\phi)$ is an $n$-dimensional chart then $\mathcal{H}^n(\phi(U))>0$ \cite[Conjecture~4.63]{Ch}. 
	This follows from a theorem of De~Philippis--Marchese--Rindler \cite[Theorem~4.1.1]{DMR}, which yields
	\[
	\phi_*(\mathbf{1}_U \mu)\ll \mathcal{L}_n.
	\]
	Let us explain how we can obtain this result as a consequence of our energy image density property.
	\begin{proposition} (\cite[Theorem 4.1.1]{DMR}) \label{prop:cheeger}
		Let $(X,d,\mu)$ be a complete metric space that satisfies the volume doubling property  and   a $(1,p)$-Poincar\'e inequality \eqref{e:def-PI} for some $p \in (1,\infty)$. For any chart $(U,\phi)$ of dimension $n$, we have $\phi_*(\mathbf{1}_U \mu)\ll \mathcal{L}_n.$ In particular, $\mathcal{H}^n(\phi(U))>0$.
	\end{proposition}
	\begin{proof}
		By inner regularity of $m$, it suffices to assume that $U$ is compact. Let $x_0 \in X, r >0$ be such that $U \subset B(x_0,r)$ By using McSchane's extension theorem \cite[Theorem 6.2]{Hei}, we can replace $\phi=(\phi_1,\ldots,\phi_n)$ with  $\widetilde{\phi}=(\widetilde{\phi_1},\ldots,\widetilde{\phi_n})$ such that $\widetilde{\phi_i} \in \Lip(X) \cap C_0(X)$ for all $1 \le i \le n$ and $\phi_i \equiv \widetilde{\phi_i}$ on $B(x_0,r)$ and $\phi_i \equiv 0$ on $B(x_0,2r)^c$. 
		Therefore, without loss of generality, we assume that $\phi \in (\Lip(X)\cap C_0(X))^n \subset \mathcal{F}_p^n$, where $(X,d,\mu,\mathcal{E}_p,\mathcal{F}_p,\Gamma_p)$ is the $p$-Dirichlet space defined in Lemma \ref{lem:upper-gradient=p-Dirichlet-space}.
		
		By considering the constant function $f \equiv 0$ and using the uniqueness of the derivative $\partial_\phi f \equiv 0$, there exists a measurable subset $V \subset U$ such that $\mu(U \setminus V)=0$ and
		\begin{equation} \label{eq:independent}
			\Lip\left(\sum_{i=1}^n a_i \phi_i \right)(x) \neq 0, \quad \mbox{for all $x \in V, (a_1,\ldots,a_n) \in \mathbb{R}^n \setminus \{(0,\ldots,0)\}$.}
		\end{equation}
		Let $S$ be a countable dense subset of $\mathbb{R}^n$. 	By the continuity of $\mathbf{a} \mapsto \Lip(\langle \mathbf{a}, \phi\rangle)(x)$  on $\mathbb{R}^n$  for each $x \in X$ (see \cite[Sublemma 7.2.4]{Kei-diff}), we have 
		\begin{equation} \label{e:lip-pos}
			\inf_{\mathbf{a}=(a_1,\ldots,a_n) \in S} 	 	 \Lip\left(\sum_{i=1}^n a_i \phi_i \right)(x) >0, \quad \mbox{for all $x \in V$.}
		\end{equation}
		Let $\Lambda_\phi:= \sum_{i=1}^n \Gamma \langle \phi_i \rangle$.
		Note that for any $f \in \Lip(X) \cap C_0(X)$, we have that the minimal generalized upper gradient $g_f$ for $f$ satisfies $g_f= \Lip(f)$ $\mu$-a.e. (due to \cite[Theorem 6.1]{Ch}). Hence we have 
		\[
		\bigwedge^{\Lambda_{\phi}}_{\blambda \in \mathbb{S}^{n-1}} \frac{d\Gamma_p\langle \sum_{i=1}^n \lambda_i \phi_i\rangle}{d\Lambda_{\bphi}}= \frac{\inf_{\mathbf{a}=(a_1,\ldots,a_n) \in S} 	 	 \Lip\left(\sum_{i=1}^n a_i \phi_i \right)^p(\cdot)}{\sum_{i=1}^n (\Lip(\phi_i)(\cdot))^p}  \stackrel{\eqref{e:lip-pos}}{>}0 
		\]
		$\Lambda_{\phi}$-a.e.~in $U$. Hence $\phi$ is $p$-independent in $U$ in the $p$-Dirichlet space $(X,d,\mu,\mathcal{E}_p,\mathcal{F}_p,\Gamma_p)$. 
		By \eqref{e:lip-pos} and \cite[Theorem 6.1]{Ch},  $\one_U \Lambda_\phi$ and $\one_U \mu$ are mutually absolutely continuous.
		Hence the desired conclusion follows from Theorem \ref{thm:eid-dirspace}.
		
		Since $\one_U \mu$ is a non-zero measure the absolute continuity $\phi_*(\mathbf{1}_U \mu)\ll \mathcal{L}_n$ implies that $\mathcal{H}^n(\phi(U))>0$ since $\phi:X \to \mathbb{R}^n$ is Lipschitz.
	\end{proof}
	One can also obtain a slight strengthening with the same argument. We say that $\phi=(\phi_1,\dots, \phi_n)\in \Lip(X)^n$ is independent on $V$ if \eqref{eq:independent} holds.  Charts are independent, but the converse does not hold. By noting that the proof above only used the independence, we get the following result.
	
	\begin{proposition}\label{prop:independent}
		Let $(X,d,\mu)$ be a complete metric space that satisfies the volume doubling property  and   a $(1,p)$-Poincar\'e inequality \eqref{e:def-PI} for some $p \in (1,\infty)$. For any  $\phi=(\phi_1,\dots, \phi_n)$, which is independent on $U$, we have $\phi_*(\mathbf{1}_U \mu)\ll \mathcal{L}_n.$ In particular, $\mathcal{H}^n(\phi(U))>0$.
	\end{proposition}
	
	\begin{remark}
		In \cite{teriseb} a map $\phi=(\phi_1,\dots, \phi_n)\in \cF_p^n$ was defined as $p$-independent if
		\[
		\bigwedge_{|\mathbf{a}|=1} g_{\langle {\bf a}, \phi\rangle}>0,
		\]
		where $g_{\langle {\bf a}, \phi\rangle}$ is the minimal generalized upper gradient for $\langle {\bf a}, \phi\rangle$. By construction, 
		\[
		g_{\langle {\bf a}, \phi\rangle}=\frac{d\Gamma_p\langle {\bf a}, \phi\rangle}{d\mu},
		\]
		and one sees that $\phi$ is $p$-independent in the sense of \cite{teriseb} if and only if it $p$-independent in the sense of Definition \ref{d:p-independent}. This observation yields Proposition \ref{prop:cheeger} with a Cheeger chart replaced with a  $p$-independent $\phi$ in the sense of \cite{teriseb}, and without the assumption of a Poincar\'e inequality. Indeed, this would yield a different proof of \cite[Theorem 1.11(c)]{teriseb}. Our definition of $p$-independence was inspired by the one in \cite{teriseb}, which in turn was based on the presentation in \cite{KM}.
	\end{remark}

	\subsection{Metric currents and a theorem of Preiss} \label{ss:preiss}
	The following theorem is a  generalization of a result due to Preiss in two dimensions \cite[Theorem 3.3]{AKircheim} (Preiss had the sharp  constant $1$ instead of $\sqrt{2}$ below). An elementary proof in the case $n=1$ is given in \cite[Remarks after Theorem 3.3]{AKircheim}. This theorem answers a problem of Ambrosio and Kirchheim as they ask if  Preiss' theorem is valid in higher dimension \cite[p.~15]{AKircheim}. 
	\begin{theorem} \label{t:detjacobian}
		Let $\mu$ be a finite Borel measure on $\mathbb{R}^n$ and assume that $\mu$ is not absolutely continuous
		with respect to $\mathcal{L}_n$. Then there exists a sequence of continuously differentiable, $\sqrt{2}$-Lipschitz function
		functions $g_k: \mathbb{R}^n \to \mathbb{R}^n$ converging pointwise to the identity, and such that
		\begin{equation} \label{e:det-lim}
			\lim_{k \to \infty} \int_{\mathbb{R}^n} \det(\nabla g_k)\,d\mu < \mu(\mathbb{R}^n).
		\end{equation}
	\end{theorem}
	\begin{proof}
		By Propositions \ref{prop:decompoabilitybundlesinuglar} and \ref{prop:decompositionconenull}, there exist a unit vector $\blambda \in \mathbb{S}^{n-1}$, compact set $K \subset \mathbb{R}^n$ such that $\mu(K)>0$ and $K$ is $C(\blambda,\pi/3)$-cone null. Without loss of generality, we may assume that $\blambda =(1,0,\ldots,0)$. By Lemma \ref{l:upper-gradient}, the function $f:\mathbb{R}^n \to \mathbb{R}$ defined by $f(y):=\langle \blambda, y \rangle= y_1$ for all $y=(y_1,\ldots,y_n) \in \mathbb{R}^n$ has upper gradient
		$$g(y)= \one_{\mathbb{R}^n \setminus K} + \frac{1}{2}\one_{K}.$$
		Hence, by Proposition \ref{prop:approximations}, there is a sequence of $1$-Lipschitz functions $f_k: \mathbb{R}^n \to \mathbb{R}, k \in \mathbb{N}$ such that $\lim_{k \to \infty} f_k(y)=f(y)$ for all $y \in \mathbb{R}^n$ and such that 
		\[
		\Lip_a[f_k](y) \le g(y) = \one_{\mathbb{R}^n \setminus K} + \frac{1}{2}\one_{K}, \quad \mbox{for all $k \in \mathbb{N}, y \in \mathbb{R}^n$.}
		\] 
		For each $k \in \mathbb{N}$, compactness of $K$, there exists $\delta_k>0$ such that
		\begin{equation} \label{e:fk-lip}
			\LIP\left(\restr{f_k}{\overline{B}(y,\delta_k)}\right) \le \frac{2}{3}, \quad \mbox{for all $y \in K, k \in \mathbb{N}$,}
		\end{equation}
		and
		\[
		\lim_{k \to \infty} \delta_k=0.
		\]
		Let $\rho: \mathbb{R} \to [0,\infty)$ be a smooth function (mollifier)  such that $\int_{\mathbb{R}^n} \rho(y)\,dy=1$ and $\rho(y)=0$ for all $y \in \mathbb{R}^n$ with $\abs{y} \ge 1$. Define
		\[
		\rho_k(y):= \delta_k^{-n} \rho(y/\delta_k), \quad \widetilde{f_k}(y):= \int_{\mathbb{R}^n} f_k(y-z) \rho_k(z)\,\mathcal{L}_n(dz), \quad \mbox{for all $y \in \mathbb{R}^n, k \in \mathbb{N}$.}
		\]
		Then $\widetilde{f_k}:\mathbb{R}^n \to \mathbb{R}$ is smooth. Since each $f_k$ is $1$-Lipschitz, we have
		\[
		\abs{\widetilde{f_k}(y)-f(y)} \le \delta_k, \quad \mbox{for all $k \in \mathbb{N}, y \in \mathbb{R}^n$.}
		\] 
		Furthermore, each $\widetilde{f_k}$ satisfies
		\begin{equation} \label{e:tilfk-grad}
			\abs{\nabla \widetilde{f_k}(y)}= \abs{\int_{\mathbb{R}^n} \nabla f_k(y-z) \rho_k(z)\,\mathcal{L}_n(dz)} \le \esssup_{w \in \overline{B}(y,\delta_k)} \abs{\nabla f_k(w)} \le	\LIP\left(\restr{f_k}{\overline{B}(y,\delta_k)}\right),
		\end{equation}
		for all $y \in \mathbb{R}^n, k \in \mathbb{N}$,
		where $\nabla f_k$ denotes the $\mathcal{L}_n$-a.e.~well-defined gradient of $f_k$. Hence $\widetilde{f}_k$ is $1$-Lipschitz for all $k \in \mathbb{N}$. Combining \eqref{e:fk-lip} and \eqref{e:tilfk-grad}, we obtain 
		\begin{equation} \label{e:gradfk-til}
			\abs{\nabla \widetilde{f_k}(y)} \le \frac{2}{3} \one_K(y) + \one_{\mathbb{R}^n\setminus K}(y), \quad \mbox{	for all $y \in \mathbb{R}^n, k \in \mathbb{N}$.}
		\end{equation}
		
		Consider the sequence of functions $g_k: \mathbb{R}^n \to \mathbb{R}^n$ defined by 
		\begin{equation} \label{e:diff-fk}
			g_k(y):= (\widetilde{f_k}(y),y_2,\ldots,y_n), \quad \mbox{for all $y=(y_1,\ldots,y_n) \in \mathbb{R}^n, k \in \mathbb{N}$.}
		\end{equation}
		Since $\widetilde{f}_k$ is $1$-Lipschitz, for any $y,z \in \mathbb{R}^n, k \in \mathbb{N}$, we have 
		\[
		\abs{g_k(y)-g_k(z)} \le \sqrt{\abs{\widetilde{f_k}(y)-\widetilde{f_k}(z)}^2+\abs{y-z}^2} \le \sqrt{2} \abs{y-z}
		\]
		Thus each $g_k$ is $\sqrt{2}$-Lipschitz and continuously differentiable. Since $\lim_{k \to \infty} f_k=f$ pointwise, by \eqref{e:diff-fk}, we have that $g_k$ converges pointwise to the identity map. By Hadamard's inequality, we have 
		\begin{equation} \label{e:detbnd}
			\abs{\det (\nabla g_k)(y)} \le 	\abs{\nabla \widetilde{f_k}(y)} \stackrel{\eqref{e:gradfk-til}}{\le}  \frac{2}{3} \one_K(y) + \one_{\mathbb{R}^n\setminus K}(y), \quad \mbox{	for all $y \in \mathbb{R}^n, k \in \mathbb{N}$.}
		\end{equation}
		Since $\mu(K)>0$,  by \eqref{e:detbnd} and passing to a subsequence if necessary, we obtain \eqref{e:det-lim}.		
	\end{proof}
	\begin{remark}
		Theorem \ref{t:detjacobian} provides a characterization of absolute continuity with respect to Lebesge measure as the converse is known.
		By   \cite[Example 3.2]{AKircheim}, if $\mu \ll \mathcal{L}_n$ is a finite Borel measure, and if $g_k: \mathbb{R}^n \to \mathbb{R}^n,k \in \mathbb{N}$ is a sequence of $\mathcal{C}^1$ functions  such that $\sup_{k \in \mathbb{N}} \LIP(g_k)<\infty$ with $g_k$ converging pointwise to identity, then 
		\[
		\lim_{k \to \infty} \int_{\mathbb{R}^n} \det(\nabla g_k)\,d\mu = \mu(\mathbb{R}^n).
		\]
	\end{remark}
	
	Let $\mathbf{M}_k(\mathbb{R}^n)$ denote the set of $k$-dimensional metric currents in $\mathbb{R}^n$ as defined by Ambrosio and Kirchheim \cite[Definition 3.1]{AKircheim}. Any $T \in \mathbf{M}_k(\mathbb{R}^n)$ has an associated finite Borel measure $\norm{T}$ called the \emph{mass of $T$} \cite[Definition 2.6, Proposition 2.7]{AKircheim}. By replacing the use of Preiss' theorem mentioned above (\cite[Theroem 3.3]{AKircheim}) with Theorem \ref{t:detjacobian} in the proof of \cite[Theorem 3.8]{AKircheim}, we obtain  a different proof of the following result of \cite[Theorem 1.15]{DR}. 
	\begin{corollary} \label{cor:metric-current}
		For any $n \in \mathbb{N}$ and for any $T \in \mathbf{M}_n(\mathbb{R}^n)$, the mass $\norm{T}$ is absolutely continuous with respect to $\mathcal{L}_n$.
	\end{corollary}
	Corollary \ref{cor:metric-current} was previously shown by Ambrosio and Kirchheim for the case $n=1,2$ \cite[Theorem 3.8]{AKircheim} and by De Philippis and Rindler for all $n \in \mathbb{N}$ using Alberti representations.
	
	\bibliographystyle{acm}

\end{document}